\documentclass{amsart}

\usepackage{amssymb,amsmath,amsthm,amsfonts,mathtools,mathabx}

\usepackage{graphicx}
\usepackage[usenames,dvipsnames]{xcolor}

\usepackage{hyperref}
\usepackage{url}

\hypersetup{
    bookmarks=true,         	
    unicode=false,          		
    pdftoolbar=true,        		
    pdfmenubar=true,        	
    pdffitwindow=false,     		
    pdfstartview={FitH},    		
    pdfnewwindow=true,      	
    colorlinks=true,       			
    linkcolor=red,          		
    citecolor=PineGreen,    
    filecolor=magenta,      		
    urlcolor=cyan           		
}

\theoremstyle{plain}
\newtheorem{theorem}{Theorem}[section]
\newtheorem{corollary}[theorem]{Corollary}
\newtheorem{lemma}[theorem]{Lemma}
\newtheorem{proposition}[theorem]{Proposition}

\theoremstyle{definition}
\newtheorem{definition}[theorem]{Definition}
\newtheorem{problem}[theorem]{Problem}

\theoremstyle{remark}
\newtheorem{remark}[theorem]{Remark}

\numberwithin{equation}{section}

\DeclareMathOperator{\ad}{ad}

\DeclareMathOperator{\Imag}{Im}

\allowdisplaybreaks

\begin{document}

\title[Global Existence for the DNLS Equation]{Global Existence  for the Derivative Nonlinear Schr\"{o}dinger Equation  by the Method of Inverse Scattering}
\author{Jiaqi Liu}
\author{Peter A. Perry}
\author{Catherine Sulem}
\address[Liu]{Department of Mathematics, University of Kentucky, Lexington, Kentucky 40506--0027}
\address[Perry]{ Department of Mathematics, University of Kentucky, Lexington, Kentucky 40506--0027}
\address[Sulem]{Department of Mathematics, University of Toronto, Toronto, Ontario M5S 2E4, Canada }
\thanks{Peter Perry supported in part by NSF Grant DMS-1208778}
\thanks {C. Sulem supported in part by NSERC Grant 46179-13}
\date{\today}

\begin{abstract}
We develop inverse scattering for the derivative nonlinear Schr\"{o}\-din\-ger equation (DNLS) on the line 
using its  gauge equivalence with a related nonlinear dispersive equation. We prove Lipschitz continuity of the direct and inverse scattering maps from the weighted Sobolev spaces $H^{2,2}(\mathbb{R})$ to itself. These results immediately imply global existence of solutions to the DNLS for initial data in a spectrally determined (open) subset of $H^{2,2}(\mathbb{R})$ containing a neighborhood of $0$. 
Our work draws ideas  from the pioneering work of Lee and from more recent work of Deift and Zhou on the nonlinear Schr\"{o}dinger equation.
\end{abstract}
\maketitle
\tableofcontents

%
%
%
%
%


\newcommand{\ba}{\breve{a}}
\newcommand{\bb}{\breve{b}}

\newcommand{\bM}{\breve{M}}

\newcommand{\tdelta}{\widetilde{\delta}}


\newcommand{\bbC}{\mathbb{C}}
\newcommand{\bbR}{\mathbb{R}}


\newcommand{\calA}{\mathcal{A}}
\newcommand{\calB}{\mathcal{B}}
\newcommand{\calC}{\mathcal{C}}
\newcommand{\calD}{\mathcal{D}}
\newcommand{\calE}{\mathcal{E}}
\newcommand{\calF}{\mathcal{F}}
\newcommand{\calG}{\mathcal{G}}
\newcommand{\calH}{\mathcal{H}}
\newcommand{\calI}{\mathcal{I}}
\newcommand{\calJ}{\mathcal{J}}
\newcommand{\calK}{\mathcal{K}}
\newcommand{\calL}{\mathcal{L}}
\newcommand{\calM}{\mathcal{M}}
\newcommand{\calN}{\mathcal{N}}
\newcommand{\calO}{\mathcal{O}}
\newcommand{\calP}{\mathcal{P}}
\newcommand{\calQ}{\mathcal{Q}}
\newcommand{\calR}{\mathcal{R}}
\newcommand{\calS}{\mathcal{S}}
\newcommand{\calT}{\mathcal{T}}
\newcommand{\calU}{\mathcal{U}}
\newcommand{\calV}{\mathcal{V}}
\newcommand{\calW}{\mathcal{W}}
\newcommand{\calX}{\mathcal{X}}
\newcommand{\calY}{\mathcal{Y}}
\newcommand{\calZ}{\mathcal{Z}}


\newcommand{\bfe}{\mathbf{e}}
\newcommand{\bfm}{\mathbf{m}}
\newcommand{\bfn}{\mathbf{n}}
\newcommand{\bft}{\mathbf{t}}


\newcommand{\bfA}{\mathbf{A}}
\newcommand{\bfB}{\mathbf{B}}
\newcommand{\bfC}{\mathbf{C}}
\newcommand{\bfD}{\mathbf{D}}
\newcommand{\bfE}{\mathbf{E}}
\newcommand{\bfF}{\mathbf{F}}
\newcommand{\bfG}{\mathbf{G}}
\newcommand{\bfH}{\mathbf{H}}
\newcommand{\bfI}{\mathbf{I}}
\newcommand{\bfJ}{\mathbf{J}}
\newcommand{\bfK}{\mathbf{K}}
\newcommand{\bfL}{\mathbf{L}}
\newcommand{\bfM}{\mathbf{M}}
\newcommand{\bfN}{\mathbf{N}}
\newcommand{\bfO}{\mathbf{O}}
\newcommand{\bfP}{\mathbf{P}}
\newcommand{\bfQ}{\mathbf{Q}}
\newcommand{\bfR}{\mathbf{R}}
\newcommand{\bfS}{\mathbf{S}}
\newcommand{\bfT}{\mathbf{T}}
\newcommand{\bfU}{\mathbf{U}}
\newcommand{\bfV}{\mathbf{V}}
\newcommand{\bfW}{\mathbf{W}}
\newcommand{\bfX}{\mathbf{X}}
\newcommand{\bfY}{\mathbf{Y}}
\newcommand{\bfZ}{\mathbf{Z}}


\newcommand{\bbfN}{\breve{\mathbf{N}}}


\newcommand{\abar}{\overline{a}}
\newcommand{\bbar}{\overline{b}}
\newcommand{\fbar}{{\overline{f}}}
\newcommand{\gbar}{{\overline{g}}}
\newcommand{\kbar}{{\overline{k}}}
\newcommand{\rbar}{{\overline{r}}}
\newcommand{\ubar}{{\overline{u}}}
\newcommand{\qbar}{{\overline{q}}}
\newcommand{\xbar}{\overline{x}}
\newcommand{\wbar}{\overline{w}}
\newcommand{\zbar}{{\overline{z}}}

\newcommand{\Pbar}{\overline{P}}


\newcommand{\mubar}{{\overline{\mu}}}
\newcommand{\xibar}{\overline{\xi}}
\newcommand{\nubar}{\overline{nu}}
\newcommand{\psibar}{\overline{\psi}}
\newcommand{\rhobar}{\overline{\rho}}
\newcommand{\zetabar}{\overline{\zeta}}


\newcommand{\dee}{\partial}
\newcommand{\dbar}{\overline{\partial}}

\newcommand{\dint}{\displaystyle{\int}}
\newcommand{\dsum}{\displaystyle{\sum}}

\newcommand{\darr}{\downarrow}
\newcommand{\uarr}{\uparrow}
\newcommand{\rarr}{\rightarrow}

\newcommand{\eps}{\varepsilon}

\newcommand{\diag}{\mathrm{diag}}
\newcommand{\offd}{\mathrm{off}}


\newcommand{\dotarg}{\, \cdot \, }
\newcommand{\diamondarg}{\, \diamond \, }


\newcommand{\bigO}[1]{{\mathcal{O}}\left( {#1} \right)}
\newcommand{\littleo}[1]{{\mathcal{o}}\left( {#1} \right)}

\newcommand{\norm}[2]{\left\Vert {#1} \right\Vert_{#2}}
\newcommand{\absval}[1]{\left\vert {#1} \right\vert}


\newcommand{\twovec}[2]
{
\left( 
	\begin{array}{c} {#1} \\ {#2} \end{array}
\right)
}

\newcommand{\Twovec}[2]
{
\left( 
	\begin{array}{c} {#1}\\  \\ {#2} \end{array}
\right)
}


%
%

\newcommand{\twomat}[4]
{
	\left(
			\begin{array}{cc}
			#1 & #2 \\
			#3 & #4
			\end{array}
	\right)
}

\newcommand{\Twomat}[4]
{
	\left(
			\begin{array}{cc}
			#1 & #2 \\
			\\
			#3 & #4
			\end{array}
	\right)
}

\newcommand{\uppermat}[1]
{
	\left(
			\begin{array}{cc}
			0	&	{#1} \\
			0	&	0		
			\end{array}
	\right)
}

\newcommand{\lowermat}[1]
{
	\left(
			\begin{array}{cc}
			0		&		0 \\
			{#1}	&		0		
			\end{array}
	\right)
}

\newcommand{\unituppermat}[1]
{
	\left(
			\begin{array}{cc}
			1	&	{#1} \\
			0	&	1		
			\end{array}
	\right)
}

\newcommand{\unitlowermat}[1]
{
	\left(
			\begin{array}{cc}
			1		&		0 \\
			{#1}	&		1		
			\end{array}
	\right)
}

\newcommand{\diagmat}[2]
{
	\left(
			\begin{array}{cc}
				{#1}	&		0\\
				0		&		{#2}
			\end{array}
	\right)
}

\newcommand{\offdiagmat}[2]
{
	\left(
			\begin{array}{cc}
			0		&	{#1} \\
			{#2}	&	0		
			\end{array}
	\right)
}

\newcommand{\Diagmat}[2]
{
	\left(
			\begin{array}{cc}
				{#1}	&		0\\
				\\
				0		&		{#2}
			\end{array}
	\right)
}

\newcommand{\Offdiagmat}[2]
{
	\left(
			\begin{array}{cc}
			0		&	{#1} \\
			\\
			{#2}	&	0		
			\end{array}
	\right)
}

\newcommand{\twodet}[4]
{
\left|
	\begin{array}{cc}
		#1		&		#2		\\[5pt]
		#3		&		#4
	\end{array}
\right|
}

%
%
%
%
%

\newcommand{\weights}[6]
{
	\begin{equation}
	\label{#1}
	({#2}^-,{#2}^+) =
				\begin{cases}
						\left(	{#3}, {#4} \right) &	\lambda \in D_+\\
						\\
						\left(  {#5}, {#6} \right) & 	\lambda \in D_-
				\end{cases}
	\end{equation}
}

%
%
%
%
%

\newcommand{\longweights}[6]
{
	\begin{multline}
	\label{#1}
	({#2}^-,{#2}^+) =\\[0.2cm]
				\begin{cases}
						\left(	{#3}, {#4} \right) &	\lambda \in D_+\\
						\\
						\left(  {#5}, {#6} \right) & 	\lambda \in D_-
				\end{cases}
	\end{multline}
}


\newcommand{\sigmathree}{\twomat{1}{0}{0}{-1}}


\newcommand{\intpm}[1][x]{{\displaystyle{\int}}_{\hspace{-3pt} #1}^{\pm \infty}}

%
%
%

\newcommand{\One}{\mathbf{1}}

\newcommand{\leftint}[1][x]{\int_{-\infty}^{#1}}
\newcommand{\rightint}[1][x]{\int_{#1}^\infty}

%
%

\newcommand{\lambdabar}{\overline{\lambda}}
\newcommand{\lam}{\lambda}
\newcommand{\qsharp}{q^\sharp}

\newcommand{\tbfn}{\widetilde{\mathbf{n}}}

%
%

\newcommand{\intR}{\int_{\mathbb{R}}}
\newcommand{\fourier}{\calF(\overline{\rho})(\xi')\calF'(\rho)(\xi-\xi')d\xi'd\xi}
\newcommand{\nusharp}{\nu^\sharp}   
\newcommand{\mixed}{L^2_xL^2_\lambda}		

\newcommand{\lambdanorm}{L^2_{|\lambda|>M}}
\newcommand{\inftynorm}{C_x^0(\mathbb{R}^\pm,L^2_{|\lambda|>M})}
\newcommand{\twonorm}{L^2_{x}(\mathbb{R}^\pm)\otimes L^2_{|\lambda|>M}}

\newcommand{\intx}{\int_{x}^{\pm\infty}}
\newcommand{\inty}{\int_{y}^{\pm\infty}}
\newcommand{\intnot}{\int_{0}^{\pm\infty}}
\newcommand{\firstdiag}{\frac{\partial\nu^\pm_{diag}(x,\lambda)}{\partial\lambda}}
\newcommand{\seconddiag}{\frac{\partial^2\nu^\pm_{diag}(x,\lambda)}{\partial\lambda^2}}
\newcommand{\firstoff}{\frac{\partial\nu^\pm_{off}(x,\lambda)}{\partial\lambda}}
\newcommand{\secondoff}{\frac{\partial^2\nu^\pm_{off}(x,\lambda)}{\partial\lambda^2}}
\newcommand{\secondh}{\frac{\partial^2 h^{\pm}(x,\lambda)}{\partial\lambda^2}}
\newcommand{\dt}{\frac{\partial T^\pm}{\partial\lambda}}
\newcommand{\ddt}{\frac{\partial^2 T^\pm}{\partial\lambda^2}}
\newcommand{\bilinear}{B(\rho,\overline{\rho})}
\newcommand{\niu}{\nu_{11}^{\sharp}}

\newcommand{\sbar}{\overline{s}}

\newcommand{\balpha}{\breve{\alpha}}
\newcommand{\bbeta}{\breve{\beta}}

%
%

\newcommand{\bN}{\breve{N}}
\newcommand{\bp}{\breve{p}}
\newcommand{\bq}{\breve{q}}
\newcommand{\bqbar}{\overline{\bq}}

\newcommand{\brho}{\breve{\rho}}
\newcommand{\bnu}{\breve{\nu}}

\newcommand{\bS}{\breve{S}}

%
%

\newcommand{\laminv}{\langle \lam \rangle^{-1}}

\newcommand{\bJ}{\breve{J}}
\newcommand{\bmu}{\breve{\mu}}
\newcommand{\bm}{\breve{m}}
\newcommand{\bw}{\breve{w}}

\newcommand{\hatrho}{\widehat{\rho}}
\newcommand{\hatsigma}{\widehat{\sigma}}

\newcommand{\HS}{\mathrm{HS}}

\newcommand{\Norm}[2]
{
\bigl\|
	{#1}
\bigr\|
_{#2}
}

\section{Introduction}
\label{sec:intro}

The Derivative Nonlinear Schr\"odinger equation (DNLS) 
$$
i u_t + u_{xx} + i (|u|^2 u)_x =0
$$
is a  canonical  dispersive equation derived from the Magneto-Hydrodynamic equations in the presence 
of the Hall effect.
The equation models the  dynamics  of Alfv\'en waves propagating along an ambient magnetic field
in a long-wave, weakly nonlinear scaling regime
\cite{CLPS99,M76}. 
Under gauge transformations, DNLS takes different equivalent forms.
For example, defining 
$$
\psi(x,t) = u(x,t)  \exp\left(\frac{i}{2} \int_{-\infty}^x |u(y,t)|^2 dy\right),
$$
DNLS becomes 
$$
i\psi_t +\psi_{xx} + i|\psi|^2 \psi_x =0
$$
which appears in optical models of ultra-short pulses and is sometimes referred to as the Chen-Liu-Lee equation \cite{CLL79}. On the other hand, under the gauge
transformation, 
$$
v(x,t) 
	=  u(x,t)  \exp\left( \frac{3i}{4} \int_{-\infty}^x |u(y,t)|^2 dy\right),
$$
the new function $v$ satisfies
$$iv_t +v_{xx} -\frac{i}{2} |v|^2 v_x +\frac{i}{2} v^2 \bar v_x +
\frac{3}{16} |v|^4 v=0
$$
where the overbar means complex conjugate.
The latter transformation has been particularly useful in the mathematical
analysis of solutions and questions of global well-posedness.
In terms of scaling properties, DNLS is invariant under the transformation 
$$
u \rightarrow u_\lambda = \lambda^{-1/2} u\left(\frac{x}{\lambda}, \frac{t}{\lambda^2}\right).
$$
In particular, it is $L^2$-critical in the sense that 
$$\|u_\lambda\|_{L^2} =\|u\|_{L^2}.$$
DNLS has the same scaling properties as
the Nonlinear Schr\"{o}dinger (NLS)  equation with $L^2$-critical focusing nonlinearity
which has solutions that blowup in a finite time. Yet, these two equations have very different structural properties. 

It was proved by Hayashi and Ozawa \cite{HO92}  that 
solutions  exist locally in time in the Sobolev space  $H^1(\mathbb{R})$ and
they can be extended for all time if the $L^2$-norm of the initial condition is   small enough, 
namely if $\|u_0\|_{L^2} < \sqrt{2\pi}$. Recently, this upper bound has been improved to $\sqrt{4\pi}$ by
Wu \cite{Wu14}.  The novelty of Wu's argument  comes from using 
the conservation of the momentum
$$
 P(v)\equiv \int  \Big( \Im(\bar{v}v_x) +\frac{1}{4} 
  |v|^4\Big) dx
$$
in addition to the conservation of mass and energy  
$$
M(v) = \|v(t)\|_{L^2}^2, \;\;
E(v)  = \|v_x(t)\|_{L^2}^2 -\frac{1}{16} \|v(t)\|^6_{L^6}.
$$

The question of global well-posedness of solutions of DNLS with large data
remains an important  open problem. 

 A central  property of DNLS, discovered by  Kaup and Newell \cite{KN78}, is that it is
solvable through the inverse scattering method. In this pioneering work, the authors 
establish the main elements of the inverse scattering analysis.  In particular, they find the Lax pair, analyze the linear spectral flow and derive the soliton solutions.
A key observation is that the 
associated spectral problem  is of second order in terms of 
the spectral parameter.  This is in contrast with 
the  {ZS-AKNS} system  \cite{AKNS74,ZS72} (the linear spectral problem  associated to NLS) 
which is of first order in the spectral parameter.  Its study was the object of Lee's thesis 
\cite{Lee83} and  his subsequent work \cite{Lee89a}. The spectral analysis also provides an  infinite number of conserved quantities \cite{KN78}.

The present paper is devoted to a rigorous analysis
of the direct and inverse  scattering map,
with the goal of providing insight and building blocks 
for the questions of global well-posedness,
long-time behavior and asymptotic stability of solitons. 
In the rest of the introduction, we present the framework for the inverse scattering approach, building on seminal works
by Beals, Coifman, Deift and Zhou \cite{BC84,DZ93,DZ03,Zhou89,Zhou98}.  We first review the Lax representation for DNLS, the spectral problem that defines the direct scattering map, and Riemann-Hilbert problem (RHP) that defines the inverse scattering map (Section  \ref{intro:DNLS}). We then use symmetry reduction to give a more precise and analytically tractable definition of the direct and inverse scattering maps (Section \ref{intro:sr}; see Definitions \ref{def:R} and \ref{def:I}).  
The introduction ends with a summary of our results (Section \ref{sec:intro.summary}). 
These include a Lipschitz continuity property of the direct and   inverse scattering maps in 
weighted Sobolev spaces (Theorems \ref{thm:R}, \ref{thm:I}).  Due to the simple time-evolution of the 
scattering data (see \eqref{ab.ev}),  this analysis provides a construction
of a global  smooth solution of the DNLS equation under some restrictions on the initial condition 
(Theorem \ref{thm:DNLS2}).  Finally, we mention that the analysis of the direct and inverse scattering 
maps and well-posedness of DNLS is also the subject of  recent work by Pelinovsky and Shimabukoro 
\cite{PS15}. They work in slightly different weighted spaces and
use a mapping \cite{KN78}  that transforms the original spectral problem to a spectral problem of  Zakharov-Shabat type. 

The class of initial data considered here (see Theorems \ref{thm:R} and \ref{thm:I}) excludes initial data with soliton solutions, and contains ``small data.'' 

\subsection{DNLS as an Integrable System} 
\label{intro:DNLS}
It is common to rewrite DNLS
in the form 
\begin{equation}
\label{DNLS1}
i u_t + u_{xx} = i \varepsilon (|u|^2 u)_x 
\end{equation}
where  $\varepsilon = \pm 1$. A
gauge transformation, different from those 
discussed
above, plays an important role in the analysis. It has the form
\begin{equation}
\label{q.gauge}
 q(x,t) = u(x,t) \exp\left({-i\eps \int_{-\infty}^x |u(y,t)|^2 dy}\right)
\end{equation}
and maps solutions of \eqref{DNLS1} into solutions of 
\begin{equation}
\label{DNLS2}
i q_t +q_{xx} +i \eps q^2 \bar q_x +\frac{1}{2} |q|^4 q=0.
\end{equation}
The latter equation is sometimes referred to as the Gerjikov-Ivanov equation \cite{F00}.
We will actually solve \eqref{DNLS2} by inverse scattering and use the inverse of the gauge transformation \eqref{q.gauge} to obtain the solution of \eqref{DNLS1}. 
As the solution spaces of \eqref{DNLS1} with $\eps=1$ and $\eps=-1$ are connected by the simple mapping $u(x,t) \mapsto u(-x,t)$, we  will now fix $\eps=1$ and consider only this case for the rest of the paper.

The advantage of this formulation manifests itself when analyzing the inverse scattering map through  the  Riemann-Hilbert problem (RHP),  allowing us to write appropriately normalized, piecewise analytic solutions.
A discussion of the use of gauge transformations in soliton theory  can be found in \cite{WS83}.

The integrable equations \eqref{DNLS1} and \eqref{DNLS2} each admit a Lax representation 
$$ L_t - A_x +[L,A]=0 $$
for suitable operators $L$ and $A$. Equivalently, \eqref{DNLS1} and \eqref{DNLS2} are the compatibility conditions for the system of equations
\begin{equation}\label{pair}
 \psi_x = L\psi, \quad \psi_t = A\psi. 
\end{equation}
The operators $A$ and $L$ for \eqref{DNLS1} and \eqref{DNLS2} and their equivalence through the gauge transformation
are given in Appendix \ref{app:gauge}.
The flow defined by \eqref{DNLS2} with $\eps=1$  is linearized by the scattering transform associated with the linear problem
\begin{equation}
\label{LS}
\frac{d}{dx} \Psi = -i\zeta^2 \sigma \Psi + \zeta Q(x) \Psi + P(x) \Psi,
\end{equation}
where $\Psi$ is  a $2 \times 2$ matrix-valued function of $x$  and
$$ 
\sigma = 	\diagmat{1}{-1},
\quad
Q(x) = 		\offdiagmat{q(x)}{\overline{q(x)}}, 
\quad
P(x) =		\diagmat{p_1(x)}{p_2(x)}
$$
with
$$
p_1(x) = -(i/2) |q(x)|^2, \quad p_2(x) = (i/2)|q(x)|^2. 
$$
 As usual in the inverse scattering theory, time is seen as fixed,  hence we omit the  time dependence of the functions under consideration
for brevity, both for the direct and inverse problems.

To describe the scattering transform, we recall the Jost solutions and their associated scattering data. First, observe  that if $q=0$,  \eqref{LS} admits the solutions 
$\Psi_0(x,\zeta)=\exp(-ix\zeta^2 \sigma)$. These solutions are 
bounded for $\zeta \in \Sigma$ where
\begin{equation} \label{contour-Sigma}
 \Sigma = \left\{ \zeta \in \bbC: \Imag(\zeta^2)=0 \right\} 
 \end{equation}
 as shown in Figure \ref{fig:contours}.

%
%

\begin{figure}[h!]
\caption{The Contours $\Sigma$ and $\bbR$}
\vskip 0.5cm
\begin{center}
\begin{tabular}{ccc}


\setlength{\unitlength}{4.5cm}
\begin{picture}(1,1)

\put(0.5,0.5){\vector(1,0){0.25}}
\put(0.75,0.5){\line(1,0){0.25}}

\put(0.5,1){\vector(0,-1){0.25}}
\put(0.5,0.5){\line(0,1){0.25}}

\put(0.5,0.5){\vector(-1,0){0.25}}
\put(0.25,0.5){\line(-1,0){0.25}}

\put(0.5,0.25){\line(0,1){0.25}}
\put(0.5,0){\vector(0,1){0.25}}

\put(0.85,0.55){$\Sigma_1$}
\put(0.55,0.9){$\Sigma_2$}
\put(0.05,0.55){$\Sigma_3$}
\put(0.55,0.05){$\Sigma_4$}


\put(0.85,0.2){$\Omega^-$}
\put(0.85,0.8){$\Omega^+$}
\put(0.15,0.2){$\Omega^+$}
\put(0.15,0.8){$\Omega^-$}


\put(0.7,0.55){$+$}
\put(0.55,0.75){$+$}

\put(0.25,0.55){$-$}
\put(0.4,0.75){$-$}

\put(0.25,0.4){$+$}
\put(0.4,0.2){$+$}

\put(0.7,0.4){$-$}
\put(0.55,0.2){$-$}

\put(0.5,0.5){\circle*{0.025}}
\end{picture}

&\hspace{2cm}&


\setlength{\unitlength}{4.5cm}
\begin{picture}(1,1)


\put(0.5,0.5){\vector(1,0){0.25}}
\put(0.75,0.5){\line(1,0){0.25}}
\put(0.0,0.5){\vector(1,0){0.25}}
\put(0.25,0.5){\line(1,0){0.25}}
\put(0.5,0.5){\circle*{0.025}}


\put(0.47,0.6){$+$}
\put(0.47,0.35){$-$}

\put(0.3,0.8){$\bbC^+$}
\put(0.3,0.2){$\bbC^-$}
\end{picture}

 \\[0.2cm]
The Contour $\Sigma =\Sigma_1\cup \Sigma_2\cup \Sigma_3\cup\Sigma_4$ && The Contour $\bbR$ 
\end{tabular}
\end{center}
\label{fig:contours}
\end{figure}
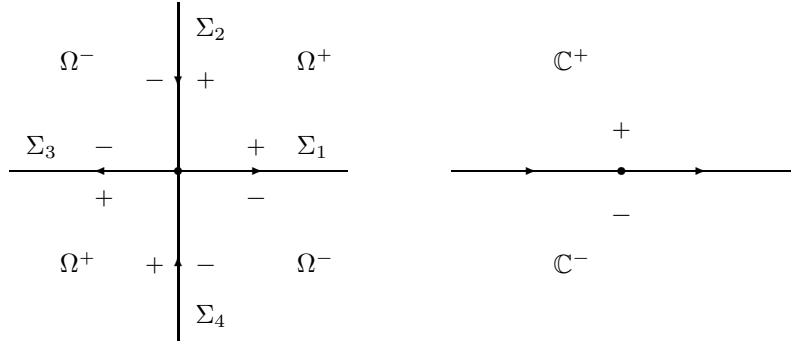

From this observation, we are led to consider 
bounded solutions $\Psi^\pm(x,\zeta)$ of \eqref{LS} with $\zeta \in \Sigma$, asymptotic as $x \rarr \pm \infty$ to $\Psi_0(x,\zeta)$. 
For such $\zeta$, we denote by $\Psi^\pm(x,\zeta)$ the unique solutions of \eqref{LS} with 
$$ \lim_{x \rarr \pm \infty} \Psi^\pm(x,\zeta) e^{ix\zeta^2 \sigma} = \One. $$
Here and in what follows, $\One$ denotes the $2\times 2$ identity matrix. The $\Psi^\pm$ are called \emph{Jost 
solutions}. Analytically, it is more convenient to work with the \emph{normalized Jost solutions} \begin{equation}
\label{LS.normal}
m^\pm(x,\zeta)=
\Psi^\pm(x,\zeta) e^{ix\zeta^2 \sigma}
\end{equation} 
These functions solve the equation
\begin{equation}
\label{LS.red}
\frac{d}{dx} m = -i\zeta^2 \ad \sigma (m) + \zeta Q(x) m + P(x) m, 
\quad \ad \sigma (A) \equiv \sigma A - A \sigma,
\end{equation}
with asymptotic condition
$$ \lim_{x \rarr \pm \infty} m^\pm(x,\zeta) = \One.$$
It follows from \eqref{LS} that $\det\Psi(x)$ is constant for any solution $\Psi$, and that  any two solutions $\Psi_1$ and $\Psi_2$ of \eqref{LS} with nonvanishing determinant are related by $\Psi_1 = \Psi_2 A$ for a constant nonsingular matrix $A$. Hence, $\det m(x) $ is constant for any solution of \eqref{LS.red}, and any two solutions $m_1$ and $m_2$ of \eqref{LS.red} with nonvanishing determinant are related by $m_1 = m_2 e^{-ix\zeta^2 \ad \sigma } A
 = m_2 e^{-ix\zeta^2\sigma} A e^{ix\zeta^2 \sigma}
$ for a nonsingular constant matrix $A$.

In particular, the Jost solutions $\Psi^\pm$ obey the relation
\begin{equation}
\label{T.rel}
\Psi^+(x,\zeta) = \Psi^-(x,\zeta) T(\zeta), \quad T(\zeta) = \twomat{a(\zeta)}{\bb(\zeta)}{b(\zeta)}{\ba(\zeta)}
\end{equation}
The matrix $T(\zeta)$
is called the \emph{transition matrix}, and the functions $a(\zeta)$, $b(\zeta)$, $\ba(\zeta)$, $\bb(\zeta)$ are called the \emph{scattering data}. We have
\begin{equation}
\label{scatt.det}
\det T(\zeta) = a(\zeta)\ba(\zeta)-b(\zeta)\bb(\zeta)=1 
\end{equation}
and $(a,\ba,b,\bb)$ obey the symmetry relations \eqref{ab.sym} described below.
Roughly and informally, the map $q \mapsto (a,\ba,b,\bb)$ is the direct scattering map.

The crucial result of inverse scattering theory for \eqref{DNLS2} is the following: suppose that $a(\lam,t)$, $\ba(\lam,t)$, $b(\lam,t)$, $\bb(\lam,t)$ are the scattering data of a solution $q(x,t)$ of \eqref{DNLS2}. It follows from the asymptotics of $\Psi^\pm$ and the equations \eqref{pair}
that
\begin{equation}
\label{ab.ev}
\dot{a}(\zeta,t) = \dot{\ba}(\zeta,t) = 0, \quad \dot{b}(\zeta,t) = 4i\zeta^4 b(\zeta,t), \quad \dot{\bb}(\zeta,t) = -4i\zeta^4 \bb(\zeta,t)
\end{equation}
(see \cite[eq. (34)]{KN78}).
Hence, to solve the Cauchy problem \eqref{DNLS2}, it suffices to compute the scattering data, solve the linear evolution equations \eqref{ab.ev}, and invert the time-evolved scattering data to recover $q(x,t)$. 

\begin{remark}
\label{rem:evr}
Let
$\Omega^\pm = \{ z \in \bbC: \pm \Imag (z^2) > 0\}$ (see Figure \ref{fig:contours} below).
The function $a$ extends analytically to $\Omega^-$, while $\ba$ extends analytically
to $\Omega^+$. It follows from Theorem A in \cite{Lee83} that any zeros of $a$ are \
contained in a bounded set.
The zero set respects the symmetries $\zeta \rarr \zetabar$ and $\zeta \rarr -\zeta$. In addition, zeros of $a(\zeta)$ and $\ba(\zeta)$ cannot occur on the real axis. This is a consequence of the symmetry
conditions \eqref{ab.sym} and the fact that the determinant of the transition matrix $T(z)$ is equal to $1$ (see  Figure  \ref{fig:spectra} below).

Zeros of $a$ in $\Omega^-$  correspond to $L^2$ eigenfunctions and give rise to bright (exponentially decaying) solitons. Zeros of $a$ on $\Sigma$ are called resonances and give rise to algebraic solitons.  Although the Kaup-Newell algebraic solitons \cite{KN78} decay too slowly to belong to $H^{2,2}(\bbR)$, we have not yet been able to  prove that zeros of 
$a$ on $\Sigma$ cannot occur for  $H^{2,2}(\bbR)$ initial data.  In the following, we exclude
initial conditions $q$ for which $a(\zeta)$ has zeros on $\Sigma$ or in $\Omega^-$. Due to \eqref{ab.ev}, this property persists for all time.
\end{remark}

\begin{figure}[h!]
\caption{Spectral Singularities in the $\zeta$- and $\lambda$-planes}
\vskip 0.4cm
\begin{center}
\begin{tabular}{ccc}


\setlength{\unitlength}{5.0cm}
\begin{picture}(1,1)

\put(0.5,0.5){\vector(1,0){0.25}}
\put(0.75,0.5){\line(1,0){0.25}}

\put(0.5,1){\vector(0,-1){0.25}}
\put(0.5,0.5){\line(0,1){0.25}}

\put(0.5,0.5){\vector(-1,0){0.25}}
\put(0.25,0.5){\line(-1,0){0.25}}

\put(0.5,0.25){\line(0,1){0.25}}
\put(0.5,0){\vector(0,1){0.25}}

\put(0.5,0.5){\circle{0.025}}


\put(0.9,0.55){$\bbR$}
\put(0.55,0.9){$i\bbR$}


\put(0.472,0.65){\color{red}$\times$}
\put(0.472,0.35){\color{red}$\times$}
\put(0.472,0.4){\color{red}$\times$}
\put(0.472,0.6){\color{red}$\times$}


\put(0.7,0.8){{\color{blue}\circle*{0.025}}}
\put(0.7,0.2){{\color{blue}\circle*{0.025}}}
\put(0.3,0.8){{\color{blue}\circle*{0.025}}}
\put(0.3,0.2){{\color{blue}\circle*{0.025}}}

\put(0.6,0.7){{\color{blue}\circle*{0.025}}}
\put(0.6,0.3){{\color{blue}\circle*{0.025}}}
\put(0.4,0.7){{\color{blue}\circle*{0.025}}}
\put(0.4,0.3){{\color{blue}\circle*{0.025}}}

\end{picture}


&\hspace{2cm}&


\setlength{\unitlength}{5.0cm}
\begin{picture}(1,1)


\put(0.5,0.5){\vector(1,0){0.25}}
\put(0.75,0.5){\line(1,0){0.25}}
\put(0.0,0.5){\vector(1,0){0.25}}
\put(0.25,0.5){\line(1,0){0.25}}
\put(0.5,0.5){\circle{0.025}}

\put(0.3,0.8){$\bbC^+$}
\put(0.3,0.2){$\bbC^-$}


\put(0.35,0.485){\color{red}$\times$}
\put(0.4,0.485){\color{red}$\times$}


\put(0.4,0.55){\color{blue}\circle*{0.025}}
\put(0.4,0.45){\color{blue}\circle*{0.025}}
\put(0.2,0.60){\color{blue}\circle*{0.025}}
\put(0.2,0.4){\color{blue}\circle*{0.025}}

\end{picture}

 \\[0.2cm]
Spectrum in $\zeta$ && Spectrum in $\lambda = \zeta^2$ 
\end{tabular}

\vskip 0.2cm

\begin{tabular}{ccc}
Origin ({\color{black} $\circ$}) &
Resonance ({\color{red} $\times$})	&	
Eigenvalue ({\color{blue} $\bullet$}) 
\end{tabular}
\end{center}
\label{fig:spectra}
\end{figure}
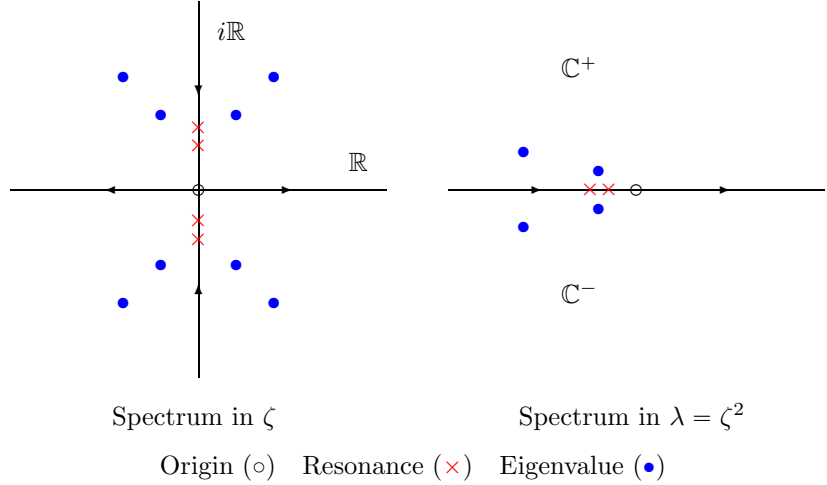

\bigskip

To formulate the inverse scattering map, we recall that equation \eqref{LS.red} admits
\emph{Beals-Coifman solutions} $M(x,z)$, piecewise analytic 
for $z \in \bbC \setminus \Sigma$,
with the following spatial normalizations:
\begin{itemize}
\item[(i)] The ``right-hand'' Beals-Coifman solutions are normalized so that 
		$$\lim_{x \rarr \infty} M(x,z) = \One$$ 
		for $z \in \bbC \setminus \Sigma$, and 
		$M(x,z)$ is bounded as $x \rarr -\infty$, for each such $z$, while
		
		\medskip
		
\item[(ii)] The ``left-hand'' Beals-Coifman solutions are normalized so that 
		$$\lim_{x \rarr -\infty}	M(x,z) = \One$$ 
		for $z \in \bbC \setminus \Sigma$, and $M(x,z)$ is bounded as $x \rarr +\infty$ 
		for each such $z$.
\end{itemize}
 We recall the construction of Beals-Coifman solutions in Section \ref{sec:BC} for sake of completeness.

Denoting by $M(x,z)$ either the left or right Beals-Coifman solution for $x \in \bbR$ and 
$z \in \bbC \setminus \Sigma$, there exist boundary values 
$M_\pm(x,\zeta)$ as $\pm \Imag(z^2) \darr 0$ which obey a jump relation of the form
\begin{equation}
\label{BC.jump}
M_+(x,\zeta) = M_-(x,\zeta) e^{-ix\zeta^2 \ad \sigma } v(\zeta), \quad \zeta \in \Sigma.
\end{equation}
The \emph{jump matrix} $v(\zeta)$ is determined by the scattering data $a$, $b$, $\ba$, and $\bb$. For the right-hand Beals-Coifman solution,
\begin{equation}
\label{BC.v.right}
 v_r(\zeta) =\Twomat{1-b(\zeta)\bb(\zeta)/a(\zeta)\ba(\zeta)}{\bb(\zeta)/a(\zeta)}{-b(\zeta)/\ba(\zeta)}{1}, 
\end{equation}
while for the left-hand Beals-Coifman solution,
\begin{equation}
\label{BC.v.left}
v_\ell(\zeta) 
	=\Twomat{1}
					{\bb(\zeta)/\ba(\zeta)}
					{-b(\zeta)/a(\zeta)}
					{1-b(\zeta)\bb(\zeta)/a(\zeta)\ba(\zeta)}. 
\end{equation}
This jump condition, together with the large-$z$ asymptotics
\begin{equation}
\label{BC.asy}
M(x,z) \sim \One + \frac{M_{-1}(x)}{z} + o\left(z^{-1}\right) 
\end{equation}
suffices to determine the Beals-Coifman solutions uniquely. Using this asymptotic expansion in 
\eqref{LS.red}, it is easy to see that the potential $Q(x)$ may be recovered from the formula
$$ Q(x) = i \ad \sigma \left[ M_{-1}(x) \right] $$
which implies that
$$ q(x) = 2i \lim_{z \rarr \infty} z M_{12}(x,z).$$
We may take the Riemann-Hilbert problem (RHP) \eqref{BC.jump}, \eqref{BC.asy} with given scattering data $(a,\ba,b,\bb)$ as a starting point to recover $q$ from the scattering data. In practice, the RHP with jump matrix \eqref{BC.v.right} gives a stable reconstruction of $q(x)$ for $x$ in half-lines $[c,\infty)$, while the RHP with jump matrix \eqref{BC.v.left} gives a stable reconstruction of the potential $q(x)$ for $x$ in half-lines  $(-\infty,c]$. Roughly and informally, the map $(a,\ba,b,\bb) \rarr q$ defined by these RHPs is the inverse scattering map.

%
%

\subsection{Symmetry Reduction}
\label{intro:sr}

To give a precise formulation of the direct and inverse maps, we reduce by symmetry from scattering data on the oriented contour $\Sigma$ to scattering data on the oriented contour 
$\bbR$. Both contours with orientation are shown in Figure \ref{fig:contours}. 
The contour $\Sigma$ can be viewed as the boundary of the regions
$$ \Omega^\pm = \left\{ \zeta \in \bbC: \pm \Imag \zeta^2 > 0 \right\}.$$
The map $\zeta \mapsto \zeta^2$ maps $\Sigma$ onto $\bbR$, $\Omega^+$ onto $\bbC^+$, $\Omega^-$ onto $\bbC^-$,  and induces the natural orientation on the contour $\bbR$.

Even functions on $\Sigma$ determine and are determined by functions on $\bbR$. This observation allows us to  reduce the scattering map defined by \eqref{LS} to a map involving functions on the real line.

We can reduce the spectral problem \eqref{LS.red} from $\Sigma$ to $\bbR$ by noting that the maps
\begin{equation}
\label{m.map1}
m(x,\zeta) \mapsto \Twomat{m_{11}(x,-\zeta)}{-m_{12}(x,-\zeta)}{- {m_{21}(x,-\zeta)}}{m_{22}(x,-\zeta)}
\end{equation}
and
\begin{equation}
\label{m.map2}
m(x,\zeta) \mapsto \twomat{0}{1}{1}{0} \overline{m(x,\zetabar)} \twomat{0}{1}{1}{0}
\end{equation}
preserves the solution space of \eqref{LS.red}. 
It follows from the unicity of normalized Jost solutions $m^\pm$ and stability of the solution space under \eqref{m.map1} that 
\begin{equation}
\label{mpm.sym1}
\begin{aligned}
m^+_{11}(x,-\zeta)&=m^+_{11}(x,\zeta), &
m^+_{12}(x,-\zeta)&=-m^+_{12}(x,-\zeta)\\
m^+_{21}(x,-\zeta)&=-m^+_{21}(x,\zeta), &
m^+_{22}(x,-\zeta)&=m^+_{22}(x,\zeta)
\end{aligned}
\end{equation}
and similarly for   $m^-(x,\zeta)$. 
In an analogous way, it follows from unicity of $m^\pm$ and stability of the solution space under \eqref{m.map2}, that
\begin{equation}
\label{mpm.sym2}
m^+_{22}(x,\zeta) = \overline{m^+_{11}(x,\zetabar)}, 
\quad 
m^+_{12}(x,\zeta)=\overline{m^+_{21}(x,\zetabar)}
\end{equation}
and similarly for $+$ replaced by $-$.

Equations \eqref{T.rel}, \eqref{mpm.sym1}, and \eqref{mpm.sym2}  imply the symmetry relations
\begin{equation}
\label{ab.sym}
\ba(\zeta) = \overline{a(\zetabar)}, \quad \bb(\zeta)=\overline{b(\zetabar)}, 
\quad a(-\zeta) = a(\zeta), \quad b(-\zeta)=-b(\zeta)
\end{equation}
for the scattering data.

Using these relations, we can now reduce the scattering problem \eqref{LS.red} with $\zeta \in \Sigma$ to scattering problem for $\lam = \zeta^2 \in \bbR$, and identify a single scattering datum $\rho(\lam)$ to define the direct scattering map. Let 
$$ m^\sharp(x,\zeta) = \Twomat{m_{11}(x,\zeta)}{ \zeta^{-1} m_{12}(x,\zeta)}{\zeta m_{21}(x,\zeta) }{m_{22}(x,\zeta)}.$$
By \eqref{mpm.sym1}, $m^\sharp$ is an even function of $\zeta$.  Then define
$$ \lambda = \zeta^2, \quad 
n(x,\lambda) = m^\sharp(x,\zeta)
$$
The map
$$ \twomat{a}{b}{c}{d} \mapsto \twomat{a}{\zeta^{-1}b }{\zeta c }{d} $$
is an automorphism of $2 \times 2$ matrices and commutes with differentiation in $x$. 
It follows that the functions $n^\pm$ obtained from $M_\pm$ by this map obey
\begin{subequations}
\label{n}
\begin{align}
\label{n.de}
\frac{dn^\pm}{dx}	&=	-i\lambda \ad \sigma  (n^\pm) +\offdiagmat{q}{\lam \qbar} n^\pm + P n^\pm\\
\label{n.ac}
\lim_{x \rarr \pm \infty} n^\pm(x,\lam)	&=	\One
\end{align}
\end{subequations}
and are related by
\begin{equation}
\label{n.T}
n^+(x,\lam) =	n^-(x,\lam)
	e^{-i\lambda x \ad \sigma } 
	\Twomat{\alpha(\lam)}
				{\beta(\lam)}
				{\lam \overline{\beta(\lam)}}
				{\overline{\alpha(\lam)}}
\end{equation}
where
$$ \alpha(\lam) = a(\zeta), \quad \beta(\lam) = \zeta^{-1} \bb(\zeta) $$
and the relation 
$$  |\alpha(\lam)|^2 - \lam |\beta(\lam)|^2 = 1 $$
holds. We used the symmetry relations \eqref{ab.sym} to compute the form of the transition matrix in \eqref{n.T}.  Denoting by $n$ one of $n^+$ or $n^-$, we also have 
\begin{equation}
\label{n.sym}
n_{22}(x,\lam) = \overline{n_{11}(x,\lam)}, 
\quad 
n_{12}(x,\lam) = \lam^{-1} \overline{n_{21}(x,\lam)}
\end{equation}
so that one column or row of $n(x,\lam)$ determines $n(x,\lam)$ completely.  
Finally we define
\begin{equation}
\label{rho}
\rho(\lam) = \beta(\lam)/\alpha(\lam).
\end{equation}
In Lemma \ref{lemma:rho.to.ab} we show that $a$, $\ba$, $b$, $\bb$ can be reconstructed from $\rho$.

\begin{definition}
\label{def:R}
The map $q \mapsto \rho$ defined by \eqref{n}, \eqref{n.T}, and \eqref{rho} is called the \emph{direct scattering map} and denoted $\calR$.
\end{definition}

Similarly,  one can reduce the RHP  \eqref{BC.jump} with contour $\Sigma$ and jump matrix \eqref{BC.v.right} to a RHP with contour $\bbR$ using symmetry
\cite{XuFan12}. It follows from the parity properties of the scattering data
(see \eqref{ab.sym}) and the jump relation \eqref{BC.jump} that the mapping
$$
M(x,\zeta) \mapsto \Twomat{M_{11}(x,-\zeta)}
									{-M_{12}(x,-\zeta)}
									{-M_{21}(x,-\zeta)}
									{M_{22}(x,-\zeta)}
$$
preserves the solution space of the RHP. This fact,
together with unicity of the solution (Lemma \ref{lemma:unique.RH.M}),
implies that the diagonal entries of $M_\pm$ are even 
under the reflection $\zeta \mapsto -\zeta$, while the off-diagonal entries are odd under this reflection (see Lemma \ref{lemma:mu.even-odd}). 

Let
$$ 
M^\sharp(x,\zeta) = \Twomat{M_{11}(x,\zeta)}
											{\zeta^{-1} M_{12}(x,\zeta)}
											{\zeta M_{21}(x,\zeta)}
											{M_{22}(x,\zeta)}
$$ 
and define
$$
N(x,\zeta^2) = M^\sharp(x,\zeta).
$$
One  arrives at the RHP
\begin{align} 
\label{N.jump}
N_+(x,\lambda)	
	&=	N_-(x,\lambda) e^{-i\lambda x \ad \sigma } J(\lambda)\\
\nonumber
J(\lam)				
	&=	\Twomat{1-\lambda |\rho(\lambda)|^2}
						{\rho(\lambda)}
						{-\lambda \overline{\rho(\lambda)}}
						{1}
\end{align}
corresponding to the RHP with jump matrix \eqref{BC.v.right}, where $\rho(\lam) = \zeta^{-1} \bb(\zeta)/a(\zeta)$. 
A similar computation for the RHP  with jump matrix \eqref{BC.v.left}
leads to a RHP in the $\lam$ variable with $\rho(\lam)$ replaced by
$ \brho(\lam) = \zeta^{-1} \bb(\zeta)/\ba(\zeta)$. 

However, this RHP is not properly normalized: a careful computation shows that the piecewise analytic function $N(x,z)$ on $\bbC \setminus \bbR$ has the asymptotics
$$N(x,z) \sim \Twomat{1}{0}{(i/2)\overline{q(x)}}{1}+\bigO{z^{-1}}$$
as $z \rarr \infty$.  It is more effective to consider the row-wise RHP
\begin{equation}
\label{N}
\begin{split}
\bfN_+(x,\lambda)	&=	\bfN_-(x,\lam) e^{-i\lam x \ad(\sigma)} J(\lam)\\
\lim_{\lambda \rarr \infty} \bfN_\pm(x,\lam) &= \bfe_1
\end{split}
\end{equation}
where
$$ \bfN(x,\lam) = \left( N_{11}(x,\lam), N_{12}(x,\lam) \right),  \; {\rm and} \;\;  \bfe_1 =(1,0).$$
A similar problem occurs in the study   of  KdV in a small dispersion limit \cite{DVZ97}.  

One recovers $q$ from the relation
\begin{equation}
\label{q}
 q(x) = 2i \lim_{z \rarr \infty} z N_{12}(x,z)
 \end{equation}
where $z \rarr \infty$ in $\bbC \setminus \bbR$. 
As we show in \S \ref{sec:RH.recon} (see  Proposition \ref{prop:recon} and \eqref{q.recon.nu}), one can also compute the limit in \eqref{q} from the integral formula
\begin{equation}
\label{q.recon}
q(x) = -\frac{1}{\pi} \int e^{-2i\lam x} \rho(\lam) \nu_{11}(x,\lam) \, d\lam.
\end{equation}
where $\nu=(\nu_{11},\nu_{12})$ solves the Beals-Coifman integral equation \eqref{RH.n.BC} associated with the RHP \eqref{N}.

\begin{definition}
\label{def:I}
The mapping $\rho \mapsto q$ determined by the RHP \eqref{N} and the asymptotic formula 
\eqref{q} is called the \emph{inverse scattering map} and denoted $\calI$. 
\end{definition}

From the evolution equations \eqref{ab.ev}, we see that\footnote{{This equation differs by a minus sign from Kaup-Newell \cite[eq.\ (34)]{KN78} since Kaup and Newell take $\rho(\lam)=\zeta^{-1}b(\zeta)/a(\zeta)$ whereas we take $\rho(\lam)= \zeta^{-1} \bb(\zeta)/a(\zeta)$.}}
$$ \dot\rho(\lam,t)= -4i \lam^2 \rho(\lam,t).$$ In Lee's thesis, it is shown that for Schwartz class initial condition $q_0$ satisfying certain spectral conditions, the formula
\begin{equation}
\label{DNLS2.sol}
q(x,t) = \calI \left[ e^{{-4it(\diamond)^2}} \left( \calR q_0 \right)(\diamond) \right](x)
\end{equation}
gives a classical solution to \eqref{DNLS2}. Lee's class includes $\calS(\bbR) \cap U$ where
$U$ is the open set in Theorem \ref{thm:DNLS2} below. We give a self-contained proof 
of the solution formula \eqref{DNLS2.sol} in section \ref{sec:RHP-t} and  Appendix \ref{app:time}.

\subsection{Summary of Results}
\label{sec:intro.summary}

Given the solution formula \eqref{DNLS2.sol} for Schwartz class data, the key to obtaining a globally defined solution map with good continuity properties is to prove precise continuity properties of the maps $\calR$ and $\calI$ in a natural function space in which the map 
$\rho \mapsto e^{{ -4it(\diamond)^2}} \rho(\diamond)$ is continuous.

Let
$ H^{2,2}(\bbR)$ be the completion of $\calS(\bbR)$ 
in the norm
$$ \norm{q} {H^{2,2}(\bbR)} = \norm{[1+ (\dotarg)^2] q(\dotarg)}{2} + \norm{q''}{2}. $$
The main result of our paper is expressed in the following theorem:

\begin{theorem}
\label{thm:DNLS2}
There is a spectrally determined open subset $U$ of $H^{2,2}(\bbR)$ containing a neighborhood of $0$ so that the solution map \eqref{DNLS2.sol} for \eqref{DNLS2}
\begin{align*}
H^{2,2}(\bbR) \times \bbR 	&	\longrightarrow	H^{2,2}(\bbR) \\
(q_0, t)								&\mapsto				q(\dotarg,t)
\end{align*}
is continuous, and Lipschitz continuous in $q_0$ for each $t$.
\end{theorem}

\begin{remark}\label{rem:U}
The set $U$ consists, by definition, of those $q \in H^{2,2}(\bbR)$ for which the scattering 
data $\alpha(\lam)$ is everywhere nonzero and has a zero-free analytic 
extension $\alpha(z)$ to $\bbC^-$. For $q=0$, the transition matrix 
$T(\zeta)$ is the identity matrix (so $\alpha(\lam)\equiv 1$).
It follows from the continuity of $q \mapsto 
\alpha$ as a map from $H^{2,2}(\bbR)$ to itself that, given any $c>0$,  there is a neighborhood $U'$ 
of $q=0$ in $H^{2,2}(\bbR)$ so that, for all $q \in U'$,  $|\alpha(q)(\lam)| 
\geq c >0$ for all $\lam \in \bbR$. 
To see that $\alpha(q)(\dotarg)$ has a zero-free analytic continuation 
to $\Omega^-$ for $q$ in a neighborhood of zero, one uses the following 
facts. First, zeros of $a$ in $\Omega^-$ correspond to $L^2$ vector-valued 
solutions of the problem \eqref{LS} with exponential decay as 
$x \rarr \pm\infty$. Second, one can show that such $L^2$ solutions satisfy a 
homogeneous Fredholm-type integral equation. Third, if $\norm{q}{H^{2,2}}$ 
is sufficiently small, the Fredholm operator has trivial kernel, so any such solutions 
must be identically zero. We omit details.
\end{remark}


Since the gauge transformation \eqref{q.gauge} defines a Lipschitz continuous self-mapping $\calG$ of $H^{2,2}(\bbR)$ onto itself with $\calG(0)=0$, we immediately obtain:

\begin{corollary}
\label{cor:DNLS}
There is an open subset of $H^{2,2}(\bbR)$ containing a neighborhood of $0$ so that the solution map for 
\eqref{DNLS1}
\begin{align*}
H^{2,2}(\bbR) \times \bbR 	&	\longrightarrow	H^{2,2}(\bbR) \\
(u_0, t)								&\mapsto				u(\dotarg,t)
\end{align*}
is continuous, and Lipschitz continuous in $u_0$ for each $t$.
\end{corollary}

The technical core of our paper consists of the following two continuity results for the direct and inverse scattering maps.

\begin{theorem}
\label{thm:R}
There is an open subset $U$ of $H^{2,2}(\bbR)$ containing a neighborhood of $0$ so that the direct scattering map $\calR$, initially defined on $\calS(\bbR) \cap U$, extends to a  Lipschitz continuous map from $U$ into $H^{2,2}(\bbR)$. Moreover, $\calR(U)$ is invariant under the map 
$\rho\mapsto {e^{-4it(\diamond)^2}} \rho(\diamond)$, 
and also contains an open neighborhood of $0$ in $H^{2,2}(\bbR)$. 
\end{theorem}

\begin{theorem}
\label{thm:I}
There is an open subset $V$ of $H^{2,2}(\bbR)$ containing $\calR(U)$
so that the inverse
scattering  map $\calI$, 
initially defined on $\calS(\bbR) \cap V$, 
extends to a Lipschitz continuous map 
from 
$V$  to $H^{2,2}(\bbR)$ 
with the property that
$\calR \circ \calI$ is the identity map on the open set $V$ and
$\calI \circ \calR$ is the identity map on the open set 
$U$ of Theorem \ref{thm:R}.
\end{theorem}

We emphasize that results from Lee's thesis \cite{Lee83} already imply that the direct scattering  map is continuous from $\calS(\bbR) \cap U$ into $\calS(\bbR)$, and that the inverse map is continuous from $\calS(\bbR) \cap V$ to $\calS(\bbR)$. Our contribution is to prove sharp continuity estimates between weighted Sobolev spaces.

The space $H^{2,2}(\bbR)$ is invariant under the Fourier transform, and, for $\rho \in H^{2,2}(\bbR)$, the map 
$t \mapsto e^{-4it\lam^2 t} \rho(\lam)$ describes a 
continuous curve in $H^{2,2}(\bbR)$. 
Since the nonlinear maps $\calR$ 
and $\calI$ linearize respectively to the direct and inverse Fourier transform, the space $H^{2,2}(\bbR)$ is well-suited to study the map \eqref{DNLS2.sol}.

\medskip

Given Theorems \ref{thm:R} and \ref{thm:I}, the proof of Theorem \ref{thm:DNLS2} is straightforward.
The solution map $S$ defined by
$$
(q_0,t) \mapsto \calI \left[ e^{{-4it(\diamond)^2}} \left( \calR q_0 \right)(\diamond) \right](\dotarg)
$$ 
has the claimed continuity properties by Theorems \ref{thm:R} and \ref{thm:I}. Thus, from Theorem \ref{thm:I} that $S(q_0,0)=q_0$. Moreover, the solution map gives a classical solution of \eqref{DNLS2} by Lee's thesis or Theorem \ref{thm:RHP.t.q}. The result for $q \in H^{2,2}(\bbR)$ now follows from Lipschitz continuity of $\calR$ and $\calI$.

In a subsequent paper \cite{LPS16}, we establish  the large-time asymptotics of the solution $q(x,t)$ for \eqref{DNLS2} 
using the steepest descent method of Deift-Zhou \cite{DZ93,DZ03} as  
recast by Dieng-McLaughlin \cite{DM08}.
 Finally, we mention that Kitaev and Vartanian developed the inverse scattering method  and found the long time behavior of solutions
for  Schwarz class initial data with and without solitons. 
\cite{KV97,KV99}.

The paper is organized as follows.
In Section \ref{sec:prelim},  we present some useful  tools of functional and complex analysis.
They include interpolation  and Sobolev embedding lemmas for weighted spaces that will be used in the analysis on the direct 
map, as well as basic properties of Cauchy projectors onto the lower and upper half complex  planes that come into play in the analysis of the
RHP. 
We also recall the Beals-Coifman formulation of the RHP which 
shows that the RHP is equivalent to the Beals-Coifman integral equation.
Section  \ref{sec:direct} is devoted to  Lipschitz  continuity properties of the direct scattering map, initially defined on $\calS(\bbR) \cap U$ and extended by continuity to
$H^{2,2}(\bbR) \cap U$.  We discuss the construction of Beals-Coifman solutions 
in Section \ref{sec:BC}.
In Section \ref{sec:RH.Schwartz}, we study the RHP and the reconstruction of $q$ for Schwartz class scattering data, and in 
Section \ref{sec:inverse}, we prove Lipschitz continuity of the inverse scattering map, initially defined on $\calS(\bbR) \cap V$ but extended by continuity  to $H^{2,2}(\bbR)\cap V$. 
In Section \ref{sec:RHP-t}, we give a self-contained proof that the formula \eqref{DNLS2.sol} gives a classical solution of \eqref{DNLS2} if the initial data belong to $\calS(\bbR)\cap U$.
For sake of completeness,  several technical calculations and proofs are presented in Appendices. In Appendix \ref{app:gauge}, we formulate  the Lax pairs for \eqref{DNLS1} and \eqref{DNLS2} 
 and show their equivalence through the gauge transformation \eqref{q.gauge}.
Finally, in Appendix \ref{app:time}, we supply some technical computations needed in Section \ref{sec:RHP-t}.

We close this introduction with a guide to notation used for various solutions of the linear systems defining the direct scattering map and the Riemann-Hilbert problem defining the inverse scattering map. Lower case letters $m$, $m^\sharp$, $n$ denote solutions to the linear systems \eqref{LS.red}, \eqref{n}, and $\pm$ superscripts designate Jost solutions obeying a boundary condition at $\pm \infty$. The boldface letter $\bfn$ denotes the renormalized first column of $n$ (see \eqref{bold.n}). Upper case letters $M$, $N$ denote respectively solutions of the Riemann-Hilbert problems \eqref{BC.jump}, \eqref{N.jump}, and subscripts $\pm$ designate boundary values on $\Sigma$ (for $M$) or $\bbR$ (for $N$). The boldface letter $\bfN$ denotes a row-vector of $N$.  The RHP's for $M$ and $\bfN$ are formulated precisely as Problems \ref{prob:RH.M} and \ref{prob:RH.n}. The letters $\mu$ and $\nu$ denote, respectively, the $2\times 2$ matrix-valued solution for the Beals-Coifman integral equation corresponding to Problem \ref{prob:RH.M} and the row vector-valued solution to the Beals-Coifman integral equation for Problem \ref{prob:RH.n}. The Beals-Coifman equation for $\nu=(\nu_{11},\nu_{12})$ can be reduced for a scalar integral equation for $\widetilde{\nu}=\nu_{11}$ and in turn to an equation for $\nu^\sharp=\widetilde{\nu}-1$ which is studied in depth in Section \ref{sec:RH.Mapping}.

\section{Preliminaries}
\label{sec:prelim}

%
%

We start this section 
with some  functional analysis results, namely
estimates for Volterra-type integral equations (Section \ref{sec:volterra}) 
useful  in the analysis of the direct scattering map. 
Since the spectral problem and the RHP are formulated for matrix-valued functions, we present in Section \ref{sec:matrix} some classical operations on matrices.
We then turn to complex analysis tools that are central for the study of the inverse scattering map  and recall some properties  of Cauchy integrals and Cauchy operators in Section \ref{sec:cauchy}. 
We present a useful change-of-variables formula for the Cauchy projectors in Section \ref{sec:lee.cv}.
Finally,   we discuss the key ideas leading to  the reduction of the  RHP to the so-called  Beals-Coifman  integral equation in Section  \ref{sec:BC-integral-eq}. 

\subsection{Some Tools from Functional Analysis}

\subsubsection{Notations}
\label{sec:notations}

If $X$ and $Y$ are Banach spaces, we   denote  by $\calB(X,Y)$ the Banach space of bounded linear operators
from $X$ to $Y$. 
We write $\calB(X)$ for $\calB(X,X)$. If $A$ is a Hilbert-Schmidt operator on a Hilbert space $\calH$, we denote by $\norm{A}{\HS}$ the Hilbert-Schmidt norm of $A$.
If $I$ is an interval on the real line, $C^0(I,X)$ denotes the space of continuous functions on $I$ taking 
values in $X$.  It is 
equipped with the norm
$$ \norm{f}{C^0(I,X)} = \sup_{x \in I } \norm{f(x)}{X}. $$ 
We  write $C^0(X)$ if there is no possibility of confusion.

We denote by $D$ the operator $-i(d/dx)$, by $\langle x \rangle$ the smooth function $(1+x^2)^{1/2}$ and by $\langle D \rangle$ the Fourier multiplier with symbol
$ (1+|\xi|^2)^{1/2}$.  Let  $L^{2,k}(\bbR)$  be the space of  $L^2(\bbR)$-functions with $\langle \dotarg \rangle^k f (\dotarg) \in L^2(\bbR)$
and $H^{\alpha,\beta}(\bbR)$ the completion of $\calS(\bbR)$  in  the norm
$$
\norm{f}{H^{\alpha,\beta}} = 
	\left( \norm{\langle D \rangle^\alpha f}{2}^2 + \norm{\langle x \rangle^\beta f}{2}^2 \right)^{1/2}. 
$$
The analysis of the scattering maps will be performed in the space   $H^{2,2}(\bbR)$.
Note that  $\norm{\langle x \rangle u'}{2} \le C \norm{u}{H^{2,2}}.$
We normalize the Fourier transform as follows:
\begin{align*}
\widehat f( \lambda)  &:= \left( \calF f \right)(\lam)  
								= \int_{-\infty}^\infty e^{-2i\lam x} f(x) \, dx\\
\widecheck{g}(x)  		&:=  \left( \calF^{-1} g \right) (x) 
								=	\frac{1}{\pi} \int_{-\infty}^\infty e^{2i\lam x} g(\lam) \, d\lam.
\end{align*}

\subsubsection{Volterra Integral Equations}\label{sec:volterra}

\begin{lemma}
\label{lemma:Volterra}
Suppose that $X$ is a Banach space and consider the Volterra-type integral equation 
\begin{equation}
\label{Volterra}
u(x) = f(x) + (Tu)(x)
\end{equation}
on the space $C^0(\bbR^+,X)$, where $f \in C^0(\bbR^+,X)$ and $T$ is an integral operator on $C^0(\bbR^+,X)$. 
$ f^*(x) = \sup_{y \geq x} \norm{f(y)}{X},$
and assume  there is a nonnegative function $h \in L^1(\bbR^+)$  such that
\begin{equation}\label{volt}
(Tf)^*(x) \leq \int_x^\infty h(t) f^*(t) \, dt. 
\end{equation}
Then   for each $f$, equation \eqref{Volterra} has a unique solution. Moreover, the resolvent 
$(I-T)^{-1}$ obeys the bound
\begin{equation} \label{resol:est}
 \norm{(I-T)^{-1}}{\calB(C^0(I,X))} \leq \exp\left( \int_0^\infty h(t) \, dy \right). 
 \end{equation}
\end{lemma}

Estimate \eqref{resol:est} is obtained by expanding  $(I-T)^{-1}$ in powers of $T$ and using \eqref{volt}
iteratively
in the form 
$$
(T^n f)^*(x)	\leq \frac{1}{n!}	 \left( \int_x^ \infty h(y) \, dy \right)^n f^*(x).
$$
to get a convergent series.

\begin{remark}
There is an obvious analogue of Lemma \ref{lemma:Volterra} for 
the negative half-line. 
\end{remark}

\subsection{Matrix Operations}
\label{sec:matrix}
For  a $2\times 2$ matrix 
$$A = \twomat{a}{b}{c}{d}, $$
we denote its Frobenius norm
$|A| = \sqrt{a^2+b^2+c^2+d^2}.$
 We will often use the decomposition
$$ A = A_\diag+ A_\offd$$
where
$$
A_{\diag} = \twomat{a}{0}{0}{d}, \quad 
A_{\offd} = \twomat{0}{b}{c}{0}.
$$
Note that
$$
(AB)_\diag = A_\diag B_\diag + A_\offd B_\offd, \quad (AB)_\offd = A_\diag B_\offd + A_\offd B_\diag.
$$
Let
$$ \sigma = \twomat{1}{0}{0}{-1}$$
and define
$ \ad \sigma (A) = [\sigma,A]= \sigma A - A \sigma.$ We have
$$\ad \sigma (A) = \twomat{0}{2b}{-2c}{0}$$  and  $\ad \sigma (A_\diag)=0$. 
If $A$ is off-diagonal, 
$$ 
(\ad \sigma)^{-1} \twomat{0}{b}{c}{0} 
	= \frac{1}{2} \twomat{0}{b}{-c}{0}.
$$
The exponential operator
$e^{i\theta \ad \sigma }$ 
acts linearly on $2 \times 2$ matrices:
$$ e^{i\theta\ad \sigma } \twomat{a}{b}{c}{d} = \twomat{a}{e^{2i\theta}b}{e^{-2i \theta}c}{d}. $$

\subsection{Cauchy Projections and Hilbert Transform}
\label{sec:cauchy}
\label{sec:prelim.Cauchy}

\subsubsection{Contours}
\label{sec:prelim.contour}

Figure \ref{fig:contours} displays the oriented contours under consideration.
Integration on the oriented contour $\Sigma$  
may be parameterized as follows:
\begin{align}
\label{Sigma.int}
\int_\Sigma f(\zeta) \, d\zeta 
	=	\int_0^\infty f(t) \, dt - i \int_0^\infty f(it) \, dt 
	 - \int_0^\infty f(-t) \, dt + i \int_0^\infty f(-it) \, dt
\end{align} 
and
integration with respect to arc length is parameterized as
$$\int_\Sigma f(\zeta) \, \left| d\zeta \right|
	=	\int_0^\infty \left( f(t) + f(-t) + f(it) + f(-it) \right) \, dt.
$$
Denote  by $L^p(\Sigma)$ the space of measurable functions on $\Sigma$ with norm
$$ \norm{f}{L^p(\Sigma)} \equiv \left( \int_\Sigma |f(\zeta)|^p \, \left| d\zeta \right| \right)^{1/p} $$
finite.  We
say that $f \in L^1(\Sigma)$ is \emph{even} if $f(-\zeta)= f(\zeta)$ and \emph{odd} if $f(-\zeta)=-f(\zeta)$. It is easy to see that the integral of an even function is zero while the integral of an odd function is given by
$$ \int_\Sigma f(\zeta) \, d\zeta = 2\int_0^\infty f(t) \, dt -2i \int_0^\infty f(it) \,  dt. $$ 
A short computation using \eqref{Sigma.int} shows that for any function $f \in H^1(\Sigma)$,  
\begin{equation}
\label{lee.cv}
 \int_\Sigma f(\zeta) \, d\zeta =
	\int_{\bbR} \left(f(\sqrt{u}) - f(\sqrt{-u})\right) \frac{du}{2 \sqrt{u}}. 
\end{equation}
where $\bbR$ is given the usual orientation.

\subsubsection{Cauchy Projectors}
We recall some basic facts about the Cauchy transform and the Cauchy projectors. 
See, for example, Deift-Zhou \cite[Section 2]{DZ03} for details and references.

Let  $\Lambda$ denote an oriented contour in the complex plane which can be either $\Sigma$ or $\bbR$  (as plotted in Figure \ref{fig:contours}). $\Omega^\pm$ denotes the region $\pm\Imag(\zeta^2)>0$ if $\Lambda = \Sigma$, 
and the region $\pm \Imag(\lambda) >0$ if $\Lambda = \bbR$.

For $f \in L^p(\Lambda)$,  $p \in (1,\infty)$, the Cauchy integral
$$ 
(Cf)(z) 
	= \frac{1}{2\pi i} \int_\Lambda \frac{1}{s-z} f(s) \, ds
$$
defines a function bounded and analytic in $\bbC \setminus \Lambda$. The nontangential limits 
$$
(C^\pm f)(\zeta) = \lim_{z \rarr \zeta, \,\, z \in \Omega^\pm} (Cf)(z)
$$
exist for almost every $z \in \Lambda$, and the estimate
$$ 
\norm{C^\pm f}{p} \leq c_p \norm{f}{p}
$$
holds. We have the Plemelj-Sokhotski formula
\begin{equation}
\label{Plemelj-Sokhotski}
C^\pm f = \pm \frac{1}{2}f -  \frac{1}{2} Hf
\end{equation}
where $H$ is the Hilbert transform
$$
(Hf)(z) = 
	\lim_{\eps \darr 0} \frac{1}{\pi i} \int_{\Lambda \cap \{|s-z|>\eps\}} \frac{1}{z-s} f(s) \, ds
$$
From this, it follows that $C^+ - C^- = I$ on $L^p(\Lambda)$. If $p=2$,  $C^\pm$ are orthogonal projections. 
In particular, if the contour is $\bbR$,  the Cauchy projectors $C^\pm$ are  simply defined  
via the Fourier transform :
\begin{align}
\label{C-+.Fourier}
\left(C^+ f\right)(\lam)	&=		\frac{1}{\pi} \int_0^\infty e^{2i\lam \zeta} \widehat{f} (\zeta) \, d\zeta\\
\label{C-.Fourier}
\left(C^-f \right)(\lam)	&=    \frac{1}{\pi}  \int_0^{-\infty}  e^{2i\lam \zeta} \widehat{ f}(\zeta) \, d\zeta
\end{align}
A short computation using \eqref{Sigma.int} shows that the Hilbert transform $H$ on $\Sigma$ preserves the subspaces of odd and even functions on $\Sigma$. We will use this fact in the sequel. 
We will also make use of the following commutator identities. 

\begin{lemma}
\label{lemma:commutator}
Suppose that $m$ is a nonnegative integer and $f \in L^{2,m}(\bbR)$. Then
\begin{align}
\label{Cpm.com}
\zeta^m C^\pm \left[ f \right](\zeta) 
	&= 			C^\pm \left[ (\dotarg)^m f(\dotarg) \right] - \sum_{j=0}^{m-1} \zeta^{m-1-j} \frac{f_j}{2 \pi i}, \quad
f_j \equiv \int_\Sigma s^j f(s) \, ds.
\end{align}
\end{lemma}

\begin{proof}  The case $m=1$  follows from \eqref{Plemelj-Sokhotski} and the commutator identity
\begin{equation} \label{comm-id}
[z,H] f = \frac{1}{\pi i} \int_\Sigma f. 
\end{equation}
Indeed, one writes
\begin{align*}
C^\pm \left((\dotarg)  f(\dotarg) \right)   &-  \zeta (C^\pm f)(\zeta) = [C^\pm ,\zeta] f \\
&= -\frac{1}{2} [H,\zeta] f =\frac{1}{2\pi i } \int_\Sigma  f(s) ds.
\end{align*}

The general formula  is derived by induction.
\end{proof}

\subsection{Change of Variables in Cauchy Projectors}
\label{sec:lee.cv}

The following change of variables formula for the Cauchy transform 
appears in Lee  \cite[\S 8]{Lee83}. We reproduce it here for the reader's convenience.

Let $\Sigma = \left\{ \zeta \in \bbC: \Im(\zeta^2)=0 \right\}$ with the orientation shown in Figure \ref{fig:contours}.  The mapping
 $\zeta \mapsto \zeta^2$ maps $\Sigma$ onto $\bbR$ and induces the usual orientation. Let  
 $C_\Sigma$ and $C_{\bbR}$  be the respective Cauchy integrals for $\Sigma$ and $\bbR$. For $u \in \bbR$ we denote by $\sqrt{u}$ the principal branch of the square root function, so that, referring to Figure \ref{fig:contours}:
\begin{itemize}
\item	 $u \mapsto \sqrt{u}$ maps $\bbR$ onto $\Sigma_1 \cup \Sigma_2$, and \smallskip
\item	$u \mapsto -\sqrt{u}$ maps $\bbR$ onto $\Sigma_3 \cup \Sigma_4$.
\end{itemize}

For $f \in H^1(\Sigma)$, define
\begin{align*}
g(u) &= \frac{1}{2} \left( f\left(\sqrt{u})+f(-\sqrt{u}\right) \right),\\
h(u) &= \frac{1}{2\sqrt{u}}\left(f(\sqrt{u})-f(-\sqrt{u})\right).
\end{align*}

\begin{lemma}
\label{lemma:lee.cv}
Let $f \in H^1(\Sigma)$ and $z \in \bbC \setminus \Sigma$. The identities
\begin{equation}
\label{Cauchy.Sigma.to.R}
\left( { C_\Sigma} f \right)(z) = 
		\left( 
			{ C_\bbR} g 
		\right)(z^2) 
	+ 	z \left( 
			{ C_\bbR  } h
		\right)(z^2)
\end{equation}
hold. Moreover, for any $\zeta \in \Sigma$,
\begin{equation}
\label{Cauchypm.Sigma.to.R}
\left( C_\Sigma^\pm f \right)(\zeta) =
\left( C_\bbR^\pm g\right)(\zeta^2) +
\zeta \left( C_\bbR^\pm h\right)(\zeta^2)
\end{equation}
\end{lemma}

\begin{proof}
Using \eqref{lee.cv}, we compute
\begin{align*}
\int_\Sigma \frac{f(s)}{s-z} \, \frac{ds}{2\pi i}	
		&=	\frac{1}{2\pi i} 
							\int_\bbR	
								\left(
									\frac{f(\sqrt{u})}{\sqrt{u}-z} -\frac{f(-\sqrt{u})}{-\sqrt{u}-z} 
								\right)
							\, \frac{du}{2\sqrt{u}}	\\[5pt]
		&=			-\frac{1}{2} 
							\int_\bbR 
								\frac{f(\sqrt{u})+f(-\sqrt{u})}{z^2-u} 
							\, \frac{du}{2\pi i} \\[5pt]
		&\quad	-	z 
							\left( 
								\frac{1}{2} 
									\int_\bbR 
										\frac{f(\sqrt{u})-f(-\sqrt{u})}{\sqrt{u}} 
										\frac{1}{z^2 -u} 
									\, \frac{du}{2 \pi i}
							\right).
\end{align*}
This gives the formula for { $C_\Sigma f$}. Observe that the quadratic mapping $z \mapsto z^2$
takes the regions $\Omega^\pm$ to the half-planes $\bbC^\pm$, and paths approaching $\Sigma$ non-tangentially from $\Omega^+$ (resp.\ $\Omega^-$) are mapped to paths approaching $\bbR$ non-tangentially from $\bbC^+$ (resp.\ $\bbC^-$).  Formula \eqref{Cauchypm.Sigma.to.R}  is now an immediate consequence of \eqref{Cauchy.Sigma.to.R}.
\end{proof}

\begin{remark}
\label{rem:lee.cv}
From Lemma \ref{lemma:lee.cv}, we easily deduce that if $f$ is an odd function on $\Sigma$ and
$ h(u)= f(\sqrt{u})/\sqrt{u} $ then
\begin{equation}
\label{Cauchypm.R.to.Sigma.odd}
\left(C_\Sigma^\pm f\right)(\zeta) = \zeta \left( C_\bbR^\pm h\right)(\zeta^2).
\end{equation}
On the other hand, if $f$ is an even function on $\Sigma$ and 
$g(u) = f(\sqrt{u})$, then
\begin{equation}
\label{Cauchypm.R.to.Sigma.even}
\left(C_\Sigma ^\pm f \right)(\zeta) = \left( C_\bbR^\pm g\right)(\zeta^2).
\end{equation}
\end{remark}

\subsection{Riemann-Hilbert Problem and Beals-Coifman Integral Equation}
\label{sec:BC-integral-eq}

We recall briefly the Beals-Coifman \cite{BC84} approach to RHPs: see, for example,
\cite[Section 2]{DZ03} for a detailed exposition and further references. Let $\Lambda$ be an oriented contour (for our 
purpose, a finite union of oriented lines) that divides $\bbC \setminus \Lambda$ into disjoint open sets $\Omega^+$ 
and $\Omega^-$. Suppose given a $2 \times 2$ measurable, matrix-valued function $v$ on $\Lambda$ with 
$v, v^{-1} \in L^\infty(\Lambda)$. Formally, the normalized RHP $(\Lambda,v)$ is stated as follows:

\begin{problem} \label{problem-RHP}
Find a piecewise analytic function $M(z)$ on $\bbC \setminus \Lambda$ so that 
\begin{itemize}
\item	$M(z) \rarr \One$ as $|z| \rarr \infty$, and

\smallskip

\item	the boundary values  $M_\pm(\zeta)$ obey the jump relation $M_+(\zeta) = M_-(\zeta) v(\zeta)$. 
\end{itemize}
\end{problem}
To formulate this notion more rigorously, we say that a pair of measurable functions $(f_+,f_-)$ on $\Lambda$
belong to $\dee C_\Lambda(L^p)$ if there is a function $h \in L^p(\Lambda)$ with the property that $f_\pm = C_\Lambda^\pm h$. In this case, $f_\pm$ are boundary values of the piecewise analytic function
$$ F(z) = \frac{1}{2\pi i} \int_\Lambda \frac{1}{s-z} \, h(s) \, ds. $$ 
Here $p \in (1,\infty)$; in the sequel, we will be concerned exclusively with the case $p=2$. We now reformulate the normalized RHP 
 $(\Lambda,v)$ as follows: 
 
\begin{problem}
Find a pair of matrix-valued functions $(M_+,M_-)$ with 
\begin{itemize}
\item		$M_\pm - \One \in \dee C_\Lambda(L^p)$, and

\smallskip

\item		$M_+(\zeta) = M_-(\zeta) v(\zeta)$ for a.e.\ $\zeta \in \Lambda$
\end{itemize}
\end{problem}

Given a solution of  the RHP 
$(\Lambda,v)$, we can then recover the piecewise
analytic function $M(z)$ through the Cauchy transform of the function $h$ with $M_\pm - \One  = C^\pm_\Lambda h$:
$$ M(z) = \One + \int_\Lambda \frac{h(s)}{s-z} \frac{ds}{2\pi i}. $$

To derive the Beals-Coifman integral equation, we assume that   the jump matrix $v(\zeta)$ admits a matrix factorization of the 
form
$$ v(\zeta) = (\One - w^-(\zeta))^{-1} (\One + w^+(\zeta))$$
for weight functions $w^\pm \in L^\infty(\Lambda)$. 
If we set
$$ \mu(\zeta) = M_+(\zeta)(\One + w^+(\zeta))^{-1} = M_-(\zeta)(\One - w^-(\zeta))^{-1} $$
assuming that $M_\pm$ solve the RHP,  it follows that the additive jump $M_+  - M_-$ is given by
$$ M_+(\zeta) - M_-(\zeta) = \mu(\zeta)(w^+(\zeta)+w^-(\zeta)) $$ so that
the piecewise analytic function $M(z)$ is given by
$$ 
M(z) = \One + \int_\Lambda \frac{1}{s-z} \left[ \mu(s)(w^+(s)+w^-(s))  \right] \frac{ds}{2\pi i}. $$
Taking boundary values from $\Omega^+$, we find that
$$
M_+(\zeta) 
	=	 \mu(\zeta) (\One +w^+(\zeta) )
	=  \One + C^+_\Lambda \left[ \mu(\dotarg) (w^+(\dotarg) + w^-(\dotarg) \right](\zeta).
$$
Using $I = C^+_\Lambda-C^-_\Lambda$, we conclude that
$\mu$ obeys the \emph{Beals-Coifman integral equation}
\begin{equation}
\label{BC.int.eq}
\mu = \One + \calC_w \mu
\end{equation}
where, for any $2 \times 2$ matrix-valued function $h \in L^p(\Lambda)$,
\begin{equation}
\label{BC.int.op}
\calC_w (h) = C_\Lambda^+ (h w^-) + C_\Lambda^-(h w^+).
\end{equation}
In \eqref{BC.int.op}, the operators $C_\Lambda^\pm$ act componentwise on matrix-valued functions. 
Note that $\calC_w$ is a bounded operator from $L^p(\Lambda)$ to itself for any $p \in (1,\infty)$ since
$w^\pm \in L^\infty(\Lambda)$ and $C^\pm_\Lambda$ are bounded operators on $L^p(\Lambda)$. 
An { important} result of the theory is the following 
(see for example \cite[Proposition 2.6]{DZ03}).

\begin{proposition}
\label{prop:BC.op}
The operator $(I-\calC_w)$ has trivial kernel as an operator on $L^p(\Lambda)$ if and only if the RHP $(\Lambda,v)$ on $L^p$ has a unique solution.
\end{proposition}
In applications, $(I - \calC_w)$ { will be}  a Fredholm operator on $L^p(\Lambda)$. In this case, 
we have:

\begin{proposition}
\label{prop:BC.sol}
Suppose that the normalized RHP $(\Lambda,v)$ has at most one solution $M_\pm$. Let 
$$v(\zeta) = (\One - w^-(\zeta))^{-1} (\One + w^+(\zeta))$$
be a factorization of $v$, with the property that the operator $(I-\calC_w)$ is Fredholm. Then,
there exists a unique solution of the Beals-Coifman integral equation \eqref{BC.int.eq} and the unique solution
to the RHP is given by
$M_\pm = \One + C^\pm h$ where $h= \mu(w^+ + w^-)$.
\end{proposition}

\section{The Direct Scattering Map}
\label{sec:direct}

%
%

This section is devoted to the analysis of the map $q \mapsto \rho$ defined by \eqref{n.T} and \eqref{rho} and  the proof of
Theorem \ref{thm:R}. We prove all estimates for $q \in \calS(\bbR)$ and use without comment Lee's result that the Jost solutions and scattering data are infinitely differentiable in their arguments. This fact justifies integrations by parts and differentiation of integral equations that occur in the proofs. The estimates so obtained
show that the map $q \mapsto \rho$ extends to a Lipschitz continuous map on $H^{2,2}(\bbR) \cap U$, where $U$ is the spectrally determined set in Theorem \ref{thm:R}.

Using the symmetry reduction \eqref{n.sym}, it suffices to study the column vector $ (n_{11}^\pm, n_{21}^\pm)^T$. 
Setting
$$ e_\lam(x) =  e^{-2i\lam x} , $$
we have from \eqref{n.de}-\eqref{n.ac}
that $n_{11}^\pm$ and $n_{21}^\pm$ obey the integral equations
\begin{subequations}
\label{n.ie1}
\begin{align}
\label{n11.pre}
n_{11}^\pm(x,\lam)	
	&=	1	-	\intpm
							\left(
								q(y) n_{21}^\pm(y,\lam) + p_1(y) n_{11}^\pm(y,\lam) 
							\right)
					\, dy\\
\label{n21.pre}
 \lam^{-1} n_{21}^\pm (x,\lam)
	&=	-\intpm    e_\lam(y-x)  
					\left( 
						\overline{q(y)} n^\pm_{11}(y,\lam) + p_2(y) \left( \lam^{-1} n_{21}^\pm(y,\lam) \right)
					\right)
					\, dy.					
\end{align}
\end{subequations}
As we will see, these equations imply that $n_{21}^\pm(x,\lam)/\lam$ is regular at $\lam=0$.
To compute $\alpha$ and $\beta$ from $n_{11}^\pm$ and $n_{21}^\pm$, we evaluate \eqref{n.T} at $x=0$ and use the symmetry \eqref{n.sym} to conclude that
\begin{align}
\label{alpha}
\alpha(\lam)	&=	n_{11}^+(0,\lam) \overline{n_{11}^-(0,\lam)} - \lam^{-1} n_{21}^-(0,\lam) \overline{n_{21}^+(0,\lam)},\\[10pt]
\label{beta}
\beta(\lam)		&=	\frac{1}{\lam}\left( n_{11}^-(0,\lam) \overline{n_{21}^+(0,\lam)} - n_{11}^+(0,\lam) \overline{n_{21}^-(0,\lam)}\right).
\end{align}
To prove Theorem \ref{thm:R}, we  need  estimates on the solutions of \eqref{n.ie1} and their derivatives in $\lambda$ as $L^2(\bbR)$-valued functions of $x$. For $\lambda$ in a bounded interval
{$I_0$}, we can study the equations \eqref{n.ie1} directly. In subsection \ref{sec:direct.small}, we will prove:

\begin{proposition}
\label{prop:n.small}
Let $I_0$ be a bounded interval in $\bbR$.
The maps 
$$ q \rarr n_{11}^+(0,\lam), \quad q \rarr n_{21}^\pm(0,\lam)/\lam,$$ 
initially defined for $q \in \calS(\bbR)$, extend to Lipschitz continuous maps from $H^{2,2}(\bbR)$ into 
{$H^2(I_0)$}. 
\end{proposition}

To obtain uniform estimates for large $\lam$, we begin with some simple algebraic manipulations on the solutions of \eqref{n}. 
Using the identity
$$ (-2i\lam)^{-1} (d/dx) e_\lam(x) = e_\lam(x) $$
and integrating by parts in \eqref{n21.pre}, we may remove the 
factor of $\lam$ at the expense of taking derivatives of $q$. Inserting \eqref{n.de} to evaluate the derivative of $n_{11}^\pm$ that occurs in the computation,
we observe some cancellations and obtain that
\begin{subequations}
\label{n.ie2}
\begin{align}
\label{n11}
n_{11}^\pm(x,\lam)
	&=	1	+	\frac{1}{2i} \intpm q(y)	\int_y^{\pm \infty} { e_\lam(z-y)} \qsharp(z) n_{11}^\pm(z,\lam) \, dz \, dy\\
\label{n21}
n_{21}^\pm(x,\lam)
	&=		-	\frac{1}{2i}\overline{q(x)} n_{11}^\pm(x,\lam) 
				- 	\frac{1}{2i} \intpm { e_\lam(y-x) } \qsharp(y) n_{11}^\pm(y,\lam) \, dy	
\end{align}
\end{subequations}
where
\begin{equation}
\label{qsharp}
\qsharp(x) = \overline{q'(x)}+ \overline{q(x)} p_1(x) = \overline{q'(x)}-\frac{i}{2}|q(x)|^2 \overline{q(x)}.
\end{equation}
Note that
$n_{21}^\pm$ does not appear in the equation for $n_{11}^\pm$. 
 We first solve the integral equation \eqref{n11} for $n_{11}^\pm$, and then use its solution to compute $n_{21}^\pm$. 

It is helpful to 
extract the leading order behavior of $n_{11}^\pm$ and $n_{21}^\pm$ for large $\lam$ by setting
\begin{equation}
\label{eta}
\eta_{11}^\pm(x,\lam) = n_{11}^\pm(x,\lam) -1, 
\quad
\eta_{21}^\pm(x,\lam) = n_{21}^\pm(x,\lam) + \frac{1}{2i} \overline{q(x)}. 
\end{equation}
From \eqref{n.ie2} and \eqref{eta}, we conclude that
\begin{subequations}
\begin{align}
\label{eta11}
\eta_{11}^\pm(x,\lam) &= F_\pm(x,\lam)
		+ \left( T_\pm \eta_{11}^\pm\right)(x,\lam)\\
\label{eta21}
\eta_{21}^\pm(x,\lam)		&= G_\pm(x,\lam) -\frac{1}{2i}\overline{q(x)} \eta_{11}^\pm 
- \frac{1}{2i} \intpm e_\lam(y-x)  \qsharp(y) \eta_{11}^\pm(y,\lam) \, dy		
\end{align}
\end{subequations}
where
\begin{subequations}
\begin{align}
\label{F}
F_\pm(x,\lam)	&=	-\intpm q(y) G_\pm(y,\lam) \, dy, 
\\[5pt]
\label{G}
G_\pm(x,\lam)	&=	 -\frac{1}{2i} \intpm  e_\lam(y-x)  \qsharp(y) \, dy	\\[5pt]
\label{T}
\left( T_\pm f \right)(x,\lam) 	
						&=	 \frac{1}{2i} \intpm q(y)	
										\int_y^{\pm \infty}{  e_\lam(z-y)} \qsharp(z) f(z)  \, dz 
\end{align}
\end{subequations}
In terms of the solutions $\eta_{11}^\pm$ and $\eta_{21}^\pm$,
the  functions $\alpha(\lambda)$ and $\beta(\lambda)$ defined in
\eqref{alpha} and \eqref{beta} are expressed as
\begin{align}
\label{alpha2}
\alpha(\lam) -1 	
	&=	\alpha_1(\lam) - \frac{1}{\lam} \alpha_2(\lam)\\
\label{beta2}
\lam \beta(\lam)		
		&=	\beta_1(\lam) + \beta_2(\lam)
\end{align}
where
\begin{align*}
\alpha_1(\lam)	&=		\eta_{11}^+(0,\lam) + \overline{\eta_{11}^-(0,\lam)}
								+	\eta_{11}^+(0,\lam)\overline{\eta_{11}^-(0,\lam)} \\[5pt]
\alpha_2(\lam)	&=		\frac{|q(0)|^2}{4} 		
										+\frac{1}{2i}q(0)\eta_{21}^-(0,\lam)	
										-\frac{1}{2i}\overline{q(0)\eta_{21}^+(0,\lam)}
										+\eta_{21}^-(0,\lam)\overline{\eta_{21}^+(0,\lam)}	
\end{align*}
and
\begin{align*}
\beta_1(\lam)		&=	\left(
										\overline{\eta_{21}^+(0,\lam)}  - \frac{1}{2i}q(0)\eta_{11}^+(0,\lam)
								\right)
							-	\left(
										\overline{\eta_{21}^-(0,\lam)} {  -} \frac{1}{2i}q(0)\eta_{11}^-(0,\lam)
								\right)\\[5pt]
\beta_2(\lam)		&=	\eta_{11}^-(0,\lam)\overline{\eta_{21}^+(0,\lam)}  
							- 	\eta_{11}^+(0,\lam) \overline{\eta_{21}^-(0,\lam)}
\end{align*}
Let $\eta^\pm = (\eta_{11}^\pm, \eta^\pm_{21})$ { and } 
$ I_\infty \equiv \left\{ \lam \in \bbR: |\lam| > 1 \right\}. $
In Section \ref{sec:direct.large}, we will prove:
 
\begin{proposition}
\label{prop:n.large}
The maps 
$$ 
q \rarr \eta^\pm(0,\lam), \quad 
q \rarr \eta^\pm_\lam(0,\lam), \quad 
q \rarr \lam^{-1} \eta^\pm_{\lam\lam}(0,\lam),
$$
initially defined for $q \in \calS(\bbR)$, extend to Lipschitz continuous from $H^{2,2}(\bbR)$ to $L^2(I_\infty)$.  
\end{proposition}
 
\begin{proof}[Proof of Theorem \ref{thm:R}, given Propositions \ref{prop:n.small} and \ref{prop:n.large}]
Propositions \ref{prop:n.small}, \ref{prop:n.large}, and Sobolev embedding show that $q \mapsto \eta(0,\dotarg)$ is Lipschitz continuous from $H^{2,2}(\bbR)$ into $H^1(\bbR)$. It follows from this fact and \eqref{alpha2} that $q \mapsto \alpha-1$ is Lipschitz continuous from $H^{2,2}(\bbR)$ to $H^1(\bbR)$.
Since $\alpha \equiv 1$ if $q=0$, there is an open neighborhood $U$ of zero in $H^{2,2}(\bbR)$ so that $\inf_{\lam \in \bbR} |\alpha(q)(\lam)| >0 $ for all $q \in U$. The map $q \rarr 1/\alpha-1$ is locally Lipschitz continuous from $U$ into $H^1(\bbR)$. 

It follows from Proposition \ref{prop:n.small} and \eqref{alpha}--\eqref{beta} that the map $q \rarr \rho$ is Lipschitz from $U$ to $H^2(I_0)$ for any bounded interval $I_0$. To show that $q \mapsto \rho$ is also Lipschitz from $U$ to $H^{2,2}(I_\infty)$, we need to show that the maps $q \rarr \lam^2 \rho$ and $q \rarr \rho''$ are Lipschitz continuous on $U$.

To prove that $q \rarr \lam^2 \rho$ is Lipschitz continuous from $U$ to $L^2(I_\infty)$, it suffices to show that $q \rarr \lam^2 \beta$ has the same Lipschitz continuity. From \eqref{beta2}, we compute
\begin{align*}
\lambda^2 \beta(\lambda)
	&=	\lambda
					\left(
						\overline{\eta_{21}^+(0,\lambda)} { + }\frac{1}{2i}q(0)\eta_{11}^+(0,\lambda)
					\right)
			-	\lambda
					\left(
						\overline{\eta_{21}^-(0,\lambda)} { - } \frac{1}{2i}q(0)\eta_{11}^-(0,\lambda) 
					\right)\\
	&\quad
			+\lambda
					\left( 
						\overline{\eta_{21}^+(0,\lambda)}\eta_{11}^-(0,\lambda)-
						\overline{\eta_{21}^-(0,\lambda)} \eta_{11}^+(0,\lambda)
					\right)
\end{align*}
To estimate the three right-hand terms, 
we rewrite \eqref{eta21} as 
$$ 
\eta_{21}^\pm =  { -} 
	\frac{1}{2i} \qbar \eta_{11}^\pm 
	-\frac{1}{2i}
		\int_x^{\pm \infty}{ e^{-2i\lambda(y-x)}}q^\sharp(1+\eta_{11}^\pm) \, dy.
$$
Setting $x=0$ and integrating by parts to remove the power of $\lambda$ we obtain
\begin{align*}
\lambda\left(\eta_{21}^\pm (0,\lambda) + \frac{1}{2i}\qbar(0) \eta_{11}^\pm(0,\lambda)\right)
	&=	-\frac{\lambda}{2i} \int_0^{\pm \infty} e^{-2i\lambda y } q^\sharp (1+\eta_{11}^\pm) \, dy\\%
	&=	\frac{1}{4}q^\sharp(0) + \frac{1}{4}q^\sharp(0)\eta_{11}^\pm(0,\lambda)+R^\pm(\lambda)
\end{align*}
where 
\begin{align*}
R^\pm(\lambda) 
	&= 		\left(	
					\frac{1}{4}\int_0^{\pm \infty} e^{-2i\lambda y} (q^\sharp)'(1+\eta_{11}^\pm)
			\right)\\
	&\quad +
			\left(
				\frac{1}{4}\int_0^{\pm \infty} e^{-2i\lambda y} 
							q^\sharp \left[ q  \eta_{21}^\pm + \frac{1}{2i}|q|^2 \eta_{11}^\pm \right]							
			\right).
\end{align*}
We can then compute 
\begin{align*}
\lambda^2 \beta(\lambda) 
	= 	&		\overline{R^+(\lambda)}-\overline{R^-(\lambda)}
				+	\eta_{11}^-(0,\lambda)\overline{R^+(\lambda)}-
					\eta_{11}^+(0,\lambda)\overline{R^-(\lambda)} \\
&+ \frac{1}{4} \overline {q}^\sharp \left[ \overline \eta_{11}^+(0,\lambda)  -\overline \eta_{11}^-(0,\lambda)  \right]	
\end{align*}				
Since $q \mapsto \eta_{11}^\pm(0,\lam)$ is Lipschitz from $H^{2,2}(\bbR)$ to $C^0(I_\infty)$ by Sobolev embedding, it suffices to show that $q \mapsto R^\pm(\lam)$ is Lipschitz from $H^{2,2}(\bbR)$ to $L^2(I_\infty)$. This follows from the estimate
\begin{align*}
\norm{R^\pm}{L^2(I_\infty)}
	&\leq		\norm{(\qsharp)_x}{2}
					\left(
							1	+	\norm{\eta_{11}^\pm}{L^2(\bbR^\pm \times I_\infty)}
					\right)\\[5pt]
	&\quad	 +
				\left(\norm{q^\sharp q}{2} + \norm{q^\sharp|q|^2}{2}\right)
				\left(
						1	+ \norm{\eta_{11}^\pm}{L^2(\bbR^\pm \times I_\infty)} 
							+ 	\norm{\eta_{21}^\pm}{L^2(\bbR^\pm \times I_\infty)}
				\right).
\end{align*}

To prove that $q \rarr \rho''$ is Lipschitz continuous on $U$, we exploit the identity
$$ \left(\frac{\beta}{\alpha}\right)'' 
	= \frac{(\beta)''}{\alpha} 
				- 2 \frac{(\beta)'(\alpha)'}{(\alpha)^2} 
				+ \frac{\beta}{\alpha}
						\left( 
								-\frac{\alpha''}{\alpha}+\frac{ 2(\alpha')^2}{\alpha^2} 
						\right).
$$
From this identity, it suffices to show that the maps
\begin{equation}
\label{maps.lip}
q \mapsto \beta''(\lam),
\quad 
q \mapsto \beta'(\lam), 
\quad 
q \mapsto \lam \beta'(\lam),
\quad
q \mapsto \alpha'(\lam),
\quad
q \mapsto \lam ^{-1}\alpha''(\lam), 
\end{equation}
are Lipschitz continuous from $H^{2,2}(\bbR)$ to $L^2(I_\infty)$, and that the map 
$q \mapsto \lam^{-1}\alpha'(\lam)$ is continuous from $H^{2,2}(\bbR)$ to $C^0(I_\infty)$. 
This last fact will follow from Lipschitz continuity of the maps $q \mapsto \alpha'(\lam)$ and 
$q \mapsto \lam^{-1} \alpha''(\lam)$ from $H^{2,2}(\bbR)$ to $L^2(I_\infty)$ and Sobolev embedding. 
Lipschitz continuity of the maps \eqref{maps.lip} is easily deduced from \eqref{alpha2}, \eqref{beta2}, and Proposition \ref{prop:n.large}.
\end{proof}

\subsection{Small-$\lam$ Estimates}
\label{sec:direct.small}

In this subsection we prove Proposition \ref{prop:n.small}. We give the proofs for $n_{11}^+$ and $n_{21}^+$ since the others are similar. We set
\begin{equation}
\label{bold.n}
\bfn = (n_{11}^+-1,\lam^{-1} n_{21}^+)
\end{equation}
so that \eqref{n.ie1} becomes 
\begin{equation}
\label{n.small}
\bfn = \bfn_0 + T_0 \bfn, \quad \bfn_0 \equiv T_0 \bfe_1
\end{equation}
where
\begin{equation}
\label{T0}
(T_0 h )(x) = -\int_x^\infty K_0(x,y,\lam) h(y) \, dy
\end{equation}
and
\begin{equation}
\label{K0}
K_0(x,y,\lam) = \Twomat{p_1(y)}{q(y)}{e_\lam(y-x) \overline{q(y)}}{p_2(y)}
\end{equation}
so that
\begin{equation}
\label{n0}
\bfn_0	=	\int_x^\infty \Twovec{p_1(y)}{e_\lam(y-x) \overline{q(y)}} \, dy.
\end{equation}
We will establish existence, uniqueness, and estimates on $\bfn$ by studying \eqref{n.small} as a Volterra integral equation. To study $\lam$-derivatives of the solution, we will solve the integral equations
\begin{align}
\label{n.lam}
\bfn_\lam			&=	\bfn_1+ T_0 (\bfn_\lam), & \bfn_1 &\equiv  (\bfn_0)_\lam + (T_0)_\lam \bfn \\[5pt]
\label{n.lamlam}
\bfn_{\lam\lam}		&= \bfn_2 + T_0(\bfn_{\lam\lam}),
&\bfn_2  &\equiv (\bfn_0)_{\lam\lam} +	(T_0)_{\lam\lam}\bfn + 2(T_0)_\lam \bfn_\lam 
\end{align}

We will prove Proposition \ref{prop:n.small} in the following steps. Let $I_0$ denote a bounded interval of $\bbR$, 
which we'll finally set to $I=(-2,2)$. We will write $T_0$ as $T_0(\lam)$ or $T_0(\lam,q)$ to emphasize its dependence on $\lam \in I$ and $q \in H^{2,2}(\bbR)$. First, we obtain basic estimates on $\bfn_0$ and its derivatives (Lemma \ref{lemma:n0})  and obtain 
mapping properties of the operators $T_0$, $(T_0)_\lam$, and $(T_0)_{\lam\lam}$ (Lemma \ref{lemma:T0.lam}). Second, we show that the 
family of operators $(I-T_0(\lam))^{-1}- I$ indexed by $\lam \in I_0$ induces bounded operators  
$$ \widehat{L_0}: C^0(\bbR^+,L^2({ I_0})) \rarr C^0(\bbR^+,L^2({ I_0})), \qquad \widehat{L_0}:  L^2(\bbR^+ \times { I_0}) \rarr L^2(\bbR^+ \times { I_0})$$  
(Lemmas \ref{lemma:res.T0}, \ref{lemma:R0} and Remark 
\ref{rem:R0}).
Third, we solve \eqref{n.small} to prove that
$$\bfn \in C^0(\bbR^+, L^2({ I_0})) \cap L^2(\bbR^+ \times { I_0})$$  
(Lemma \ref{lemma:n}).
Fourth, we use this result to show that 
$$\bfn_1 \in C^0(\bbR^+, L^2({ I_0})) \cap L^2(\bbR^+ \times { I_0})$$ 
and solve \eqref{n.lam} to show that 
$$\bfn_\lam \in C^0(\bbR^+, L^2({ I_0})) \cap L^2(\bbR^+ \times { I_0})$$
(Lemma \ref{lemma:n.lam}).
Fifth, we use this result to show that 
$\bfn_2 \in C^0(\bbR^+,L^2({ I_0}))$ and solve \eqref{n.lamlam} to prove that 
$$\bfn_{\lam\lam} \in C^0(\bbR^+,L^2({ I_0}))$$ 
(Lemma \ref{lemma:n.lamlam}).
Combining these results, we conclude that $\bfn(0,\lam) \in H^2({ I_0})$. 
Lipschitz continuity of the map $q \mapsto \bfn(0,\lam)$ follows from resolvent bounds established on $(I-T_0)^{-1}$ and the second resolvent formula.

In what follows, we define 
\begin{align*}
\gamma_1(y) 			&=2 |q(y)| + |p_1(y)| + |p_2(y)|.
\end{align*}
Note that $\norm{q}{H^{2,2}}$ bounds $\norm{\gamma_1}{L^1}$ and $\norm{\gamma_1}{L^{2,2}}$. 

\bigskip

(1) \textbf{Estimates on $\bfn_0$ and $T_0$}.
Let 
\begin{align*}
g_1(x,y,\lam)	&= 	-2i(x-y) e_\lam(x-y) \overline{q(y)}, \\
g_2(x,y,\lam)	&=	4(x-y)^2 e_\lam(x-y) \overline{q(y)}
\end{align*}
Then
\begin{equation}
\label{n0.lams}
(\bfn_0)_\lam	= \Twovec{0}{\dint_x^\infty g_1(x,y,\lam) \, dy}, 
\quad
(\bfn_0)_{\lam\lam} = \Twovec{0}{\dint_x^\infty g_2(x,y,\lam) \, dy}
\end{equation}
while the integral kernels of $(T_0)_\lam$ and $(T_0)_{\lam\lam}$ are
\begin{align}
\label{K0.lam}
(K_0)_\lam(x,y,\lam) &= \Twomat{0}{0}{g_1(x,y,\lam)}{0},  
\\[5pt]
\label{K0.lamlam}
(K_0)_{\lam\lam}(x,y,\lam) &= \Twomat{0}{0}{g_2(x,y,\lam)}{0}.
\end{align}

\begin{lemma}
\label{lemma:n0}
Let $I_0$ be a bounded interval. The following estimates hold.
\begin{align}
\label{n.est}
\norm{\bfn_0}{C^0(\bbR^+,L^2({ I_0}))} &\leq \norm{p_1}{L^1} + \norm{q}{L^2}, \quad
\norm{\bfn_0}{L^2(\bbR^+ \times { I_0})}\leq \norm{q}{L^{2,1/2}}.\\
\label{n.lam.est}
\norm{(\bfn_0)_\lam}{C^0(\bbR^+,L^2({ I_0}))} &\leq \norm{q}{L^{2,1/2}} , \quad
\norm{(\bfn_0)_\lam}{L^2(\bbR^+ \times I)} \lesssim \norm{q}{L^{2,1}} \\
\label{n.lamlam.est}
\norm{(\bfn_0)_{\lam\lam}}{C^0(\bbR^+,L^2({ I_0}))} &\leq \norm{q}{L^{2,1}}.
\end{align}
\end{lemma}

\begin{proof}
To prove \eqref{n.est}, we note that the first component is independent of $\lam$, bounded by $\norm{p_1}{L^1}$, and continuous. To bound the second component, let $\varphi \in C_0^\infty({ I_0})$,  compute 
$$ \int_I \varphi(\lam) \int_x^\infty e_\lam(y-x) \overline{q(y)} \, dy =
    \int_x^\infty \widehat{\varphi}(y-x) \overline{q(y)} \, dy
 $$
so that 
$$ 
\norm{\int_x^\infty e_{(\dotarg)}(x-y) \overline{q(y)} \, dy}{L^2({ I_0})} 
\leq
\left(\int_x^\infty |q(y)|^2 \, dy\right)^{1/2}.
$$
The first estimate is immediate and the second follows by integration in $x$.

A similar argument shows that
$$ 
\norm{\int_x^\infty g_1(x,y,\dotarg) \, dy}{L^2({ I_0})} 
\leq 
\left( \int_x^\infty y |q(y)|^2 \, dy \right)^{1/2}
$$
from which \eqref{n.lam.est} follows.

Similarly,
$$
\norm{\int_x^\infty g_2(x,y,\dotarg) \, dy}{L^2({ I_0})}
\leq
\left( \int_x^\infty y^2 |q(y)|^2 \, dy \right)^{1/2}.
$$
\end{proof}

The operator $(T_0)_\lam$ induces  linear mappings $L^2(\bbR^+ \times { I_0}) \rarr L^2(\bbR^+\times { I_0})$ and $L^2(\bbR^+ \times { I_0}) \rarr C^0(\bbR^+, L^2({ I_0}))$ by the formula
$ g(x,\lam) = (T_0)_\lam(f(\dotarg,\lam))(x)$, and similarly for $(T_0)_{\lam\lam}$. We will need the following estimates on these induced maps.

\begin{lemma}
\label{lemma:T0.lam}
Suppose that $q \in H^{2,2}(\bbR)$. The following operator bounds hold uniformly 
in $q \in H^{2,2}(\bbR)$, and the operators are Lipschitz functions of $q$.
\begin{itemize}
\item[(i)]		$\norm{(T_0)_\lam}{L^2(\bbR^+ \times { I_0}) \rarr L^2(\bbR^+\times { I_0})} 
						\lesssim \norm{q}{L^{2,3/2}}$,
\item[(ii)]	$\norm{(T_0)_\lam}{L^2(\bbR^+ \times { I_0}) \rarr C^0(\bbR^+, L^2({ I_0}))} 
						\lesssim \norm{q}{L^{2,2}}$
\item[(iii)]	$\norm{(T_0)_{\lam\lam}}{L^2(\bbR^+ \times { I_0}) \rarr C^0(\bbR^+,L^2({ I_0}))} 
						\lesssim \norm{q}{L^{2,2}}$
\end{itemize}
\end{lemma}

\begin{proof}
For an operator $T(\lam)$ with integral kernel $k(x,y,\lam)$ satisfying the estimate
$ \sup_{\lam \in I} |k(x,y,\lam)| \leq h(y)$ and satisfying $k(x,y,\lam)=0$ if $x>y$, the 
$\calB(L^2(\bbR^+ \times { I_0}))$-norm is controlled by 
$$
\left(\int_0^\infty \int_x^\infty h(y)^2 \, dy\, dx\right)^{1/2} 
= 
\left(  \int_0^\infty y \, h(y)^2 \, dy\right)^{1/2}
$$
and the $\calB(L^2(\bbR^+ \times { I_0}),  C^0(\bbR^+,L^2({ I_0})))$-norm is controlled by 
$$\sup_x \left(\int_0^\infty h(y)^2 \, dy\right)^{1/2}.$$ 
The conclusions follow from this observation and 
the estimates
$$ |g_1(x,y,\lam)| \leq |y||q(y)|, \quad |g_2(x,y,\lam)| \leq y^2 |q(y)| $$
true for $x \leq y$.  Since these operators are linear in $q$ the Lipschitz continuity is immediate.
\end{proof}

(2) \textbf{Resolvent estimates}.  Our construction of the resolvent is based on the estimate (see Lemma \ref{lemma:Volterra} and \eqref{volt})
\begin{equation}
\label{T0.volt}
(T_0 f)^*(x)  \leq \int_x^\infty \gamma_1(y) f^*(y) \, dy.
\end{equation}
which is an easy consequence of \eqref{K0}.

\begin{lemma}
\label{lemma:res.T0}
For each $\lam \in \bbR$ and $q \in H^{2,2}(\bbR)$, 
the operator $(I-T_0)^{-1}$ exists as a bounded operator 
from $C^0(\bbR^+)\otimes \bbC^2$ to itself. Moreover, 
$(I-T_0)^{-1}-I$ is an integral operator with continuous integral kernel $L_0(x,y,\lam)$, $L_0(x,y,\lam)=0$ for $x >y$. The integral kernel $L_0(x,y,\lam)$ satisfies 
the estimate
\begin{equation}
\label{L0}
\left| L_0(x,y,\lam) \right| \leq \exp\left( \norm{\gamma_1}{L^1}\right) \gamma_1(y)
\end{equation}
\end{lemma}

\begin{proof}
Because $T_0$ is a Volterra operator, we deduce from Lemma \ref{lemma:Volterra} that $(I-T_0)^{-1}$ exists as a bounded operator on $C^0(\bbR^+)\otimes \bbC^2$. We can obtain rather precise estimates on the resolvent through the Volterra series. The integral kernel $K_0(x,y,\lam)$ obeys the estimate
$\left| K_0(x,y,\lam) \right| \leq \gamma_1(y)$ where on the left, $\left| \dotarg \right|$ denotes the operator norm on $2 \times 2$ matrices. The operator 
$$ L_0 \equiv (I-T_0)^{-1} - I$$ 
is an integral operator with integral kernel 
$L_0(x,y,\lam)$ given by 
$$ L_0(x,y,\lam) = \begin{cases}
						 	\sum_{n=1}^\infty K_n(x,y,\lam), 	&	x \leq y \\
						 	\\
						 	0,												&  x > y
						\end{cases}
$$
where

\begin{multline*}
K_n(x,y,\lam) = \\
	\int_{x \leq y_1 \leq \dots \leq y_{n-1}}
		K_0(x,y_1,\lam) K_0(y_1,y_2,\lam) \ldots K_0(y_{n-1},y,\lam)
	\, dy_{n-1} \, \ldots \, dy_1
\end{multline*}

\noindent
and the estimate
$$ 
\left| K_n(x,y,\lam) \right| \leq \frac{1}{(n-1)!} \left( \int_x^\infty \gamma_1(t)   dt \right)^{n-1} \gamma_1(y)	
$$
holds. The estimate \eqref{L0} follows.
\end{proof}

Now suppose that $f \in C^0(\bbR^+,L^2({ I_0}))$ and let
\begin{equation}
\label{fLg.def}
g(x,\lam)= \int_x^\infty L_0(x,y,\lam) f(y,\lam) \, dy.
\end{equation}
Denote by $\widehat{L}_0$ the map $f \rarr g$. We will prove:

\begin{lemma}
\label{lemma:R0} The estimates
\begin{equation}
\label{R0.1}
\Norm{\widehat{L}_0}{\calB(C^0(\bbR^+,L^2({ I_0})))} \leq e^{\norm{\gamma_1}{L^1}} \norm{\gamma_1}{L^1}.
\end{equation}
and
\begin{equation}
\label{R0.2}
 \Norm{\widehat{L}_0}{\calB(L^2(\bbR^+ \times { I_0}))} 
 \leq e^{\norm{\gamma}{L^1}}\norm{\gamma_1}{L^{2,2}} 
\end{equation}
hold.
\end{lemma}

\begin{proof}
Suppose that $g\in C^0(\bbR^+,L^2({ I_0}))$. 
Then $f$ belongs to $C^0(\bbR^+,L^2({ I_0}))$ since 
$$ |g(x,\lam)| \leq e^{\norm{\gamma_1}{L^1}} \int_x^\infty  \gamma_1(y) |f(y,\lam)| \, dy $$
and we may conclude from Minkowski's integral equality that
$$
\norm{g(x,\dotarg)}{L^2({ I_0})} \leq e^{\norm{\gamma_1}{L^1}} \int_x^\infty \gamma_1(y) \norm{f}{C^0(\bbR^+,L^2({ I_0}))}  dy. 
$$
It follows that $L$ induces a bounded mapping $\widehat{L}_0$ from $ C^0(\bbR^+,L^2({ I_0}))$ to itself obeying the estimate
\eqref{R0.1}.

Similarly, suppose that $f \in L^2(\bbR^+ \times { I_0})$. Defining $g$ as in \eqref{fLg.def}, we estimate
$$
|g(x,\lam)| 
	\leq \left( \int_x^\infty |L_0(x,y,\lam)|^2 \, dy \right)^{1/2} \left(\int_x^\infty |g(y,\lam)|^2 \, dy \right)^{1/2}
$$
so that
\begin{align*}
\norm{g}{L^2(\bbR^+ \times { I_0})}^2
	&\leq 	\int_0^\infty \left( \sup_{\lam \in I} \int_x^\infty |L_0(x,y,\lam)|^2 \, dy\right) 
							\left( \int_I \left( \int_x^\infty |g(y',\lam)|^2 \, dy' \right) \, d\lam\right) \, dx\\[5pt]
	&\leq		\left( \int_0^\infty \int_x^\infty \gamma_1(y) e^{\norm{\gamma_1}{L^1}}\, dy \, dx\right)
				\norm{g}{L^2(\bbR^+ \times { I_0})}\\[5pt]
	&\leq		e^{\norm{\gamma_1}{L^1}} \norm{\gamma_1}{L^{2,2}} \norm{g}{L^2(\bbR^+ \times { I_0})}.
\end{align*}
so that the operator bound \eqref{R0.2} holds.
\end{proof}

\begin{remark}
\label{rem:R0}
As an immediate consequence of Lemma \ref{lemma:R0}, we see that $(I-T_0)^{-1}$ induces bounded operators 
on $\calB(C^0(\bbR^+,L^2({ I_0})))$ and $\calB(L^2(\bbR \times { I_0}))$ with respective norms bounded by
$ 1+\norm{\gamma_1}{L^1} \exp{\norm{\gamma_1}{L^1}}$
and
$ 1+\norm{\gamma_1}{L^1} \exp{\norm{\gamma_1}{L^{2,2}}}$.
\end{remark}

\textbf{(3) Solving for $\bfn$.} We can now use these resolvent estimates to solve \eqref{n.small}.

\begin{lemma}
\label{lemma:n}
Suppose that $q \in H^{2,2}(\bbR)$ and let $I_0 \subset \bbR$ be a bounded interval. There exists a unique solution of \eqref{n.small} for each $\lam \in 
{ I_0}$ so that 
$\bfn \in C^0(\bbR^+,L^2({ I_0})) \cap L^2(\bbR^+ \times { I_0})$. Moreover the map $q \rarr \bfn$ is Lipschitz continuous from $H^{2,2}(\bbR)$ to $C^0(\bbR^+,L^2({ I_0})) \cap L^2(\bbR^+ \times { I_0})$.
\end{lemma}

\begin{proof}
An immediate consequence of Lemma \ref{lemma:n0}, \eqref{n.est}, Lemma \ref{lemma:R0}, and Remark \ref{rem:R0}.
\end{proof}

\textbf{(4) Solving for $\bfn_\lam$.} Next, we estimate $\bfn_\lam$ by controlling $\bfn_1$ and solving \eqref{n.lam}.

\begin{lemma}
\label{lemma:n.lam}
Suppose that $q \in H^{2,2}(\bbR)$ and let $I_0 \subset \bbR$ be a bounded interval. There exists a unique solution of \eqref{n.lam} belonging to $\bfn \in C^0(\bbR^+,L^2({ I_0})) \cap L^2(\bbR^+ \times { I_0})$. Moreover, the map $q \rarr \bfn$ is Lipschitz continuous from $H^{2,2}(\bbR)$ to 
$C^0(\bbR^+,L^2({ I_0})) \cap L^2(\bbR^+ \times { I_0})$.
\end{lemma}

\begin{proof}
From Lemma \ref{lemma:n}, estimate \eqref{n.lam.est} of Lemma \ref{lemma:n0}, and Lemma 
\ref{lemma:T0.lam}(i) and (ii), we may conclude that 
$\bfn_1 \in C^0(\bbR^+,L^2({ I_0})) \cap L^2(\bbR^+ \times { I_0})$ 
and is Lipschitz continuous in $q$. We may then solve \eqref{n.lam} for 
$\bfn_\lam \in C^0(\bbR^+,L^2({ I_0})) \cap L^2(\bbR^+ \times { I_0})$ using Lemma \ref{lemma:R0} and 
Remark \ref{rem:R0}. The map $q \rarr \bfn_\lam$ is Lipschitz continuous from 
$H^{2,2}(\bbR)$ to $C^0(\bbR^+,L^2({ I_0})) \cap L^2(\bbR^+ \times { I_0})$ since $q \rarr \bfn_1$ has this continuity and the resolvents are Lipschitz continuous as operator-valued functions.
\end{proof}

\textbf{(5) Solving for $\bfn_{\lam\lam}$. } Finally, we control $\bfn_2$ and solve \eqref{n.lamlam} to estimate $\bfn_{\lam\lam}$.

\begin{lemma}
\label{lemma:n.lamlam}
Suppose that $q \in H^{2,2}(\bbR)$ and $I_0 \subset \bbR$ is a bounded interval. There exists a unique solution of \eqref{n.lamlam} in $C^0(\bbR^+,L^2({ I_0}))$. Moreover, the map $q \rarr \bfn_{\lam\lam}$ is Lipschitz continuous from $H^{2,2}(\bbR)$ to $C(\bbR^+,L^2({ I_0}))$.
\end{lemma}

\begin{proof}
From Lemma \ref{lemma:n0}, 
eq. \eqref{n.lamlam.est}, Lemma \ref{lemma:n}, Lemma \ref{lemma:n.lam}, 
and Lemma \ref{lemma:T0.lam}(ii), (iii), we deduce that 
$\bfn_2 \in C^0(\bbR^+;L^2({ I_0}))$ with $q \rarr \bfn_2$ Lipschitz as a map from 
$H^{2,2}(\bbR)$ to $C^0(\bbR^+,L^2({ I_0}))$. We now use Lemma \ref{lemma:R0} and Remark \ref{rem:R0} to solve for $\bfn_{\lam\lam}$ as before.
\end{proof}

\begin{proof}[Proof of Proposition \ref{prop:n.small}]
An immediate consequence of Lemmas \ref{lemma:n}, \ref{lemma:n.lam}, \ref{lemma:n.lamlam},
and the fact that the restriction map $f \rarr f(0)$ from $C^0(\bbR^+;L^2({ I_0}))$ to $L^2({ I_0})$ is continuous.
\end{proof}


\subsection{Large-$\lam$ Estimates}
\label{sec:direct.large}

In this subsection we prove Proposition \ref{prop:n.large}. 
To study $\eta_{11}^+$, $(\eta_{11}^+)_{\lam}$ and $(\eta_{11}^+)_{\lam\lam}$, we solve \eqref{eta11} 
and the derived equations
\begin{align}
\label{eta11.lam}
(\eta_{11}^+)_\lam &= (F_+)_\lam + (T_+)_\lam \left[ \eta_{11} \right] + T_+ \left[ (\eta_{11}^+)_\lam \right] \\[5pt]
\label{eta11.lamlam}
(\eta_{11}^+)_{\lam\lam} &= (F_+)_{\lam\lam}+ 2 (T_+)_\lam \left[ (\eta_{11})_\lam \right] 
						+ \left(T_+ \right)_{\lam\lam} \left[\eta_{11} \right]
						+ T_+ \left[  (\eta_{11}^+)_{\lam\lam} \right].
\end{align}
With good estimates in hand for $\eta_{11}^+$ and its derivatives, it will be a simple matter to prove the corresponding estimates on $\eta_{21}^+$ using \eqref{eta21}.

In the rest of this section, we will drop the $\pm$ and obtain estimates $\eta_{11}^+$ and $\eta^+_{21}$ since 
the analogous estimates for $\eta_{11}^-$ and $\eta_{21}^-$ are similar. We will write $\eta_{11}$ for $
\eta_{11}^+$, $F$ for $F_+$, $T$ for $T^+$, etc. 
{ We recall that $I_\infty = \{ \lam \in \bbR: |\lam| > 1\}$.}

Overall, we follow a strategy similar to that of section \ref{sec:direct.small} to study the scalar equation 
\eqref{eta11} { for $\eta_{11}$}, and then use these results to obtain comparable estimates on $\eta_{21}$. First, we will obtain 
estimates on $F$ and $G$ and derivatives of these functions in $\lam$ 
(Lemmas \ref{lemma:FG}, \ref{lemma:T.est},  \ref{lemma:T.cont}, and \ref{lemma:T.lam}). Second, we will obtain resolvent estimates for $(I-T)^{-1}$ by a method similar to
that used in the previous subsection (Lemmas \ref{lemma:T+} and \ref{lemma:R}). Third, we will solve 
\eqref{eta11} for $\eta_{11}$ (Lemma \ref{lemma:eta11}). Fourth, we'll solve 
\eqref{eta11.lam} for $\dee \eta_{11}/\dee \lam$ (Lemma \ref{lemma:eta11.lam}).
Fifth, we'll solve \eqref{eta11.lamlam} for 
$\dee^2 \eta_{11}/\dee \lam^2$ (Lemma \ref{lemma:eta11.lamlam}). 
Finally we will use \eqref{eta21} to obtain estimates on $\eta_{21}$ (Lemma \ref{lemma:eta21}). 

\bigskip

\textbf{(1) Estimates on $F$,  $G$, and $T$.}

\medskip

\begin{lemma}
\label{lemma:FG}
Suppose $q \in H^{2,2}(\bbR)$. The following define Lipschitz maps from $H^{2,2}(\bbR)$ into $C^0(\bbR^+, L^2(I_\infty)) \cap L^2(\bbR^+ \times I_\infty)$:
$$ 
\mathrm{(i)}~G, \quad
\mathrm{(ii) }~F, \quad
\mathrm{(iii) }~\frac{\dee G}{\dee \lam}, \quad
\mathrm{(iv) }~\frac{\dee F}{\dee \lam}. 
$$
The following define Lipschitz maps from $H^{2,2}(\bbR)$ into 
$C^0(\bbR^+, L^2(I_\infty))$:
$$ 
\mathrm{(v) }~\lam^{-1} \frac{\dee^2 G}{\dee \lam^2}, \quad
\mathrm{(vi) }~\lam^{-1} \frac{\dee^2 F}{\dee \lam^2}
$$
\end{lemma}

\begin{proof}
Observing that
\begin{align*}
\norm{F}{C^0(\bbR^+,L^2(I_\infty))} 
&\leq \norm{q}{L^1} \norm{G}{C^0(\bbR^+,L^2(I_\infty))}, 
\\
\norm{F}{L^2(\bbR^+I_\infty)}
&\leq
\norm{q}{L^{2,1/2}}\norm{G}{L^2(\bbR^+ \times I_\infty)}
\end{align*}
we see that (i) $\Rightarrow$ (ii). To prove (i) we pick $\varphi \in C_0^\infty(I_\infty)$ and mimic the proof of Lemma \ref{lemma:n0}.
The Lipschitz continuity follows from the fact that $G$ is linear in $q$ and $F$ is bilinear in $q$.

One can similarly check that (iii) $\Rightarrow$ (iv), so it suffices to prove (iii). We do so by mimicking the proof of Lemma \ref{lemma:n0} for the function
$$ \frac{\dee G}{\dee \lam} = -\frac{1}{2i} \int_x^\infty -i(y-x)e_\lam(y-x)q^\sharp(y) \, dy. $$ 

Finally, it is easy to see that (v) $\Rightarrow$ (vi). To prove (v), we recall
{ $q^\sharp = \qbar' - \dfrac{i}{2}|q|^2\qbar$} and split
$$ 
\frac{\dee^2 G}{\dee \lam^2}(x,\lam)  = 
	h_1(x,\lam)+h_2(x,\lam)
$$
where
\begin{align*}
h_1(x,\lam)	&=	2i \int_x^\infty (y-x)^2 e_\lam(y-x) { \overline{q'(y)}} \, dy\\
h_2(x,\lam)	&=	- \int_x^\infty (y-x)^2 e_\lam(y-x) { \overline{q(y)}}|q(y)|^2 \, dy
\end{align*}
We can estimate $h_2$ as before but for $h_1$ we integrate by parts to obtain
$$ 
h_1(x,\lam) = \int_x^\infty  \left(2i\lam (y-x)^2 + 2(y-x) \right) q(x)e_\lam(y-x) 
$$
We can now use previous techniques to bound $\lam^{-1} h_1(x,\lam)$
for $I_\infty$. 
\end{proof}

\medskip

The operator $T$ defined in \eqref{T} has the integral kernel
\begin{equation}
\label{K2}
 K^q_+(x,y,\lam) 
 	= 	\begin{cases}
 				\left( \dint_{\hspace{-1.5mm}x}^{\hspace{.5mm}y} 
 					e_\lam({ y-z})q(z)\, dz\right) q^\sharp(y), 	& x < y, \\
 				\\
 				0																&  x > y.
 								\end{cases}
\end{equation}

From this computation, we can prove:

\begin{lemma}
\label{lemma:T.est}
The Volterra estimate
\begin{equation}
\label{T.Volt2}
(T f){ ^*}(x) 
	\leq \left(
					{  \norm{q}{1} } \int_x^\infty |\qsharp(y)| \, dy
			\right) f{ ^*}(x)
\end{equation} holds. Moreover
\begin{align}
\label{T.L2toC0.2}
&	\sup_{x \in \bbR^+} \left |  \intpm |K^q_+(x,y,\lam)|^2 \, dy\right |^{1/2}
	\lesssim	\lam^{-1}\norm{q}{H^{2,2}} \left | \intpm |q^\sharp(y)|^2 \, dy \right |^{1/2}\\
\label{T.Fred2}
&\left( \int_{\bbR^+ \times \bbR^+} |K^q_+(x,y,\lam)|^2 \, dx \, dy \right)^{1/2}
	\lesssim \lam^{-1} \norm{q}{H^{2,2}} \left | \int_0^{+ \infty} |y||\qsharp(y)|^2 \, dy\right |^{1/2}\\
\label{T.C0toL2.2}
&\left( 
			\int_{\bbR^+} 
				\left( 
					\int_{\bbR^+} 
						\left| K^q_+(x,y,\lam) \right| 
					 dy 
				\right)^2 
			 dx 
\right)^{1/2}
	\lesssim	\norm{q}{L^{2,1}} \norm{q^\sharp}{L^{2,1}}.
\end{align}
\end{lemma}

We omit the proof.

\begin{remark}
\label{rem:T.bd}
It follows respectively from \eqref{T.Volt2}, \eqref{T.L2toC0.2}, \eqref{T.Fred2}, and \eqref{T.C0toL2.2} that $T$ is a bounded operator from $C^0(\bbR^+)$ to itself, from $L^2(\bbR^+)$ to $C^0(\bbR^+)$, and from $C^0(\bbR^+)$ to $L^2(\bbR^+)$. The map $q \mapsto T$
is bilinear and Lipschitz continuous from $H^{2,2}(\bbR)$ to the corresponding Banach spaces of bounded operators with constants uniform in $\lam \in I_\infty$.
\end{remark}

\begin{lemma}
\label{lemma:T.cont}
Let $\delta \in [0,1/2)$. Then, the continuity estimates
\begin{align}
\label{T.C0.lam}
\norm{T_{q,\lam_1} - T_{q,\lam_2}}{\calB(C^0)}
	&	\lesssim_\delta		|\lam_1-\lam_2|^\delta \norm{\qsharp}{L^{2,1}} \norm{q}{L^1},\\ 
\label{T.C0.q}
\norm{T_{q_1,\lam} - T_{q_2,\lam}}{\calB(C^0)}
	&	\lesssim	 	\norm{\qsharp_1-\qsharp_2}{L^1} \norm{q_1}{L^1}+
						\norm{\qsharp_2}{L^1}\norm{q_1-q_2}{L^1},\\ 
\label{T.L2.lam}
\norm{T_{q,\lam_1} - T_{q,\lam_2}}{\calB(L^2)}
	&	\lesssim   |\lam_1-\lam_2|^\delta 
								\left( \int_0^\infty |y|^{1+2\delta} |\qsharp(y)|^2 \, dy \right)^{1/2} \\
\label{T.L2.q}
\norm{T_{q_1,\lam} - T_{q_2,\lam}}{\calB(L^2)}
	&	\lesssim		\left | \int_0^{\pm \infty} |y| |\qsharp_1-\qsharp_2|^2 \, dy \right |^{1/2} 
						\norm{q_1}{L^1} \\
\nonumber
	&\qquad		+
						\left | \int_0^{\pm \infty} |y| |\qsharp_2|^2 \, dy \right |^{1/2} 
						\norm{q_1-q_2}{L^1}
\end{align}
hold,
where the implied constants in \eqref{T.C0.q} and \eqref{T.L2.q} are uniform in $\lam$ with $I_\infty$.
\end{lemma}

Finally we need mapping properties of the operators $T_\lam$ and $T_{\lam\lam}$.

\begin{lemma}
\label{lemma:T.lam}
The estimates
\begin{itemize}
\item[(i)]  $\norm{T_\lam}{\calB( L^2(\bbR^+ \times I_\infty)) } \leq 
			   \norm{q^\sharp}{L^{2,1}} \norm{q}{L^1}$,
\item[(ii)]  $\norm{T_\lam}{\calB(L^2(\bbR^+ \times I_\infty),C^0(\bbR^+,L^2(I_\infty)))} \leq \norm{q^\sharp}{L^{2,1}} \norm{q}{L^{2,2}}$,
\item[(iii)]  $\norm{\lam^{-1}T_{\lam\lam}\left[ h \right]}{C^0(\bbR^+,L^2(I_\infty))}$
				$\lesssim \norm{q}{H^{2,2}}^2 
				\left( 
						\norm{h}{L^2(\bbR^+ \times I_\infty)}+ 
						\norm{h_x}{L^2(\bbR^+ \times I_\infty)}
				\right)
				$.
\end{itemize}
hold.
\end{lemma}

\begin{proof}
(i), (ii)  From the formula
$$ 
\frac{\dee T}{\dee \lam} [h] (x,\lam) = \int_x^\infty q(y) \int_y^\infty (z-y) e_\lam(z-y) q^\sharp(z) h(z,\lam) \, dz \, dy.
$$
we may estimate 
$$
\left| \frac{\dee T}{\dee \lam}\left[h\right](x,\lam) \right| 
\leq 
\norm{q^\sharp}{L^{2,1}} \norm{h(\dotarg,\lam)}{L^2(\bbR^+)}  \int_x^\infty |q(y)| \, dy.
$$
We easily conclude that
\begin{align*}
\norm{\frac{\dee T}{\dee \lam} \left[ h \right]}{C^0(\bbR^+,L^2(I_\infty))}
	&	\leq 	\norm{q^\sharp}{L^{2,1}}\norm{q}{L^1} \norm{h}{L^2(\bbR^+ \times I_\infty)},\\
\norm{\frac{\dee T}{\dee \lam} \left[ h \right]}{L^2(\bbR^+ \times I_\infty)}
	& \lesssim	\norm{q^\sharp}{L^{2,1}} \norm{q}{L^{2,2}} \norm{h}{L^2(\bbR^+ \times I_\infty)}
\end{align*}
which imply (i) and (ii). The maps are Lipschitz since they are bilinear in $q$. 

(iii) From the formula
\begin{align*}
\frac{\dee^2 T}{\dee \lam^2} \left[ h \right](x,\lam)
&=	2i \int_x^\infty q(y) \int_y^\infty (z-y)^2 e_\lam(z-y) q^\sharp(z) h(z,\lam) \, dz \, dy \\
&=	I_1 + I_2
\end{align*}
where
\begin{align*}
I_1	&=	2i \int_x^\infty q(y) \int_y^\infty (z-y)^2 e_\lam(z-y) q'(z) h(z,\lam) \, dz \, dy, \\
I_2	&=	-  \int_x^\infty q(y) \int_y^\infty (z-y)^2 e_\lam(z-y) |q(z)|^2 h(z,\lam) \, dz \, dy.
\end{align*}
Since $z^2 |q(z)|^2 \in L^{2,1}$ for $q \in H^{2,2}(\bbR)$, we can estimate $I_2$ using the same techniques used for (i), (ii). The expression $I_1$ makes sense for $q \in \calS(\bbR)$ but we must integrate by parts to obtain an expression that is meaningful for arbitrary 
$q \in H^{2,2}(\bbR)$. We compute
\begin{align*}
I_2	
	&=		2i\int_x^\infty q(y)
						 \int_y^\infty 
						 		e_\lam(z-y)q(z) \left[ -2i\lam (z-y)^2 h(z,y) +  2(z-y)h(z,y)\right] 
						 \, dz 
					\, dy \\ 
	&\quad + 2i\int_x^\infty q(y)
						 \int_y^\infty 
						 		e_\lam(z-y)q(z) 	
						 			\left[(z-y)^2 h_z(z,\lam) \right] 
						 \, dz 
					\, dy.
\end{align*}
from which (iii) follows.
\end{proof}

\bigskip

\textbf{(2) Resolvent Estimates.} As before we exploit Volterra estimates to construct the resolvent,  obtain an integral kernel, and extend the resolvent to a bounded operator on the spaces $C^0(\bbR^+,L^2(I_\infty))$ and $L^2(\bbR^+ \times I_\infty)$.  

\begin{lemma}
\label{lemma:T+}
Suppose that $q \in H^{2,2}(\bbR)$. The resolvent $(I-T)^{-1}$ exists as a 
bounded operator in $C^0(\bbR^+)$ and the operator $L \equiv (I-T)^{-1}-I$
is an integral operator with integral kernel $L(x,y,\lam)$ so that $L(x,y,\lam)=0$ for $x>y$, $L(x,y,\lam)$ continuous in $(x,y,\lam)$ for $x<y$, and obeying the estimates
$$ 
|L(x,y,\lam)| \leq \exp\left(\norm{q}{L^1} \norm{q^\sharp}{L^1}\right)\norm{q}{L^1} |q^\sharp(y)|.
$$					
Moreover, the map $q \rarr \widehat{L}$ is Lipschitz continuous from $H^{2,2}(\bbR)$ into 
$$\calB(C^0(\bbR^+,L^2(I_\infty)) \cap \calB(L^2(\bbR^+ \times I_\infty)).$$
\end{lemma}

We omit the proof, which is very similar to the proof of Lemma \ref{lemma:res.T0}. The integral kernel $L$ defines an operator $\widehat{L}$ much as the integral kernel $L_0$ defined an operator $\widehat{L}_0$ in \eqref{fLg.def} and Lemma \ref{lemma:R0}. Following 
that analysis, one has:

\begin{lemma}
\label{lemma:R}
The estimates
\begin{equation}
\label{R.1}
\Norm{\widehat{L}}{\calB(C^0(\bbR^+,L^2(I_\infty)))}
	\leq	\norm{q}{L^1} \norm{q^\sharp}{L^1} 
			\exp\left(   	\norm{q}{L^1} \norm{q^\sharp}{L^1}  \right)
\end{equation}
and
\begin{align}
\label{R.2}
\Norm{\widehat{L}}{\calB(L^2(\bbR^+,L^2(I_\infty)))}
	&\leq 	\norm{q}{L^{2,1}}\norm{q^\sharp}{L^{2,1}} 
	+ \norm{q}{H^{2.2}}^6
					\exp\left(   	\norm{q}{L^1} \norm{q^\sharp}{L^1}  \right)
\end{align}
\end{lemma}

\begin{proof}
The estimate \eqref{R.2} follows from \eqref{R.1}, the formula
$$(I-T)^{-1} - I = T + T(I-T)^{-1} T,$$
and the bounds on $T:L^2 \rarr C^0$, $T:C^0 \rarr L^2$, and $T:L^2 \rarr L^2$ obtained in Lemma \ref{lemma:T.est} and Remark \ref{rem:T.bd}.
The estimate \eqref{R.1} follows from the Volterra estimate \eqref{T.Volt2} and the same argument used to in the proof of Lemma \ref{lemma:R0} to prove \eqref{R0.1}.
\end{proof}

\bigskip

\textbf{(3) Solving for $\eta_{11}$.}  From the resolvent construction above, we can solve for $\eta_{11}$.

\begin{lemma}
\label{lemma:eta11}
For each $q \in H^{2,2}(\bbR)$ and $\lam \in I_\infty$,
the equation \eqref{eta11} admits a unique solution $\eta_{11} \in C^0(\bbR^+,L^2(I_\infty)) \cap L^2(\bbR^+ \times I_\infty)$.
Moreover, $q \rarr \eta_{11}$ is Lipschitz continuous as a map 
from $H^{2,2}(\bbR)$ to  
$C^0(\bbR^+,L^2(I_\infty)) \cap L^2(\bbR^+ \times I_\infty)$.
\end{lemma}

\begin{proof}
A direct consequence of Lemma \ref{lemma:FG}(ii) and Lemma \ref{lemma:R}.
\end{proof}

\bigskip

\textbf{(4) Solving for $\dee \eta_{11}/\dee \lam$.} By controlling the inhomogeneous term in \eqref{eta11.lam}, we can estimate $(\eta_{11})_{\lam}$.

\begin{lemma}
\label{lemma:eta11.lam}
For each $q \in H^{2,2}(\bbR)$ and $\lam \in I_\infty$, 
the equation \eqref{eta11.lam} admits a unique solution 
$(\eta_{11})_\lam \in C^0(\bbR^+,L^2(I_\infty)) \cap L^2(\bbR^+ \times I_\infty)$. Moreover, $q \rarr (\eta_{11})_\lam$ is Lipschitz continuous as a map 
from $H^{2,2}(\bbR)$ to  
$C^0(\bbR^+,L^2(I_\infty)) \cap L^2(\bbR^+ \times I_\infty)$.
\end{lemma}

\begin{proof}
By \eqref{eta11.lam} and Lemma \ref{lemma:R}, it suffices to show that 
the inhomogeneous term
$$ F_\lam + T_\lam \left[ \eta_{11} \right] $$
belongs to $C^0(\bbR^+,L^2(I_\infty)) \cap L^2(\bbR^+ \times I_\infty)$. This follows from Lemma \ref{lemma:FG}(iv), Lemma \ref{lemma:T.lam}(i),(ii), and Lemma \ref{lemma:eta11}.
\end{proof}

\bigskip

\textbf{(5) Solving for $\dee^2 \eta_{11}/\dee \lam^2$.}  Next we obtain estimates on 
$(\eta_{11})_{\lam\lam}$ using \eqref{eta11.lamlam}.

\begin{lemma}
\label{lemma:eta11.lamlam}
For each $q \in H^{2,2}(\bbR)$, $\lam \in I_\infty$, 
equation \eqref{eta11.lamlam} admits a unique solution $(\eta_{11})_{\lam\lam}$ with
$\lam^{-1}(\eta_{11})_{\lam\lam} \in C^0(\bbR^+,L^2(I_\infty)$. Moreover, 
$q \rarr \lam^{-1}(\eta_{11})_{\lam\lam}$ is Lipschitz continuous as a map 
from $H^{2,2}(\bbR)$ to  
$C^0(\bbR^+,L^2(I_\infty))$.
\end{lemma}

\begin{proof}
By \eqref{eta11.lamlam} and Lemma \ref{lemma:R}, it suffices to show that the inhomogeneous term
$$ \lam^{-1} F_{\lam\lam} + 2\lam^{-1} T_\lam \left[(\eta_{11})_\lam\right] + \lam^{-1} T_{\lam\lam} \left[\eta_{11}\right]
$$
belongs to $C^0(\bbR^+,L^2(I_\infty)$. For the first term, this follows from Lemma \ref{lemma:FG}(vi), for the second term from Lemma \ref{lemma:T.lam}(i), (ii), Lemma \ref{eta11}, and for the third term from Lemma \ref{lemma:T.lam}(iii) and Lemmas \ref{lemma:eta11} and \ref{lemma:eta11.lam}.
\end{proof}

\bigskip

\textbf{(6) Estimates on $\eta_{21}$ and its derivatives.} It is now a simple matter to use \eqref{eta21}, estimates on $G$, and results already proved for $\eta_{11}$ to obtain Lipschitz continuity of $\eta_{21}$ and its derivatives.

\begin{lemma}	
\label{lemma:eta21}
The maps $q \rarr \eta_{21}$, $q \rarr (\eta_{21})_\lam$, and
$q \rarr \lam^{-1} (\eta_{21})_{\lam\lam}$ are Lipschitz continuous 
from $H^{2,2}(\bbR)$ to $C^0(\bbR^+,L^2(I_\infty))$
\end{lemma}

\begin{proof}
Referring to \eqref{eta21} and dropping the $\pm$ signs, the term $G(x,\lam)$ has the required properties by Lemma \ref{lemma:FG}(i), (iii), (v), and the second right-hand term of \eqref{eta21} has the required properties since $q$ is bounded and $\eta_{11}$ has the correct mapping properties by Lemmas \ref{lemma:eta11}, \ref{lemma:eta11.lam}, and \ref{lemma:eta11.lamlam}.
Thus, it remains to analyze the third term of \eqref{eta21} (dropping the $+$ sign on $\eta_{11}$)
$$
W(x,\lam) = \frac{1}{2i} \int_x^\infty e_\lam(y-x)  \qsharp(y) \eta_{11}(y,\lam) \, dy.	
$$
It is easy to see that
\begin{align*}
\norm{W}{C^0(\bbR^+,L^2(I_\infty))}
	&\leq \frac{1}{2} 
				\norm{q^\sharp}{L^2} 
				\norm{\eta_{11}}{L^2(\bbR^+ \times I_\infty)}\\
\norm{W_\lam}{C^0(\bbR^+,L^2(I_\infty))}
	&\leq	\frac{1}{2}
				\left(
					\norm{q^\sharp}{L^{2,1}} 
					\norm{\eta_{11}}{L^2(\bbR^+ \times I_\infty)} 
				+
					\norm{q^\sharp}{L^2}
					\norm{(\eta_{11})_\lam}
							{L^2(\bbR^+ \times I_\infty)}
				\right)		
\end{align*}
which shows that $q \rarr W$ and $q \rarr W_\lam$ have the required 
properties.  

To analyze $W_{\lam\lam}$, recall \eqref{qsharp} to write
$ W=W_1+W_2$ where
\begin{align*}
W_1(x,\lam)	
	&-=	\frac{1}{2i}
				\int_x^\infty e_\lam(y-x) \overline{q'(y)} \eta_{11}(y,\lam) \, dy,\\
W_2(x,\lam)
	&-=	-\int_x^\infty e_\lam(y-x) \overline{q(y)} |q(y)|^2 \eta_{11}(y,\lam) \, dy.
\end{align*}

We first control $W_1$. Differentiating in $\lam$ we have
$$
(W_1)_{\lam\lam}(x,\lam) = W_{11}(x,\lam) + W_{12}(x,\lam)
$$
where
\begin{align*}
W_{11}(x,\lam)	&=			-2i\int_x^\infty e_\lam(y-x)(y-x)^2 \overline{q'(y)} \eta_{11}(y,\lam), \, dy\\
W_{12}(x,\lam)	&=			-\int_x^\infty e_\lam(y-x) \overline{q'(y)} (\eta_{11})_{\lam\lam}(y,\lam) \, dy
\end{align*}
It is easy to see that
$$
\sup_{x \geq 0} \norm{(\dotarg)^{-1}W_{12}(x,\dotarg)}{L^2(I_\infty)} \leq \norm{q'}{L^2} \norm{(\diamond)^{-1}(\eta_{11})_\lam(\dotarg,\diamond)}{L^2(\bbR^+ \times I_\infty)}
$$
so that $q \rarr \lam^{-1} (W_{11})_{\lam\lam}$ has the correct mapping property. Turning to $W_{11}$, we integrate by parts to remove the derivative on $q$ and obtain
\begin{align}
\label{W11}
W_{11}(x,\lam)	&=	-2i \int_x^\infty e_\lam(y-x) \overline{q(y)}
											\left( 
												2(y-x) \eta_{11}(x,\lam) + (y-x)^2(\eta_{11})_x(y,\lam) 
											\right)
									\, dy \\
\nonumber
						&\quad	+ 4\lam \int_x^\infty e_\lam(y-x)(y-x)^2 \, \overline{q(y)} \eta_{11}(y,\lam) \, dy.
\end{align}
The first right-hand term in \eqref{W11} has $C^0(\bbR^+,L^2(I_\infty))$-norm bounded by 
\begin{equation}
\label{W11.1}
\norm{q}{L^{2,1}} \norm{\eta_{11}}{L^2(\bbR^+ \times I_\infty)}.
\end{equation}
Since, by \eqref{n11.pre} and \eqref{eta},
$$
(\eta_{11})_x 	= (n_{11})_x 
					= q(y)\left( \eta_{21}(x,\lam) - \frac{1}{2i}\overline{q(x)}   \right) 
							 -\frac{i}{2}|q(y)|^2 \left(1+\eta_{11}(x,\lam) \right), 
$$  
we can reexpress the second right-hand term in \eqref{W11} as
\begin{multline*}
-2i \int_x^\infty e_\lam(y-x) (y-x)^2 \left( |q(y)|^2(-\frac{1}{2i}\overline{q(y)} + \eta_{21}(y,\lam) \right) \, dy\\
	-2i \int_x^\infty e_\lam(y-x) (y-x)^2 \left( -\frac{i}{2} \overline{q(y)}|q(y)|^2(1+\eta_{11}(y,\lam) \right) \, dy
\end{multline*}
which has $C^0(\bbR^+,L^2(I_\infty))$-norm bounded by constants times
\begin{equation}
\label{W11.2}
\norm{q}{H^{2,2}}^3
	\left(
		1+ \norm{\eta_{11}}{L^2(\bbR^+ \times I_\infty)}  + \norm{\eta_{21}}{L^2(\bbR^+ \times I_\infty)} 
	\right).
\end{equation}
Finally, dividing by $\lam$ in the third term, we can estimate the $C^0(\bbR^+, L^2(I_\infty))$ norm of the quotient by 
\begin{equation}
\label{W11.3}
\norm{q}{L^{2,2}}\norm{\eta_{11}}{L^2(\bbR^+ \times I_\infty)}.
\end{equation}
Combining \eqref{W11.1}, \eqref{W11.2}, and \eqref{W11.3}, we see that
\begin{align}
\label{W1}
\norm{(\diamond)^{-1}W_{11}(\dotarg,\diamond)}{C^0(\bbR^+,I_\infty)} 
		&\lesssim \left(1+\norm{q}{H^{2,2}}^3 \right) \\
\nonumber
		&\quad
					\times  \left(
		1+ \norm{\eta_{11}}{L^2(\bbR^+ \times I_\infty)}  + \norm{\eta_{21}}{L^2(\bbR^+ \times I_\infty)} 
	\right),
\end{align}
which shows that $W_1$ has the required mapping property. 

Now we turn to $W_2$. Since
\begin{align*}
(W_2)_{\lam\lam}(x,\lam)	
	&=	-\int_x^\infty e_\lam(y-x)|q(y)|^2 
					\bigl([
							(\eta_{11})_{\lam\lam}(y,\lam)
					\bigr. \\
\nonumber
	&\qquad \qquad
					\bigl.
							 -2i(y-x) (\eta_{11})_\lam(y,\lam)
							-4(y-x)^2 \eta_{11}(y,\lam)
					\bigr]
				\, dy,
\end{align*}
we may estimate
\begin{align}
\label{W2}
\norm{(\diamond)^{-1} (W_2)_{\lam\lam}(\dotarg,\diamond)}{C^0(\bbR^+, L^2(I_\infty))}
	&	\leq	\norm{q}{H^{2,2}}^2  
						\Bigl( 
							\norm{\lam^{-1}(\eta_{11})_\lam}{C^0(\bbR^+,L^2(I_\infty ))} 
						\Bigr.\\
\nonumber
	&  \qquad \qquad
								+  \norm{(\eta_{11})_\lam}{C^0(\bbR^+,L^2(I_\infty ))} 
						\\
\nonumber
	&	\qquad	\qquad
						\Bigl.
								+  \norm{\eta_{11}}{C^0(\bbR^+,L^2(I_\infty))} 
						\Bigr).		
\end{align}
This shows that $q \rarr W_2$ has the correct mapping properties.

Combining \eqref{W1} and \eqref{W2}, we conclude that $q \mapsto W_{\lam\lam}$ has the correct mapping property, and hence, also, $q \mapsto (\diamond)^{-1} \eta_{\lam\lam}(\dotarg,\diamond)$. 
\end{proof}

\bigskip

\begin{proof}[Proof of Proposition \ref{prop:n.large}]
An immediate consequence of Lemmas \ref{lemma:eta11}, \ref{lemma:eta11.lam}, \ref{lemma:eta11.lamlam}, 
and \ref{lemma:eta21} and the fact that the restriction map $f \rarr f(0)$ from 
$C^0(\bbR^+;L^2(I_\infty))$ to $L^2(I_\infty)$ is continuous.
\end{proof}

\section{Beals-Coifman Solutions}
\label{sec:BC}

In this section we construct the Beals-Coifman solutions for \eqref{LS.red}. It follows from \eqref{LS.normal} and the discussion in the Introduction that 
the scattering data is given by 
\begin{equation}
\label{m.scatt}
m^+(x,\zeta) 
= m^-(x,\zeta) e^{-ix\zeta^2 \ad(\sigma)} \twomat{a(\zeta)}{\bb(\zeta)}{b(\zeta)}{\ba(\zeta)}
\end{equation} 
where the symmetries \eqref{ab.sym} hold. In order to elucidate properties of the 
scattering data we recall the integral equations for $m^\pm$.

Assuming that $q \in L^1 \cap L^2$ (so that both $Q$ and $P$ are $L^1$ matrix-valued functions), the Jost solutions $m^\pm$ are solutions
of the integral equations
\begin{align*}
m^+(x,\zeta)	&=	I
	-	\int_x^\infty e^{-i(x-y)\zeta^2 \ad(\sigma)} 
		\left[\left(\zeta Q(y) + P(y)\right) m(y,\zeta)\right] \, dy\\[5pt]
m^-(x,\zeta)	&=	I
	+	\int_{-\infty}^x  e^{-i(x-y)\zeta^2 \ad(\sigma)} 
		\left[\left(\zeta Q(y) + P(y)\right) m(y,\zeta)\right] \, dy
\end{align*}
with $\det m^+(x) = \det m^-(x) =1$. 

Observe that
\medskip
\begin{align}
\label{m1+}
\twovec{m_{11}^+(x,\zeta)}{m_{21}^+(x,\zeta)}
	&=	\twovec{1}{0}
			-\int_x^\infty 
				\twovec{\zeta qm _{21}^+ + p_1 m_{11}^+}
							{e^{2i\zeta^2 (x-y)}
								\left[ 
									\overline{q}m_{11}^+ + p_2m_{21}^+ 
								\right]}				
				\, dy	
\\[10pt]
\label{m2+}
\twovec{m_{12}^+(x,\zeta)}{m_{22}^+(x,\zeta)}
	&=	\twovec{0}{1}
			-\int_x^\infty
				\twovec{e^{-2i\zeta^2(x-y)}
								\left[\zeta q m_{22^+}+ p_1 m_{12}^+ \right]}
							{\zeta q m_{12}^+ + p_2 m_{22}^+}
					\, dy					
\\[10pt]
\label{m1-}
\twovec{m_{11}^-(x,\zeta)}{m_{21}^-(x,\zeta)}
	&=	\twovec{1}{0}
			+\int_{-\infty}^x
				\twovec{\zeta qm _{21}^- + p_1 m_{11}^-}
							{e^{2i\zeta^2 (x-y)}
								\left[ 
									\overline{q}m_{11}^- + p_2m_{21}^- 
								\right]}				
				\, dy	
\\[10pt]
\twovec{m_{12}^-(x,\zeta)}{m_{22}^-(x,\zeta)}
	&=	\twovec{0}{1}
			+ \int_{-\infty}^x
				\twovec{e^{-2i\zeta^2(x-y)}
					\left[\zeta q m_{22}^-+ p_1 m_{12}^- \right]
					}
					{\zeta q m_{12}^- + p_2 m_{22}^-}
				\, dy
\label{m2-}
\end{align}

Using the fact that $\det m^+ = \det m^- = 1$, it is easy to deduce that 
\begin{align}
\label{a.Wronski}
a(\zeta) 		&=
		\twodet{m_{11}^+}{m_{12}^-}{m_{21}^+}{m_{22}^-} =
		W\left(m_{(1)}^-, m_{(2)}^+ \right)\\
\nonumber\\
\label{ba.Wronski}
\ba(\zeta) &= 
		\twodet{m_{11}^-}{m_{12}^+}{m_{21}^-}{m_{22}^+}  =
		W\left(m^+_{(1)}, m_{(2)}^-\right)
\end{align}

\medskip

It follows from \eqref{m.scatt}, the first line of \eqref{m1+}, and 
the second line of \eqref{m2+} that
\begin{align}
\label{a.int.r}
a(\zeta)	
	&=	1-\int_{-\infty}^\infty \left(\zeta q m_{21}^+ +  p_1 m_{11}^+\right)\, dy,\\
\label{b.int.r}
\ba(\zeta)
	&=	1-\int_{-\infty}^\infty \left(\zeta q m_{12}^+ + p_2 m_{22}^+ \right) \, dy.\\
\end{align}
Using \eqref{scatt.det}, \eqref{m.scatt}, the first line of \eqref{m1-}, and the second line of \eqref{m2-}, we also have
\begin{align}
\label{a.int.l} 
a(\zeta)
	&=	1+\int_{-\infty}^\infty \left(\zeta q m_{12}^- + p_1 m_{22}^-\right) \, dy
\\
\ba(\zeta)
	&=	1+\int_{-\infty}^\infty \left(\zeta q m_{21}^- + p_1 m_{11}^-\right) \, dy
\label{b.int.l}
\end{align}

\medskip

Write $m^+_{(1)} = (m_{11}^+, m_{21}^+)^T$, 
$m^+_{(2)} = (m_{12}^+, m_{22}^+)^T$, and similarly for
$m^-_{(1)}$ and $m^-_{(2)}$.
An easy argument with Volterra estimates, paying attention to 
the modulus of $\exp(\pm 2i\zeta^2 (x-y))$,  shows that

\medskip

\begin{itemize}
\item	$m^+_{(1)}$ has a bounded analytic continuation to $\Im(\zeta^2) < 0$
\item	$m^+_{(2)}$ has a bounded analytic continuation to $\Im(\zeta^2) > 0$
\item $m^-_{(1)}$ has a bounded analytic continuation to $\Im(\zeta^2) >0 $
\item $m^-_{(2)}$ has a bounded  analytic continuation to $\Im(\zeta^2) < 0$
\end{itemize}
\medskip

It follows from these observations, \eqref{a.Wronski}, and \eqref{ba.Wronski}
that

\medskip

\begin{itemize}
\item	$a(\zeta)$ has an analytic extension to $\Im(\zeta^2) < 0$
\item	$\ba(\zeta)$ has an analytic extension to $\Im(\zeta^2)>0$
\end{itemize}

To construct the Beals-Coifman solutions, we will need the asymptotic behavior of 
$m_{(1)}^\pm$ as $x \rarr \mp \infty$ and  $m_{(2)}^\pm$ as $x \rarr \mp \infty$. An argument with the dominated convergence theorem, exploiting the decay of the 
exponential $\exp(\pm i (x-y)\zeta^2 \ad(\sigma))$, shows that
\begin{align*}
\lim_{x \rarr -\infty} m_{21}^+(x,\zeta) 	&=	0, \quad \Im(\zeta^2) < 0\\
\lim_{x \rarr -\infty} m_{12}^+(x,\zeta)	&=	0, \quad \Im(\zeta^2) > 0\\
\lim_{x \rarr +\infty} m_{12}^-(x,\zeta)	&=	0, \quad \Im(\zeta^2) < 0\\
\lim_{x \rarr +\infty} m_{21}^-(x,\zeta)	&=	0, \quad \Im(\zeta^2) >0.
\end{align*}
It follows from these relations, the integral equations \eqref{m1+}, \eqref{m2+}, 
\eqref{m1-}, \eqref{m2-}, and the integral identities \eqref{a.int.r}, \eqref{b.int.r}, \eqref{a.int.l}, 
and \eqref{b.int.l} that
\begin{align}
\label{m1+.lim}
\lim_{x \rarr -\infty} 	m_{(1)}^+(x,\zeta) 	&=	\twovec{a(\zeta)}{0}	\quad	\Im(\zeta^2) <0 \\[5pt]
\label{m2+.lim}
\lim_{x \rarr -\infty}	m_{(2)}^+(x,\zeta)	&=	\twovec{0}{\ba(\zeta)}	\quad	\Im(\zeta^2)>0\\[5pt]
\label{m1-.lim}
\lim_{x \rarr +\infty} m_{(1)}^-(x,\zeta)	&=	\twovec{\ba(\zeta)}{0}	\quad	\Im(\zeta^2)>0\\[5pt]
\label{m2-.lim}
\lim_{x \rarr +\infty}	m_{(2)}^-(x,\zeta)	&=	\twovec{0}{a(\zeta)}	\quad	\Im(\zeta^2)<0.
\end{align}

\subsection{Construction of Beals-Coifman Solutions}

We now define the right-hand Beals-Coifman solutions by 
\begin{align*}
M_r(x,\zeta)	&=	\begin{cases}
								\left( 
									\dfrac{m^-_{(1)}(x,\zeta)}
											{\ba(\zeta)}, \, 
											m^+_{(2)}(x,\zeta)
								\right),
								&		\Im(\zeta^2) > 0,\\
								\\
								\left(
										m^+_{(1)}(x,\zeta), \, 
										\dfrac{m^-_{(2)}(x,\zeta)}{a(\zeta)}
								\right),
								&		\Im(\zeta^2) < 0.
							\end{cases}
\end{align*}
These solutions are piecewise analytic, and bounded as $x \rarr -\infty$ by the boundedness of the normalized Jost solutions and the functions $a(\zeta)$ and $\ba(\zeta)$ (so long as $a(\zeta)$ and 
$\ba(\zeta)$ have no zeros). By \eqref{m1-.lim} and \eqref{m2+.lim}, they are 
normalized so that
$$\lim_{x \rarr \infty} M_r(x,\zeta) = I, \quad \Im \zeta^2 \neq 0,$$
and are bounded as $x \rarr -\infty$.

Similarly, the left-hand Beals-Coifman solutions, given by
\begin{align*}
M_\ell(x,\zeta)	&=	\begin{cases}
									\left(
										m_{(1)}^-(x,\zeta), \, 
										\dfrac{m_{(2)}^+(x,\zeta)}{\ba(\zeta)}
									\right),
									&	\Im(\zeta^2) > 0,
									\\
									\\
									\left(
										\dfrac{m_{(1)}^+(x,\zeta)}{a(\zeta)}, \, 
											m^-_{(2)}(x,\zeta)
									\right),
									&	\Im(\zeta^2) < 0
								\end{cases}
\end{align*}
are piecewise analytic, bounded as $x\rarr +\infty$, and normalized so that 
$$\lim_{x \rarr -\infty} M_\ell(x,\zeta) = I, \quad \Im \zeta^2 \neq 0.$$

Both $M_r$ and $M_\ell$ have boundary values as $\pm \Im \zeta^2 \darr 0$. We denote these respectively by $M_r^\pm$ and $M_\ell^\pm$. We now compute the jump relations for these boundary values. In what follows, we write 
$$f(x) \underset{x \rarr \pm \infty}{\sim} g(x)$$ 
if $\lim_{x \rarr \pm \infty} |f(x)-g(x)|=0$.

From \eqref{m.scatt} 
it is easy to see that, for $\Im \zeta^2 =0$, 
\begin{align*}
m^+_{(1)}(x,\zeta)	
	&\underset{x \rarr -\infty}{\sim} e^{-ix\zeta^2 \ad(\sigma)} \twovec{a(\zeta)}{b(\zeta)}\\[5pt]	
m^+_{(2)}(x,\zeta)	
	&\underset{x \rarr -\infty}{\sim} e^{-ix\zeta^2 \ad(\sigma)} \twovec{\bb(\zeta)}{\ba(\zeta)}\\[5pt]
m^-_{(1)}(x,\zeta)	
	&\underset{x \rarr +\infty}{\sim} e^{-ix\zeta^2 \ad(\sigma)} \twovec{\ba(\zeta)}{-b(\zeta)}\\[5pt]
m^-_{(2)}(x,\zeta)	&\underset{x \rarr +\infty}{\sim} e^{-ix\zeta^2 \ad(\sigma)} \twovec{-\bb(\zeta)}{a(\zeta)}
\end{align*}
It follows from these relations that 
\begin{align}
\label{Ml.+}
M^+_\ell(x,\zeta) 	
	&\underset{x \rarr +\infty}{\sim}	 e^{-ix\zeta^2 \ad(\sigma)} 
		\Twomat{\ba(\zeta)}{0}{-b(\zeta)}{\dfrac{1}{\ba(\zeta)}}
		\\[5pt]
\label{Ml.-}
M^-_\ell(x,\zeta)
	&\underset{x \rarr +\infty}{\sim} e^{-ix\zeta^2 \ad(\sigma)}
		\Twomat{\dfrac{1}{a(\zeta)}}{-\bb(\zeta)}{0}{a(\zeta)} 
\end{align}
and
\begin{align}
\label{Mr.+}
M^+_r(x,\zeta)	
	&\underset{x \rarr -\infty}{\sim}	e^{-ix\zeta^2 \ad(\sigma)} 
		\Twomat{\dfrac{1}{\ba(\zeta)}}{\bb(\zeta)}{0}{\ba(\zeta)}\\[5pt]
\label{Mr.-}
M^-_r(x,\zeta)
	&\underset{x \rarr -\infty}{\sim} e^{-ix\zeta^2 \ad(\sigma)}
		\Twomat{a(\zeta)}{0}{b(\zeta)}{\dfrac{1}{a(\zeta)}}
\end{align}
From \eqref{Ml.+} and \eqref{Ml.-}, we compute that
\begin{equation}
\label{Ml.jump}
M^+_\ell = M^-_\ell e^{-ix\zeta^2 \ad(\sigma)} v_\ell, \quad
v_\ell =\Twomat{1}{\bb/ \ba}{-b/a}{1- b \bb /a  \ba}
\end{equation}
while, from \eqref{Mr.+} and \eqref{Mr.-}, we see that
\begin{equation}
\label{Mr.jump}
M^+_r = M^-_r e^{-ix\zeta^2 \ad(\sigma)} v_r, \quad
v_r=\Twomat{1-b\bb/ a \ba}{\bb/a}{-b/\ba}{1}
\end{equation}

\subsection{The Riemann-Hilbert Problems in the $\lam$ Variables}

We can recast the left and right RHP's \eqref{Ml.jump} and \eqref{Mr.jump} in terms of the dependent variables $N^\pm$ and the spectral variable $z = \zeta^2$. The new RHP is an RHP with contour $\bbR$. 
Applying the automorphism 
$$ \varphi: \twomat{a}{b}{c}{d} \mapsto \twomat{a}{\zeta^{-1}b}{\zeta c}{d} $$
to the Jost and Beals-Coifman solutions and exploiting the odd symmetry of off-diagonal components, and even symmetry of diagonal components, with respect to the reflection $\zeta \mapsto -\zeta$, we may define first
$$n^\pm(x,\lam) =  \varphi(m^\pm(x,\zeta)), \quad \lam=\zeta^2$$
and then
$$N_*(x,\lam) = \varphi(N_*(x,\zeta)), \quad \lam=\zeta^2$$
where $*=\ell,r$. Set
$$\alpha(\lam) = a(\zeta), 
\quad
\beta(\lam) = \zeta^{-1} b(\zeta), 
\quad
\balpha(\lam) = \ba(\zeta), 
\quad
\bbeta(\lam) = \zeta^{-1} \bb(\zeta).
$$
From the symmetries \eqref{ab.sym}, it follows that
\begin{equation}
\label{alphabeta.sym}
\balpha(\lam) = \overline{\alpha(\lam)}, \quad
\bbeta(\lam) = \overline{\beta(\lam)}.
\end{equation}

We can now compute jump relations for the pairs $(N_\ell^+, N_\ell^-)$ and $(N_r^+, N_r^-)$.
It follows from \eqref{Ml.jump}, \eqref{Mr.jump}, and the definitions above that
\begin{align}
\label{Nl.jump.pre}
N_\ell^+ 	&=		N_\ell^-	\Twomat{1}
													{\bbeta/\balpha}
													{-\lam\beta/\alpha}
													{1-\lam\beta\bbeta/\alpha \balpha}\\
\nonumber \\
\label{Nr.jump.pre}
N_r^+		&=		N_r^-		\Twomat{1-\lam \beta\bbeta/\alpha\balpha}
													{\bbeta/\alpha}
													{-\lam \beta/\balpha}
													{1}
\end{align}
Setting $\brho=\bbeta/\balpha$  and $\rho=\bbeta/\alpha$ respectively in 
\eqref{Nl.jump.pre} and \eqref{Nr.jump.pre} and using the symmetries \eqref{alphabeta.sym}, we 
conclude that
\begin{align}
\label{Nl.jump}
N_\ell^+ 	&=		N_\ell^-	\Twomat{1}
													{\brho}
													{-\lam\overline{\brho}}
													{1-\lam|\brho|^2}\\
\nonumber \\
\label{Nr.jump}
N_r^+		&=		N_r^-		\Twomat{1-\lam |\rho|^2}
													{\rho}
													{-\lam \overline{\rho}}
													{1}
\end{align}

\section{The Riemann-Hilbert Problem for Schwartz Class Data}
\label{sec:RH.Schwartz}

%
%

In this section we study the RHP that defines the inverse scattering map. We only discuss the `right'  RHP problem since the discussion for the `left' RHP is similar. 

We begin by formulating precisely the RHPs  $(\Sigma,v)$ (see Problem \ref{prob:RH.M})  and $(\bbR,J)$ (see Problem \ref{prob:RH.n}). Next, we prove that these two problems are equivalent under suitable hypotheses on the respective data $v$ and $J$. We then prove that the RHP $(\Sigma,v)$ has a unique solution for suitable $a$, $b$, $\ba$, $\bb$.
We use these facts to show that the Beals-Coifman integral equation associated to the RHP for $(\bbR,J)$ has a unique solution provided that $\rho \in H^{2,2}(\bbR)$ and $1-\lam|\rho(\lam)|^2 >0 $ 
strictly. Finally, we show that the solution $M_\pm$ of Problem \ref{prob:RH.M} obeys \eqref{LS.red} as a function of $x$, and obtain reconstruction formulas for $Q(x)$ and $P(x)$ in terms of the solution $\mu$ of the Beals-Coifman integral equation for Problem \ref{prob:RH.M}.
Using the equivalence of Problems \ref{prob:RH.M} and \ref{prob:RH.n}, we obtain the reconstruction formula \eqref{q.recon} that will be used in the next section to analyze the inverse scattering map.

\subsection{Two RH Problems}
\label{sec:RH}

We formulate precisely the RHP problems respectively on the contour $\Sigma$ and on the contour $\bbR$ for Schwartz class data. In the next subsection we prove their equivalence.

\begin{problem}
\label{prob:RH.M}
Find a matrix-valued function $M(x,z)$, analytic for $z \in \bbC \setminus \Sigma$, satisfying
\begin{align*}
M_+(x,\zeta) 					&= 	M_-(x,\zeta) e^{-ix\zeta^2 \ad (\sigma)} v(\zeta), 
									\\[5pt]
 v(\zeta) 						&= 	\Twomat{1-b(\zeta) \bb(\zeta)/ a(\zeta) \ba(\zeta)}
											{\bb(\zeta)/a(\zeta)}
											{-b(\zeta)/\ba(\zeta)}
											{1}\\[5pt]
M_\pm(x,\dotarg) - \One &\in 	\dee C_\Sigma (L^2).
\end{align*}
We assume the following properties of the given data $a$, $b$, $\ba$, $\bb$:
\begin{enumerate}
\item[(i)]  $b \in \calS(\Sigma)$, $b(-\zeta)= -b(\zeta)$
\item[(ii)] $a \in C^\infty$, $a(-\zeta) = a(\zeta)$
\item[(iii)] $a-1$ has a complete asymptotic expansion in $\zeta^{-2}$ for large $\zeta$ with $a(\zeta) = 1 + \mathcal{O}(\zeta^{-2})$
\item[(iv)] (analytic continuation) $a$ and $\ba$ satisfy the condition
$$
(\ba^{-1}-1,a -1) \in \dee C_\Sigma(L^2).
$$
\item[(v)] (symmetry reduction) 
$\ba(\zeta)=\overline{a(\zetabar)}$, $ \bb(\zeta)=\overline{b(\zetabar)}$
\item[(vi)] (determinant) $a(\zeta)\ba(\zeta) - b(\zeta) \bb(\zeta) =1$
\item[(vii)] (spectral condition) $\inf_{\zeta \in \Sigma} |a(\zeta)| = c >0$
\end{enumerate}
\end{problem}

We recall that the meaning of $\dee C_\Sigma(L^2)$
is given after the statement of Problem \ref{problem-RHP}.

\begin{problem}
\label{prob:RH.n}
Find a row vector-valued function $\bfN(x,z)$, analytic for $z \in \bbC \setminus \bbR$, satisfying
\begin{align*}
\bfN_+(x,\lam)					&=	\bfN_-(x,\lam) e^{-i\lam x \ad (\sigma)} J(\lam), 
											\\[5pt]
J(\lam)							&=	\Twomat{1-\lam |\rho(\lam)|^2}{\rho(\lam)}
															{-\lam \overline{\rho(\lam)}}{1}\\[5pt]
\bfN_\pm(x,\dotarg) - \bfe_1	&\in 	\dee C_\bbR(L^2)
\end{align*}
We assume the following properties of the data $\rho$:
\begin{enumerate}
\item[(i)] 	$\rho \in \calS(\bbR)$ 
\item[(ii)]  	$1 - \lam |\rho(\lam)|^2 > 0$ strictly. 
\end{enumerate}
\end{problem}

\begin{remark}
\label{rem:ab.rh}
Conditions (i)-(vi) in Problem \ref{prob:RH.M} are motivated by well-known constraints on the scattering data for the direct problem: see \eqref{ab.sym} in the introduction for (i), (ii), (v). The regularity conditions in (i), (ii), the asymptotic expansion in (iii),  and the analytic continuation properties of $a$ and $\ba$ follow from a careful study of the Volterra integral equations that define the normalized Jost solutions.
Condition (vii) in Problem \ref{prob:RH.M}  is a spectral assumption on the initial data that rules out algebraic soliton solutions. Condition (iv) means that there is a function $h \in L^2(\Sigma)$ so that
$\ba^{-1} -1 = C^+h$ and $a-1 = C^- h$. Conditions (iv) and (vi) imply that the functions $a_+^\sharp = 1/\ba$ and $a_-^\sharp = a$
are related by
$$ a_+^\sharp(\zeta)=a_-^\sharp(\zeta) \left(1-\frac{b(\zeta)\bb(\zeta)}{a(\zeta)\ba(\zeta)} \right). $$
This observation motivates the RHP \eqref{rho.to.ab.RH} used below to recover $a$ and $\ba$ from $\rho$.
Note that condition (iv) rules out bright soliton solutions by insuring that $\ba(\zeta)$ (and hence, also $a(\zeta)$) are zero-free in their respective regions of analytic continuation.
\end{remark}

\begin{remark}
Condition (iii) { in Problem \ref{prob:RH.M}} is not independent of the others; the asymptotic expansion can be deduced by viewing $a$ and $\ba$ as $1/a_+^\sharp$ and $a_-^\sharp$ in the RHP discussed in Remark \ref{rem:ab.rh}.
\end{remark}

\subsubsection{Scattering Data}

The data in Problems \ref{prob:RH.M} and \ref{prob:RH.n} are related in the following way. 

\medskip

For the direction ($a$, $b$, $\ba$, $\bb$) $\Longrightarrow \rho$, we have:

\begin{lemma}
\label{lemma:ab.to.rho}
Given $a(\zeta)$, $b(\zeta)$, $\ba(\zeta)$, $\bb(\zeta)$ satisfying conditions (i)--(vii) in Problem \ref{prob:RH.M}, the function $\rho(\lam)$ given by
\begin{equation}
\label{ab.to.rho}
\rho(\zeta^2) = \zeta^{-1}\bb(\zeta)/a(\zeta) 
\end{equation}
satisfies conditions (i)--(ii) in Problem \ref{prob:RH.n}.
\end{lemma}

\begin{proof}
The function $\rho(\lam)$ is well-defined 
{
by (i), (ii)
}
since
since an even function on $\Sigma$ induces a well-defined function on $\bbR$. Since $b$ is odd, $\bb$ vanishes to first order at $\zeta=0$ (condition (i)) and the right-hand side of \eqref{ab.to.rho} has a removable singularity there. 
{
It easily follows from 
conditions (i), (ii), (vii)
}
that $\rho \in \calS(\bbR)$.  Observe that, since $\zetabar^2 = \zeta^2$ on $\Sigma$, we may compute
$$
\overline{\rho(\zeta^2)} = \overline{\rho(\zetabar^2)} = \zeta^{-1} \frac{b(\zeta)}{\ba(\zeta)}
$$
where in the last step we used (v).

{
From conditions (i), (ii), (v), and \eqref{ab.to.rho}, we conclude that
$$ 
1-\zeta^2 \left|\rho(\zeta^2)\right|^2 
		= 		1-\frac{b(\zeta)\bb(\zeta)}{a(\zeta)\ba(\zeta)}.
$$
From the symmetries (i), (ii), and (v), we may compute
\begin{equation}
\label{T.det}
a(\zeta)\ba(\zeta) - b(\zeta)\bb(\zeta) =
\begin{cases}
|a(\zeta)|^2 - |b(\zeta)|^2, 	& 	\zeta \in \bbR,\\[5pt]
|a(\zeta)|^2 + |b(\zeta)|^2, 	&	\zeta \in i \bbR.
\end{cases}
\end{equation}
so that 
by (vi)
$$
1-\zeta^2 \left| \rho(\zeta^2) \right|^2 = 
\begin{cases}
\dfrac{1}{|a(\zeta)|^2}, 									&\zeta \in \bbR\\[10pt]
1+ \dfrac{|b(\zeta)|^2}{|a(\zeta)|^2},	&\zeta \in i \bbR 
\end{cases}
$$
It now follows from (ii), (iii) ($\zeta \in \bbR$) and positivity ($\zeta \in i \bbR)$ that $1-\lam |\rho(\lam)|^2 >0$ strictly.
}
\end{proof}

\medskip

For the direction $\rho \Longrightarrow $ $a$, $b$, $\ba$, $\bb$, we have:
\begin{lemma}
\label{lemma:rho.to.ab}
Given $\rho$ satisfying conditions (i) and (ii) of Problem \ref{prob:RH.n}, there exist unique functions $a$, $b$, $\ba$, $\bb$ satisfying (i)--(vii) of Problem \ref{prob:RH.M}.
\end{lemma}

\begin{proof}
Given $\rho \in \calS(\bbR)$ with $1 - \lam |\rho(\lam)|^2 >0$ strictly, we may recover $a$,
$b$, $\ba$, and $\bb$ satisfying conditions (i)--(vii) in Problem \ref{prob:RH.M}  as follows. The scalar RHP (with contour $\bbR$)
\begin{align}
\label{rho.to.ab.RH}
a^\sharp_+(\lam)	&=	a^\sharp_- (\lam) (1-\lam |\rho(\lam)|^2)\\
\nonumber
a^\sharp_\pm -1		&\in 	\dee C_\bbR (L^2)
\end{align}  
has the unique solution
\begin{equation}
\label{RH.asharp.sol}
a^\sharp(z) 
	= \exp
			\left( 
				\int_\bbR \frac{\log\left(1-\lam |\rho(\lam)|^2\right)}{\lam-z} \frac{d\lam}{2\pi i}
			\right)
\end{equation}
for $z \in \bbC \setminus \bbR$. 
We set
\begin{equation}
\label{rho.to.a}
a(\zeta) = a^\sharp_-(\zeta^2), \quad \ba(\zeta) = \left(a^\sharp_+(\zeta^2)\right)^{-1}.
\end{equation}

It is easy to see that:

(a) 	By construction, $a$ and $\ba$ satisfy property (iv), 

(b)	$a$ and $\ba$ so defined are bounded below by a strictly positive constant, verifying property (vii),

(c) 	$a$ and $\ba$ are smooth functions of $\zeta$,  verifying property (ii),

(d)  	$a(\zeta)=1+\mathcal{O}(\zeta^{-1})$, that $\ba(\zeta)=1+\mathcal{O}(\zeta^{-1})$, and that $a-1$ and $\ba-1$ have complete asymptotic expansions in powers of $\zeta^{-2}$ for large $\zeta$ since $\rho \in \calS(\bbR)$, verifying property (iii).

Note that
\begin{equation}
\label{rho.to.aa}
a(\zeta)  \ba(\zeta) = \left(1-\zeta^2 |\rho(\zeta^2)|^2 \right)^{-1} >  0 \text{ strictly.}
\end{equation}

From the solution \eqref{RH.asharp.sol} we may compute
$$
a_-^\sharp(\lam) \overline{a_+^\sharp(\lam)} 
 	= 	\exp\left[ C^+_\bbR f \right]
 			\overline{\exp\left[ C^-_\bbR f \right]}\\
 	= 1
$$
where we used $\overline{C^-_\bbR (f)} = -C^+_\bbR(\overline{f})$ and the fact that 
$f(\lam) = \log(1-\lam|\rho(\lam)|^2)$ is real. From this relation and the fact that $\zetabar^2 = \zeta^2$, we conclude that $\ba(\zeta) = \overline{a(\zetabar)}$, verifying property (v) for $a$ and $\ba$.

We then solve for $\bb(\zeta)$ in the relation
\begin{equation}
\label{rho.to.b}
\rho(\zeta^2) = \zeta^{-1} \frac{\bb(\zeta)}{a(\zeta)} 
\end{equation}
and compute $b(\zeta) = \overline{\bb(\zetabar)}$, so that property (v) for $b$ holds by construction. These formulae show that $b$ and $\bb$ are odd functions of $\zeta$ and belong to $\calS(\Sigma)$, verifying property (i) for $b$ and $\bb$.  Note that, $\rho(\zeta^2) = \rho(\zetabar^2)$ so that, by construction,
\begin{equation}
\label{rhobar}
\overline{\rho(\zeta^2)} 	= 	\overline{\rho(\zetabar^2)}
									=	\zeta^{-1} \frac{b(\zeta)}{\overline{a(\zetabar)}}
									=	\zeta^{-1} \frac{b(\zeta)}{\ba(\zeta)}.
\end{equation}
It also follows from the construction that
\begin{equation}
\label{rho.to.bb}
b(\zeta) \bb(\zeta) = \zeta^2 |\rho(\zeta^2)|^2 a(\zeta) \ba(\zeta).
\end{equation}
From \eqref{rho.to.aa} and \eqref{rho.to.bb}, we conclude that
$$ a(\zeta)\ba(\zeta) - b(\zeta) \bb(\zeta) = 1, $$
verifying property (vi). 
\end{proof}

\subsubsection{Beals-Coifman Integral Equations}

Recall the discussion in \S \ref{sec:BC-integral-eq} about the Beals-Coifman integral equation associated to a given factorization for the jump matrix of a RHP.  We may factorize the jump matrices of the respective RHPs \ref{prob:RH.M} and \ref{prob:RH.n}:
\begin{align*}
v(\zeta) 							&= 	(\One-w^-(\zeta))^{-1}(\One+w^+(\zeta)),\\[5pt]
(w^-(\zeta),w^+(\zeta)) 	&=	\left( 
												\twomat{0}{\bb(\zeta)/a(\zeta)}{0}{0}, \,\,
												\twomat{0}{0}{-b(\zeta)/\ba(\zeta)}{0}
											\right),
\end{align*}
and
\begin{align*}
J(\zeta)						&=	(\One - W^-(\lam))^{-1} (\One + W^+(\lam)),\\[5pt]
\nonumber
(W^-(\lam),W^+(\lam))	&=		\left(
												\twomat{0}{\rho(\lam)}{0}{0}, \, \,
												\twomat{0}{0}{-\lam\overline{\rho(\lam)}}{0}
											\right).
\end{align*}
			
By Proposition \ref{prop:BC.sol}, unique solvability of Problem \ref{prob:RH.M} is equivalent to unique solvability of the Beals-Coifman integral equation
\begin{equation}
\label{RH.M.BC}
 \mu = \One + \calC_w \mu 
 \end{equation}
for $\mu$ with $\mu-\One \in L^2(\Sigma,|d\zeta|)$, where for a matrix-valued function $h$,
\begin{align*}
\calC_w h 	&= 	C^+_\Sigma (h w_x^-) + C^-_\Sigma (h w_x^+) \\[5pt]
				&= 	\Twomat{-C^-_\Sigma	\left[	h_{12} (b/\ba)e^{2ix(\dotarg)^2}	\right]}
									{C_\Sigma^+	\left[	h_{11}(\bb/a)e^{-2ix(\dotarg)^2}	\right]}
									{-C_\Sigma^-	\left[	h_{22}(b/\ba)e^{2ix(\dotarg)^2}	\right]}
									{C_\Sigma^+	\left[	h_{21}(\bb/a)e^{-2ix(\dotarg)^2}	\right]}.
\end{align*}
Here $C^\pm_\Sigma$ are Cauchy projectors for the contour $\Sigma$, and 
$$ w^\pm_x = e^{-ix\zeta^2 \ad \sigma } w^\pm. $$

Similarly, unique solvability of Problem \ref{prob:RH.n} is equivalent to unique solvability 
of the integral equation
\begin{equation}
\label{RH.n.BC}
\nu	=	\bfe_1  + \calC_W \nu
\end{equation}
for $\nu$ with $\nu-\bfe_1 \in L^2(\bbR)$, where for a row vector-valued function $h$, 
\begin{align*}
\calC_W h 	&= 	C^+_\bbR ( h W_x^-) + C^-_\bbR( h W_x^+) \\[5pt]
				&=	\left( 
							\begin{array}{ll} 
								-C^-_\bbR
									\bigl[ 
											h_{12}(\dotarg) (\dotarg)\overline{\rho(\dotarg)} e^{2ix\dotarg} 
									\bigr]
								&
								C^+_\bbR
									\bigl[ 
											h_{11}(\dotarg) \rho(\dotarg) e^{-2ix\dotarg} 
									\bigr]
							\end{array}
						\right).
\end{align*}
Here $C^\pm_\bbR$ are the Cauchy projectors for the contour $\bbR$ with its standard orientation, and 
$$ W^\pm_x = e^{-i\lam x \ad \sigma } W^\pm. $$

\subsection{Equivalence of the RHPs}
\label{sec:RH.equiv}

In this section, we continue to assume that:

\medskip
\textbf{(i)} For the RHP \ref{prob:RH.M}, the data $a$, $b$, $\ba$, $\bb$ satisfy conditions (i)-(vii) in Problem \ref{prob:RH.M}.

\medskip

\textbf{(ii)} For the RHP \ref{prob:RH.n}, $\rho$ satisfies conditions (i)--(ii) in Problem \ref{prob:RH.n}.

\bigskip

 First, we prove:

\begin{proposition}
\label{prop:mu.to.nu}
Suppose given $a$, $\ba$, $b$, $\bb$ satisfying (i)--(vii) of Problem \ref{prob:RH.M}, and 
suppose that $\mu$ solves the Beals-Coifman integral equation \eqref{RH.M.BC}. Define 
$$ \nu(x,\lam) = (\nu_{11}(x,\lam), \nu_{12}(x,\lam))$$ 
by
\begin{equation}
\label{mu.to.nu} 
\nu_{11}(x,\zeta^2) = \mu_{11}(x,\zeta), \quad \nu_{12}(x,\zeta^2) = \zeta^{-1}\mu_{12}(x,\zeta)
\end{equation}
and let $\rho$ be determined from ($a$, $\ba$, $b$, $\bb$) as in Lemma \ref{lemma:ab.to.rho}.
Then $\nu = (\nu_{11},\nu_{12})$ solves the Beals-Coifman integral equation \eqref{RH.n.BC}. 
\end{proposition}

\begin{proof}
From \eqref{RH.M.BC} we have
\begin{align}
\label{M.BC.11}
\mu_{11}(x,\zeta)		
									&=	1-C_\Sigma^-
												\left[ 
														(\dotarg) \mu_{12}(x,\dotarg) 
														\left(
															(\dotarg)^{-1} b(\dotarg)/\ba(\dotarg) 
														\right)
														e^{2ix(\dotarg)^2} 
												\right](\zeta),\\
\label{M.BC.12}
\zeta^{-1}\mu_{12}(x,\zeta)	
									&=	\zeta^{-1} C^+_\Sigma
												\left[(\dotarg)
														\mu_{11}(x,\dotarg)
														\left(
															(\dotarg)^{-1}(\bb(\dotarg)/a(\dotarg)
														\right)
														e^{-2ix(\dotarg)^2}
												\right](\zeta).
\end{align}
We now use the change-of-variables formula \eqref{Cauchypm.Sigma.to.R} and Remark \ref{rem:lee.cv},  noting that the argument of $C_\Sigma^-$ is an even function of $\zeta$, while the argument of $C^+_\Sigma$ is odd in $\zeta$. If we \emph{define} $\rho(\lam)$ as in \eqref{ab.to.rho}, it now follows from \eqref{Cauchypm.R.to.Sigma.even} and \eqref{Cauchypm.R.to.Sigma.odd}
\begin{align}
\label{n.BC.11}
\nu_{11}(x,\lam)		&=		1-C_\bbR^-
												\left( 
														\nu_{12}(x,\dotarg) (\dotarg) \overline{\rho(\dotarg)}
														e^{2ix(\dotarg)}
												\right)(\lam)\\
\label{n.BC.12}
\nu_{12}(x,\lam)		&=		C_\bbR^+
												\left(
														\nu_{11}(x,\dotarg) \rho(\dotarg)
														e^{-2ix(\dotarg)}
												\right)(\lam)
\end{align}
which is \eqref{RH.n.BC} in component form.
\end{proof}

On the other hand, the solution to the Beals-Coifman integral equation \eqref{RH.n.BC} for Problem \ref{prob:RH.n} can be completed, by symmetry, to a matrix-valued function, and transformed to furnish a solution to the Beals-Coifman integral equation \eqref{RH.M.BC} for Problem \ref{prob:RH.M}.

\begin{proposition}
\label{prop:nu.to.mu}
Suppose given $\rho$ satisfying (i)--(ii) of Problem \ref{prob:RH.n}, and suppose that $\nu=(\nu_{11},\nu_{12})$ solves the Beals-Coifman integral equation \eqref{RH.n.BC}. Define
\begin{equation}
\label{nu.to.mu}
\mu(x,\zeta) = \Twomat	{\nu_{11}(x,\zeta^2)}
									{\zeta \nu_{12}(x,\zeta^2)}
									{\zeta \overline{\nu_{12}(x,\zeta^2)}}
									{\overline{\nu_{11}(x,\zeta^2)}}
\end{equation}
and compute $a$, $\ba$, $b$, $\bb$ from $\rho$ as in Lemma \ref{lemma:rho.to.ab}.
Then $\mu$ solves \eqref{RH.n.BC}.
\end{proposition}

\begin{proof}
First we show that, if $\nu_{11}$ and $\nu_{12}$ solve \eqref{n.BC.11}--\eqref{n.BC.12}, then
$\mu_{11}$ and $\mu_{12}$ as defined by \eqref{nu.to.mu} solve \eqref{M.BC.11}--\eqref{M.BC.12}. We then argue by symmetry that $\mu$ solves \eqref{RH.M.BC}.

From \eqref{n.BC.11} and \eqref{Cauchypm.R.to.Sigma.even}, we deduce that
\begin{align}
\label{nu11.to.mu11}
\mu_{11}(x,\zeta)	&=	1-C_\Sigma^-
										\left[
											\mu_{12}(x,\dotarg)
											 (\dotarg) \overline{\rho((\dotarg)^2)} 
											e^{2ix(\dotarg)^2}
										\right](\zeta)\\
\nonumber
							&=	1-C_\Sigma^-
										\left[
											\mu_{12}(x,\dotarg)
											\left(b(\dotarg)/\ba(\dotarg)\right)
											e^{2ix(\dotarg)^2}
										\right](\zeta)
\end{align}
where we compute $a$, $b$, $\ba$, $\bb$ from $\rho$ following the steps outlined in the proof of Lemma \ref{lemma:rho.to.ab} above,  \eqref{rho.to.a} and \eqref{rho.to.b}.  In the last step, we used the identity \eqref{rhobar}. This is exactly the $(11)$ component of \eqref{RH.M.BC}.

Next, from \eqref{n.BC.12} and \eqref{Cauchypm.R.to.Sigma.odd}, we conclude that
\begin{align}
\label{nu12.to.mu12}
\mu_{12}(x,\zeta)	&=	\zeta C_\bbR^+
											\left[
												\nu_{11}(x,\dotarg) 
												\rho(\dotarg)
												e^{-2ix(\dotarg)}
											\right]
											(\zeta^2)\\
\nonumber
							&=	C_\Sigma^+
											\left[
												 \nu_{11}(x,(\dotarg)^2)
												 \left( \bb(\dotarg)/a(\dotarg) \right) e^{-2ix(\dotarg)^2}
											\right](\zeta)\\
\nonumber			
							&=	C_\Sigma^+
											\left[
												\mu_{11}(x,\dotarg)
												 \left( \bb(\dotarg)/a(\dotarg) \right) e^{-2ix(\dotarg)^2}
											\right](\zeta)
\end{align}
where in the second line we identified
$$h(\zeta^2) =  \nu_{11}(x,\zeta^2) \left(\zeta^{-1} \bb(\zeta)/a(\zeta) \right) e^{-2ix\zeta^2}$$
and 
$$ f(\zeta) =  \nu_{11}(x,\zeta^2) \left( \bb(\zeta)/a(\zeta) \right) e^{-2ix\zeta^2}$$
in order to apply \eqref{Cauchypm.R.to.Sigma.odd}.

Using \eqref{n.BC.12} again, we may compute
\begin{align}
\label{nu21.to.mu21}
\zeta \overline{\nu_{12}(x,\zeta^2)}
							&=	- \zeta C_\bbR^- 
											\left[ 
												\overline{\nu_{11}(x,\dotarg)} 
												\overline{\rho(\dotarg)} 
												e^{-2ix(\dotarg)}
											\right](\zeta^2)\\
\nonumber
							&=	-  C_\Sigma^-
											\left[
												\overline{\mu_{11}(x,\dotarg)}
												\left(b(\dotarg)/\ba(\dotarg)\right)
												e^{-2ix(\dotarg)^2} 
											\right](\zeta)
\end{align}
where in the first step we used the relation
$ \overline{C_\bbR^+( f )} = -C_\bbR^- (\overline{f}) $, 
and in the second line we used \eqref{Cauchypm.R.to.Sigma.odd} with the identifications
$$ 
h(\lam) = \overline{\nu_{11}(x,\lam)} 
					\overline{\rho(\lam)} 
					e^{-2ix(\lam)}
$$
and
$$
f(\zeta)	=	\overline{\mu_{11}(x,\zeta)} \frac{b(\zeta)}{\ba(\zeta)} e^{-2ix\zeta^2}.
$$

Finally, using \eqref{n.BC.11}, we have
\begin{align}
\label{nu22.to.mu22}
\overline{\nu_{11}(x,\zeta^2)}
							&=	1+C_\bbR^+
										\left[
											\overline{\nu_{12}(x,\dotarg)} (\dotarg) \rho(\dotarg) e^{2ix(\dotarg)}
										\right](\zeta^2)\\
\nonumber
							&=	1+C^+_\Sigma
										\left[
											\mu_{12}(x,\dotarg) 
											\left(\bb(\dotarg)/a(\dotarg)\right) 
											e^{2ix(\dotarg)^2}
										\right](\zeta)
\end{align}
where in the first step we used $\overline{C_\bbR^-(f)} = -C_\bbR^+(\overline{f})$ and in the second line we used \eqref{Cauchypm.R.to.Sigma.even} with the identifications
$$ g(\lam) = \overline{\nu_{12}(x,\lam)} \lam \rho(\lam) e^{2ix(\lam)}$$
and
$$ f(\zeta) = \mu_{12}(x,\zeta) \left(\bb(\zeta)/a(\zeta)\right) e^{2ix\zeta^2}. $$

Collecting \eqref{nu11.to.mu11}--\eqref{nu22.to.mu22} we see that $\mu$ as defined by \eqref{nu.to.mu} solves \eqref{RH.M.BC}.
\end{proof}

\subsection{Uniqueness Theorems}
\label{sec:RH.Unique}

In this section we prove that Problems \ref{prob:RH.M} and \ref{prob:RH.n} are uniquely solvable for the Schwartz class data described there.  First, we establish uniqueness for Problem \ref{prob:RH.M}.

\begin{lemma}
\label{lemma:unique.RH.M}
For given $a$, $b$, $\ba$, $\bb$ satisfying properties (i)--(vii) there exists a unique solution $M$ to Problem \ref{prob:RH.M} with $M_\pm - \One \in \dee C_\Sigma(L^2)$. 
\end{lemma}

\begin{proof} Given that $\det v(\zeta)=1$, this is actually a standard result in the theory of RHPs (see, for example, \cite[Theorem 9.1]{Zhou89}) but we give a proof for completeness. Suppose that $M$ solves the RHP and consider the scalar function $D(z)=\det M(z)$. The function $D(z)$ is analytic on $\bbC \setminus \Sigma$ with continuous boundary values $D_\pm$ on $\Sigma$, and $D(z) = 1 + \bigO{z^{-1}}$ as $|z| \rarr \infty$ in directions not tangent to $\Sigma$. From the jump relation for $M_\pm$ and the fact that $\det v(\zeta)=1$, we see that $D_+(\zeta) = D_-(\zeta)$. Hence $D(z)$ is entire and, by Liouville's Theorem, $D(z) \equiv 1$. Supposing now that $M_1$ and $M_2$ are two solutions, $M_1(z)^{-1}$ exists and the function $G(z) = M_1(z) M_2(z)^{-1}$ is analytic in $\bbC \setminus \Sigma$. By the jump relation for $M_1$ and $M_2$, $G_+(\zeta) = G_-(\zeta)$. Arguing as before we conclude that $G(z) \equiv 1$ and
$M_1 = M_2$.
\end{proof}

Now we use Proposition \ref{prop:nu.to.mu} to prove:

\begin{lemma}
\label{lemma:unique.RH.n}
For given $\rho \in \calS(\bbR)$ with $1-\lam |\rho(\lam)|^2 >0$ strictly, there exists a unique solution to the Problem \ref{prob:RH.n}.
\end{lemma}

\begin{proof}
Suppose that { $\nu^{(1)} $ and $\nu^{(2)}$ solve Problem \ref{prob:RH.n}. Let $\mu^{(1)}$ and $\mu^{(2)}$ 
be the solutions of Problem \ref{prob:RH.M} obtained from $\nu^{(1)}$ and $\nu^{(2)}$  via \eqref{nu.to.mu}. From Lemma \ref{lemma:unique.RH.M}, 
we have $\mu^{(1)} =
\mu^{(2)}$. It now follows from the formulas \eqref{mu.to.nu} that $\nu^{(1)} = \nu^{(2)}$.
}
\end{proof}

To study the integral equation \eqref{RH.n.BC}, we note its reduction to the scalar Fredholm integral equation
\begin{equation}
\label{BC.nusharp}
\widetilde{\nu} = 1+ S \widetilde{\nu}
\end{equation}
where, for a scalar function $h \in L^2(\bbR)$,
\begin{equation}
\label{S}
(Sh)(\lam) = -C^-_\bbR \left[ 
					C^+_\bbR 
							\left( 
								h(\dotarg) \rho(\dotarg) e^{-2ix(\dotarg)} 
							\right)(\diamond) 
					\left((\diamond) \overline{\rho(\diamond)} e^{2ix(\diamond)}\right) 
				\right](\lam).
\end{equation}

\begin{lemma}
\label{lemma:RH.n.BC.Fred}
Suppose that $\rho \in \calS(\bbR)$ with $1-\lam|\rho(\lam)|^2 > 0$ strictly, and fix $x \in \bbR$. Then 
\begin{itemize}
\item[(i)]		If $\nu=(\nu_{11},\nu_{12})$ with $\nu-\bfe_1 \in L^2(\bbR)$ solves \eqref{RH.n.BC}, then
 $\widetilde{\nu} = \nu_{11}$ solves the scalar integral equation \eqref{BC.nusharp}.
\item[(ii)]	If $\widetilde{\nu}$ with $\widetilde{\nu}-1 \in L^2(\bbR)$ solves \eqref{BC.nusharp}, then 
	$(\nu_{11},\nu_{22})$ given by
	$$
	\nu_{11} = \widetilde{\nu}, \quad 
	\nu_{12} = C_\bbR^+\left[ \widetilde{\nu}(\dotarg) \rho(\dotarg) e^{2ix(\dotarg)} \right]
	$$
	solves \eqref{RH.n.BC}.
\end{itemize}
\end{lemma}

\begin{proof}
A straightforward computation.
\end{proof}

We will show in Lemma \ref{lemma:S[h]} and Remark \ref{rem:S.cont} that the operator 
$S$ on $L^2(\bbR)$ defined by \eqref{S} is compact  depends continuously on $\rho$. Since $\rho \in H^{2,2}(\bbR)$, and $C^\pm_\bbR$ are bounded operators on $L^2(\bbR)$, it is easy to see that 
$$ 
S(1) \coloneq
	-C^-_\bbR \left[ 
						C^+_\bbR 
								\left( 
									\rho(\dotarg) e^{-2ix(\dotarg)} 
								\right) 
						(\diamond) \left((\diamond)\overline{\rho(\diamond)} e^{2ix(\diamond} \right)
					\right](\lam)
$$
defines an element of $L^2(\bbR)$.

Now, we prove the main result of this section.

\begin{proposition}
\label{prop:RH.n.BC.unique}
Suppose that $\rho\in H^{2,2}(\bbR)$ with $1-\lam |\rho(\lam)|^2 >0$ strictly. Then, the integral equation \eqref{BC.nusharp} has a unique solution $\widetilde{\nu}$ with 
$\widetilde{\nu} - 1 \in L^2(\bbR)$. 
\end{proposition}

\begin{proof}
It suffices to show that the resolvent $(I-S)^{-1}$ exists as an operator in $L^2$ since, if so, we can recover a unique solution via the formula
$$ \widetilde{\nu} = 1 + (I-S)^{-1}(S(1)). $$
Since the map $H^{2,2}(\bbR) \ni \rho \mapsto S \in \calB(L^2)$ is continuous and the invertible 
operators are open in $\calB(L^2)$, it suffices to show that $(I-S)^{-1}$ exists for $\rho \in \calS(\bbR)$ 
with $1-\lam|\rho(\lam)|^2 >0$ strictly. Finally, since $S$ is compact it suffices to show that $\ker(I-S)$ is 
trivial. 

Suppose that $\widetilde{\nu}_{(1)}$ and $\widetilde{\nu}_{(2)}$ are two solutions of \eqref{BC.nusharp}. By Lemma \ref{lemma:RH.n.BC.Fred}(ii), these solutions induce solutions $\nu_{(1)}$ and $\nu_{(2)}$ of \eqref{RH.n.BC} which are distinct as a consequence of the explicit formulas in Lemma \ref{lemma:RH.n.BC.Fred}(ii)
if  $\widetilde{\nu}_{(1)} \neq \widetilde{\nu}_{(2)}$. Hence, by the uniqueness result  of Lemma \ref{lemma:unique.RH.n},
$\widetilde{\nu}_{(1)} = \widetilde{\nu}_{(2)}$
and $\ker(I-S)$ is trivial.
\end{proof}

\subsection{Reconstruction of the Potential}
\label{sec:RH.recon}

In this subsection we show that the solution of the RHP \eqref{prob:RH.M} solves a 
differential equation of the form \eqref{LS.red} and obtain explicit formulas for $Q(x)$ and $P(x)$ 
having the correct structure (see Remark \ref{rem:integrand}). From these formulas and a change of variables, we can rewrite the reconstruction formula for $q$ in terms of the solution $\nu$ for the Beals-Coifman integral equation \eqref{RH.n.BC}.

Recall the Beals-Coifman integral equation \eqref{RH.M.BC}. We first note a simplification analogous to Lemma \ref{lemma:RH.n.BC.Fred}. Consider the integral equations
\begin{align}
\label{RH.M.BC.11}
\mu_{11}^\sharp 	&=	1 + S_{11} \mu_{11}^\sharp\\
\label{RH.M.BC.22}
\mu_{22}^\sharp		&=	1 + S_{22} \mu_{22}^\sharp
\end{align}
where
\begin{align*}
\left( S_{11} h \right)(\zeta) 
	&= 	-C_\Sigma^-
					\left[ C_\Sigma^+
							\left( 
									h(\dotarg) (\bb/a)(\dotarg) e^{-2ix(\dotarg)^2}
							\right) 
							(b/\ba)(\dotarg) e^{2ix(\dotarg)^2}
					\right]
					(\zeta)\\
\left( S_{22} h \right)(\zeta)	
	&=	-C_\Sigma^+
					\left[	C_\Sigma^-
						\left(
								h(\dotarg) (b/\ba)(\dotarg)e^{2ix(\dotarg)^2}
						\right)
						(\bb/a)(\dotarg) e^{-2ix(\dotarg)^2}
					\right]
					(\zeta)
\end{align*}
obtained from iterating \eqref{RH.M.BC}.

\begin{lemma}
\label{lemma:RH.M.BC.Fred}
Suppose that $(a,\ba,b,\bb)$ satisfy properties (i)--(vii) in Problem \ref{prob:RH.M} and fix $x \in \bbR$. Then:
\begin{itemize}
\item[(i)] If $(\mu_{11}^\sharp,\mu_{22}^\sharp)$ solve \eqref{RH.M.BC.11}--\eqref{RH.M.BC.22},
and we define
\begin{align*}
 \mu^\sharp_{12} 
 	&= C_\Sigma^+ \left[ 
 								\mu^\sharp_{11}(\dotarg) (\bb/a)(\dotarg) e^{2ix(\dotarg)^2} 
 							\right]\\
\mu^\sharp_{21}
	&=	C_\Sigma^- 
							\left[
								\mu^\sharp_{22}(\dotarg) (b/\ba)(\dotarg) e^{-2ix(\dotarg)^2}
							\right]
\end{align*}
then 
$$ 
\mu^\sharp = \twomat{\mu_{11}^\sharp}{\mu_{12}^\sharp}{\mu_{21}^\sharp}{\mu_{22}^\sharp}
$$
solves \eqref{RH.M.BC}.
\item[(ii)] If $\mu$ solves \eqref{RH.M.BC}, then $(\mu_{11},\mu_{22})$ solve \eqref{RH.M.BC.11}--\eqref{RH.M.BC.22}
\end{itemize}
\end{lemma}

We omit the proof. Lemma \ref{lemma:unique.RH.M} and Lemma \ref{lemma:RH.M.BC.Fred} imply that the system \eqref{RH.M.BC.11}--\eqref{RH.M.BC.22} are also uniquely solvable if $(a,\ba,b,\bb)$ satisfy properties (i)--(vii).

Now we can prove a symmetry result about the solution $\mu$ of \eqref{RH.M.BC}.

\begin{lemma}
\label{lemma:mu.even-odd}
Let $$\mu= \twomat{\mu_{11}}{\mu_{12}}{\mu_{21}}{\mu_{22}}$$ be the unique solution of \eqref{RH.M.BC}. Then $\mu_{11}$ and $\mu_{22}$ are even functions of $\zeta$ while $\mu_{21}$ and $\mu_{21}$ are odd functions of $\zeta$. Moreover, $\mu_{22}(x,\zeta) = \overline{\mu_{11}(x,\zetabar)}$ and $\mu_{21}(x,\zeta)=\overline{\mu_{12}(x,\zetabar)}$. 
\end{lemma}

\begin{proof}
Recall the Cauchy projectors $C^\pm$ preserve the subspaces of even and odd functions on $\Sigma$.
Using the fact that  $b$ and $\breve{b}$ are odd functions of $\zeta$ we first conclude that, 
if $(\mu_{11},\mu_{22})$ solve \eqref{RH.M.BC.11}--\eqref{RH.M.BC.22}, then $(\mu^\sharp_{11}(\zeta),\mu^\sharp_{22}(\zeta))=(\mu_{11}(-\zeta),\mu_{22}(-\zeta))$ also solve the same system of equations, and $(\mu_{11},\mu_{22})=(\mu_{11}^\sharp,\mu_{22}^\sharp)$ by unicity. It now follows directly from \eqref{RH.M.BC} that are odd functions of $\zeta$ as claimed.

To prove the statement about conjugates, we first observe that if $\mu_{11}^\sharp(x,\zeta)$ solves \eqref{RH.M.BC.11}, $\overline{\mu^\sharp_{11}(x,\zetabar)}$ solves \eqref{RH.M.BC.22} by straightforward computation and property (v) of $(a,\ba,b,\bb)$. We can use this relation and \eqref{RH.M.BC} to prove that $\overline{\mu_{12}}=\mu_{21}$. 
\end{proof}

Now, we show that $M_\pm$ obey the linear problem \eqref{LS.red} and obtain effective formulas for $Q$ and $P$. 

\begin{proposition}
\label{prop:recon}
The functions $M_\pm$ obey the differential equation
\eqref{LS.red} where $P$ and $Q$ are constructed from the solution $\mu$ of \eqref{RH.M.BC} as follows:
\begin{align}
\label{Q.recon}
Q(x) 	&=	-\frac{1}{2\pi} \ad \sigma  
						\left( \int_\Sigma 
									\mu(x,\zeta) 
										\left(
											w_x^+(\zeta)+ w_x^-(\zeta)
										\right) 
								\, d\zeta
						\right)\\
\label{P.recon}
P(x)	&=	Q(x) i (\ad \sigma)^{-1} Q(x) 
\end{align}
\end{proposition}

\begin{remark}
\label{rem:integrand}
Note that, by Lemma \ref{lemma:mu.even-odd}, the integrand in \eqref{Q.recon} 
\begin{align}
\label{integrand}
f(x,\zeta)		&=		\mu(\zeta) \left(w_x^+(\zeta)+ w_x^-(\zeta) \right)\\
\nonumber		&=		\Twomat{-\mu_{12}(x,\zeta) \overline{s(\zetabar)} e^{2i\zeta^2 x}}
											{\mu_{11}(x,\zeta) s(\zeta) e^{-2i\zeta^2 x}}
											{-\mu_{22}(x,\zeta) \overline{s(\zetabar)} e^{2i\zeta^2 x}}
											{\mu_{21}(x,\zeta) s(\zeta) e^{-2i\zeta^2 x}}
\end{align}
(where $s(\zeta) = \bb(\zeta)/a(\zeta)$) is odd off-diagonal and even on-diagonal. Given the orientation of the contour $\Sigma$ as shown in Figure \ref{fig:contours}, the even terms integrate to zero while the odd terms persist. 
Moreover, by Lemma \ref{lemma:mu.even-odd} again, $f(x,\zeta)$ takes the form
$$ f(x,\zeta) = \Twomat{f_{11}(x,\zeta)}
							{-\overline{f_{21}(x,\zeta)}}
							{f_{21}(x,\zeta)}
							{-\overline{f_{11}(x,\zeta)}}
$$
which, together with the formula \eqref{Q.recon}, shows that
$$ 
Q(x) = \twomat{0}{q(x)}{\overline{q(x)}}{0},
$$
as required. Here
\begin{equation}
\label{q.recon.pre}
q(x) = \frac{1}{\pi} \int_\Sigma  
								e^{-2ix\zeta^2} 
								\frac{\bb(\zeta)}{a(\zeta)} 
								\mu_{11}(x,\zeta) 
						\, d\zeta.
\end{equation}
\end{remark}

\begin{proof}
Let
$$ 
v_{x}(\zeta) = e^{-ix\zeta^2 \ad \sigma } v(\zeta).
$$
Differentiating the jump relation
$$ 
M_+(x,\zeta) = M_-(x,\zeta)  v_{x}(\zeta)
$$
with respect to $x$ and using the fact that $\ad \sigma $ is a derivation, we
compute
\begin{align*}
\frac{dM_+}{dx} 
	&= 	\frac{dM_-}{dx}v_{x}
				+M_-\left(-i\zeta^2 \ad \sigma  (v_x) \right)\\
\nonumber
	&=	\frac{dM_-}{dx}v_{x} 
				+ M_-\left(
						-i\zeta^2 \ad \sigma  
								\left( (M_-)^{-1} M_+ \right)
						\right) \\
\nonumber
	&=	\frac{dM_-}{dx}v_x + i \zeta^2\ad \sigma  (M_-) v_x - i\zeta^2 \ad \sigma  (M_+)
\end{align*}
We conclude that
\begin{equation}
\label{jump.x}
\frac{d}{dx}M_+(x,\zeta) + i\zeta^2 \ad \sigma  (M_+) =
\left( \frac{d}{dx} M_-(x,\zeta) + i \zeta^2 \ad \sigma  (M_-)\right) v_x.
\end{equation}
Using the fact that
\begin{equation}
\label{Npm.Cpm}
M_\pm(x,\zeta) -  { \One} = 
	C^\pm	\left[ 
						\mu(x,\dotarg) 
						\left( w_x^-(\dotarg)+ w_x^+(\dotarg) \right) 
			 	\right]
\end{equation}
and Lemma \ref{lemma:commutator}(ii), we conclude that
\begin{align*}
i\zeta^2 \ad \sigma  (M_\pm)(x,\zeta)
	&= i\ad \sigma 
			\Bigl[ \, C^\pm\left( (\dotarg)^2 f(x,\dotarg) \right) \\
	&\quad
				- \frac{\zeta}{2\pi i}\int_\Sigma f(x,s) \, ds
					- \frac{1}{2\pi i} \int_\Sigma s f(x,s) \, ds\,
			\Bigr]
\end{align*}
where
$ f(x,\zeta) $ is given by \eqref{integrand}. It follows from Lemma 
\ref{lemma:mu.even-odd} and \eqref{integrand} that the matrix-valued 
integral $\int_\Sigma \zeta f(x,\zeta) \, d\zeta$ is a diagonal matrix. Hence,
$$\ad \sigma \left( \int_\Sigma s f(x,s) \, ds \right) =0.$$ 
Hence defining $Q(x)$ by \eqref{Q.recon}, we have
\begin{equation}
\label{Com1} 
i \zeta^2 \ad \sigma  (M_\pm)(x,\zeta) = 
		i \ad \sigma  \left[ C^\pm\left( (\dotarg)^2 f(x,\dotarg) \right) \right]  +		
		\zeta Q(x). 
\end{equation}
Since Cauchy projection is bounded on $L^2$ and $\nu_{11}-1\in L^2$ for each fixed $x$ and $s$ and $\overline{s}$ are in the Schwartz class. Through change of variable we can show that the first term defines an $L^2$ function of $\zeta$ for each $x$.
Next, observe that, by \eqref{Npm.Cpm} and Lemma \ref{lemma:commutator}(i),
\begin{equation}
\label{Com2}
\zeta Q(x) (M_\pm - \textbf{1}) = 
	Q(x) C^\pm\left[ (\dotarg) f(x,\dotarg) \right]- Q(x) R(x)
\end{equation}
where $R(x)$ is given by 
\begin{align*}
R(x)	&=	\frac{1}{2\pi i} \int_\Sigma	
						\mu(x,\zeta) (w_x^+(\zeta)+ w_x^-(\zeta)) \, d\zeta\\
		&=	i (\ad \sigma)^{-1} Q(x)
\end{align*}
and the first right-hand term of \eqref{Com2} is an $L^2$ function of $\zeta$ for each $x$.

Now define
$$ 
W_\pm = \frac{dM_\pm}{dx} + i\zeta^2 \ad \sigma  (M_\pm )
	-\zeta Q(x) M_\pm(x) - Q(x) R(x) M_\pm(x).
$$
By \eqref{Npm.Cpm}, \eqref{Com1}, \eqref{Com2}, and the identity
\begin{align*}
 -\zeta Q(x) M_\pm(x) - Q(x) R(x) M_\pm(x) 
 	&=			-\zeta Q(x) - \zeta Q(x) (M_\pm-{ \One})\\
	 &\quad		- Q(x) R(x) - Q(x) R(x) (M_\pm - { \One}),
\end{align*}
it now follows that $(W_+,W_-) \in \dee C_\Sigma(L^2)$ each fixed $x$. Since, also,
$$W_+ = W_- v_x,$$
it follows from Lemma \ref{lemma:unique.RH.M} that $W_+ = W_- = 0$. 

\end{proof}

If we construct $\nu=(\nu_{11},\nu_{12})$ from $\mu$ as in Proposition \ref{prop:mu.to.nu}, we may use the reconstruction formula \eqref{q.recon.pre}, the  change of variables formula \eqref{lee.cv}, and the odd symmetry of the integrand in \eqref{q.recon.pre} to conclude that
\begin{equation}
\label{q.recon.nu}
q(x) = -\frac{1}{\pi} \int e^{-2ix\lam} \rho(\lam) \nu_{11}(x,\lam) \, d\lam.
\end{equation}

\section{The Inverse Scattering Map}
\label{sec:inverse}

%
%

In this section we prove Theorem \ref{thm:I} by studying the RHP
\begin{align}
\label{I:RH+}
\bfN_+(x,\lam)	
	&=	\bfN_-(x,\lam) e^{-i\lam x\ad \sigma } J(\lam)\\[5pt]
\nonumber
J(\lam)			
	&=	\Twomat{1-\lam |\rho(\lam)|^2}
						{\rho(\lam)}
						{-\lam\overline{\rho(\lam)}}
						{1}\\[5pt]
\nonumber
\bfN_\pm(x,\dotarg)-\bfe_1	
	&\in \dee C_\bbR(L^2)
\end{align}
to reconstruct $q$ on $(-a,\infty)$ for any $a>0$, and the  RHP
\begin{align}
\label{I:RH-}
\bbfN_+(x,\lam)
	&=	\bbfN_-(x,\lam) e^{-i\lam x\ad \sigma }\bJ(\lam)\\[5pt]
\nonumber
\bJ(\lam)
	&=	\Twomat{1}
						{\brho(\lam)}
						{-\lam \overline{\brho(\lam)}}
						{1-\lam|\brho(\lam)|^2}\\[5pt]
\nonumber
\bbfN_\pm(x,\dotarg) - \bfe_1
	&\in \dee C_\bbR(L^2).
\end{align}
to reconstruct $q$ on $(-\infty,a)$ for any such $a$. 
Here $\rho \in \calS(\bbR)$ with $1-\lam|\rho(\lam)|^2 \geq c>0$
is given scattering data. The data $\brho$ is constructed from $\rho$ by solving the scalar RHP
\begin{align*}
\delta_+(\lam) &= \delta_-(\lam) (1-\lam |\rho(\lam)|^2 )^{-1}, \\
\delta_\pm - 1 &\in \dee C(L^2).
\end{align*}
In Lemma \ref{lemma:delta.left} we show that this problem has a unique solution
and that the function $\Delta(\lam) = \delta_+(\lam)\delta_-(\lam)$ 
satisfies $|\Delta(\lam)|=1$, $\overline{\Delta(\lam)} = (\Delta(\lam))^{-1}$. We show that 
$\Delta$ and $\Delta^{-1}$ are bounded multiplication operators on $H^{2,2}(\bbR)$. We take 
$$\brho(\lam)=\rho(\lam)/\Delta(\lam).$$

The reconstruction procedure can be reduced in each case to a scalar integral equation and reconstruction formula; see \eqref{RH.n.BC} for the Beals-Coifman integral equation for the RHP \eqref{I:RH+}, see \eqref{BC.nusharp} for its reduction to a scalar integral equation,and \eqref{q.recon.nu} for the reconstruction formula. In the notation of \eqref{BC.nusharp}, set $\nu^\sharp = \widetilde{\nu} -1$.  For
\eqref{I:RH+}, we solve
\begin{equation}
\label{RH+.nu}
\nu^\sharp	=	\nu_0^\sharp + S\nu^\sharp, 
\quad \nu_0^\sharp = S[1]
\end{equation}
for $\nu^\sharp(x,\dotarg) \in L^2(\bbR)$, and compute
\begin{equation}
\label{RH+.q}
q(x) = -\frac{1}{\pi} \int_\bbR e^{-2i\lam x} \rho(\lam) (1+\nu^\sharp(x,\lam)) \, d\lam
\end{equation}
where $S$ is the operator
$$ (Sh)(\lam)
			 = -C_\bbR^-
				\left[ C_\bbR^+ 
					\left( \rho(\dotarg) h(\dotarg) e^{-2ix(\dotarg)} \right) 		
								(\diamond) 
					\left( (\diamond) \overline{\rho(\diamond)}
							e^{2ix(\diamond)}
					\right)
				\right] (\lam).
$$		
Similarly, for \eqref{I:RH-}, we solve
\begin{equation}
\label{RH-.nu}
\bnu^\sharp	=	\bnu^\sharp_0 + \bS\bnu^\sharp, 
\quad \bnu^\sharp_0  = \bS[1]
\end{equation}
for $\bnu^\sharp(x,\dotarg) \in L^2(\bbR)$, and compute
\begin{equation}
\label{RH-.q}
\bq(x) = -\frac{1}{\pi} \int_\bbR e^{-2i\lam x} \brho(\lam) (1+\bnu^\sharp(x,\lam)) \, d\lam
\end{equation}
where $\bS$ is the operator
$$ (\bS h)(\lam)
			 = -C_\bbR^-
				\left[ C_\bbR^+ 
					\left( \brho(\dotarg) h(\dotarg) e^{-2ix(\dotarg)} \right) 		
								(\diamond) 
					\left( (\diamond) \overline{\brho(\diamond)}
							e^{2ix(\diamond)}
					\right)
				\right] (\lam).
$$		

Denote by $V$ the set of $\rho \in H^{2,2}(\bbR)$ with $1-\lam|\rho(\lam)|^2 \geq c>0$ for some strictly positive constant $c$. Note that, if $\rho \in V$, we also have $\brho \in V$. We will first prove:

\begin{proposition}
\label{prop:RH+}
For any $a>0$, the map 
\begin{equation}
\label{RH+.map}
\begin{split}
V	&\longrightarrow		H^{2,2}(-a,\infty)\\
\rho 	&\mapsto			q
\end{split}
\end{equation}
initially defined on $V \cap \calS(\bbR)$ by the RHP \eqref{I:RH+} and the formula \eqref{RH+.q}, extends to a locally Lipschitz continuous map from $V$ to $H^{2,2}(-a,\infty)$.
\end{proposition}

We give the proof of Proposition \ref{prop:RH+} in section \ref{sec:RH.Mapping}.   
By essentially identical arguments, we may prove:

\begin{proposition}
\label{prop:RH-}
For any $a>0$, the map 
\begin{equation}
\label{RH-.map}
\begin{split}
V	&\longrightarrow		H^{2,2}(-\infty,a)\\
\rho 	&\mapsto			\bq
\end{split}
\end{equation}
initially defined on $V \cap \calS(\bbR)$ by the RHP \eqref{I:RH-} and the formula \eqref{RH+.q}, extends to a locally Lipschitz continuous map from $V$ to $H^{2,2}(-\infty,a)$.
\end{proposition}

We omit the proof.

Finally, we will prove:

\begin{proposition}
\label{prop:RH.map}
For any $a>0$, we have $q(x) = \bq(x)$ for $x \in (-a,a)$, so the maps \eqref{RH+.map}--\eqref{RH-.map} together define a locally Lipschitz mapping $\calI$
\begin{align*}
V 		&\longrightarrow		H^{2,2}(\bbR)\\
\rho	&\mapsto 	q
\end{align*}
with the property that $\calI \left( \calR (q) \right) = q$ on $U$, an open neighborhood of $0$ in $H^{2,2}(\bbR)$,
and $\calR (\calI(\rho)) = \rho$ on $V$. 
\end{proposition}

\begin{proof}[Proof of Theorem \ref{thm:I}]
By Proposition \ref{prop:RH.map}, $ \calR \circ \calI$ extends to the identity map on $V$, and $\calR$ is one-to-one from $U$ to $V$. Theorem \ref{thm:I} now follows.
\end{proof}

\subsection{Mapping Properties}
\label{sec:RH.Mapping}

In this subsection we prove Proposition \ref{prop:RH+}. For $\rho \in \calS(\bbR)$ it is easy to compute from \eqref{RH+.q} that
\begin{align}
\label{RH+.qxx}
q_{xx}	
	&= \frac{1}{\pi} 
				\int_\bbR e^{-2i\lam x} 
						4\lam^2 \rho (1+\nu^\sharp) 
			-\frac{1}{\pi}
				\int_\bbR e^{-2i\lam x}\rho(\lam) 
						\left(-4i\lam  \nu^\sharp_x
							  +\nu^\sharp_{xx}
						\right)\\[5pt]
\label{RH+.xxq}
x^2 q(x)
	&=	-\frac{1}{4 \pi}
				\int_\bbR e^{-2i\lam x} 
					\rho''(1+\nu^\sharp) 
			-\frac{1}{4\pi} 
			\int_\bbR e^{-2i\lam x}
					\left( 
						2\rho' \nu^\sharp_\lam
						+ \rho\nu^\sharp_{\lam \lam}
					\right)
\end{align}
where for the second identity we used $(-2i)^{-1} (d/d\lam)e^{-2i\lam x}  = xe^{-2i\lam x}$ and integrated by parts. From these formulas and the mapping properties of the Fourier transform,
we can easily obtain the following sufficient conditions for Lipschitz continuity of the map $\rho \rarr q$.

\begin{lemma}
\label{lemma:RH.mapping}
Suppose that the functions
\smallskip
\begin{equation}
\label{nu.todo}
\nu^\sharp, 
\quad \nu^\sharp_x, 
\quad \nu^\sharp_{xx}, 
\quad \nu^\sharp_\lam,
\quad \langle \lam \rangle^{-1} \nu^\sharp_{\lam \lam}
\end{equation}
\medskip
all belong to $L^2((-a,\infty) \times \bbR)$ as functions of $(x,\lam)$ for any $a>0$. Then $q$ defined by \eqref{RH+.q} belongs to $H^{2,2}(-a,\infty)$ for any such $a$. If, moreover, the functions \eqref{nu.todo}are Lipschitz continuous as functions of $\rho \in V$, then the map $\rho \mapsto q$ is locally Lipschitz continuous from $V$ into $H^{2,2}(-a,\infty)$. 
\end{lemma}

\begin{proof}
In the formulas \eqref{RH+.q}, \eqref{RH+.qxx}, \eqref{RH+.xxq}, there is a linear term which has the correct mapping properties by Fourier theory and one or more terms involving $\nu^\sharp$ or its derivatives which integrate to $L^2$ functions of $x$ by the mapping properties of $\nu^\sharp$ and H\"{o}lder's inequality. The Lipschitz continuity follows from the fact that these expressions are bilinear in $\rho$ and $\nu$ and their derivatives.
\end{proof}

To prove that $\nu^\sharp$ and its derivatives have the desired properties, we study the integral equation \eqref{RH+.nu} in detail. First, we show that the inhomogeneous term $\nu_0^\sharp$ has the required properties
(Lemma \ref{lemma:S[1]}). Second, we study mapping properties of $S$, $\dee S/\dee x$, and $\dee^2 S/\dee x^2$ (Lemma \ref{lemma:S[h]}). Third, we prove the existence of the resolvent $(I-S)^{-1}$ (Lemma \ref{lemma:S.res}). Fourth, we obtain the required estimates on $\nu^\sharp$, $\nu_x^\sharp$, and $\nu^\sharp_{xx}$ using the integral equation and differentiation of the parameter (Lemma \ref{lemma:nu.x}). Finally, we obtain the needed estimates on $\nu^\sharp_\lam$ and $\nu^\sharp_{\lam\lam}$ by direct differentiation (Lemma \ref{lemma:nu.lam}). 

We'll make repeated use of the explicit formulas
\begin{align}
\label{S.h}
S[h](x,\lam)	&= 
             -\frac{1}{2i\pi^2}
             		\int_{0}^{-\infty}
             				e^{2i\lambda\xi}
             					\int_{x}^{\infty}
             							(\widehat{{\rho}} *\widehat{h})(\xi') \,
             						{	\widehat{\overline{{\rho}}}~'(\xi-\xi')}
             				\, d\xi'\, d\xi, \quad h \in L^2 \\[5pt]
\label{S.1}
S[1](x,\lam)	 &=
				-\frac{1}{2i\pi^2}
					\int_{0}^{-\infty}
							e^{2i\lambda\xi}
								\int_{x}^{\infty}\widehat{{\rho}}(\xi') \,
								{ \widehat{\overline{{\rho}}}~'(\xi-\xi')}
						\, d\xi' \, d\xi.	
\end{align}
which follow from elementary properties of the Fourier transform
and the fact that $C^\pm$ act in Fourier representation as multiplication by the characteristic functions of $\bbR^\pm$.

\textbf{(1) Estimates on $\nu_0^\sharp$.} Using \eqref{S.1} and Fourier theory, we obtain basic estimates on $S[1]\equiv \nu_0^\sharp$ and its derivatives.

\begin{lemma}
\label{lemma:S[1]}
Suppose that $\rho \in H^{2,2}(\bbR)$,  $a >0$, and $x \geq -a$. 
{ Let $\nu_0^\sharp = S[1]$.} Then:
\begin{subequations}
\label{S[1].est}
\begin{align}
\label{S[1]}
\norm{{\nu_0^\sharp}(x,\dotarg)}{L^2_\lam}	
	&\lesssim (1+|x|)^{-1} \norm{\rho}{H^{2,0}} \norm{\rho}{L^2} \\
\label{S[1].x}
\norm{(\nu_0^\sharp)_x}{L^2(\bbR \times \bbR)}
	&\leq		\norm{\rho}{L^{2,1}} \norm{\rho}{L^2}\\
\label{S[1].xx}
\norm{(\nu_0^\sharp)_{xx}}{L^2((-a,\infty)\times \bbR)}
	&\leq	\norm{\rho}{L^{2,2}} \norm{\rho}{L^{2,1}}\\
\label{S[1].l}
\norm{(\nu_0^\sharp)_{\lam}(x,\dotarg)}{L^2_\lam}
	&\lesssim	(1+|x|)^{-1}\norm{\rho}{H^2}\norm{\rho}{H^{2,2}}\\
\label{S[1].ll}
\norm{ \langle \dotarg \rangle^{-1} (\nu_0^\sharp)_{\lam\lam}(x,\dotarg)}
		{L^2_\lam}
	&\lesssim	(1+|x|)^{-1} \norm{\rho}{H^{2,2}}^2
\end{align}
\end{subequations}
where the implied constants depend only on $a$. 
\end{lemma}

\begin{remark}
\label{rem:S[1]}
Since ${\nu_0^\sharp}$ and its derivatives are bilinear in $\rho$, the
estimates used to prove Lemma \ref{lemma:S[1]} can easily be adapted to show that ${\nu_0^\sharp}$ and its derivatives are Lipschitz continuous as a function of $\rho \in H^{2,2}(\bbR)$. 
\end{remark}

\begin{proof}
It follows from \eqref{S.1} that, Fourier transforming in $\lam$, 
\begin{equation}
\label{S.1.hat}
\widehat{{\nu_0^\sharp}}(x,\xi)
	 = -\chi_-\frac{1}{2\pi^2 i} 
	 			\int_x^\infty
	 				{ \widehat{\overline{{\rho}}}~'(\xi-\xi')}\, 
	 				\widehat{\rho(\xi')}
	 				\, d\xi'.
\end{equation}
Note that, in \eqref{S.1.hat},  $\xi \leq 0$ and $\xi' \geq x \geq -a$.  We'll sketch the proofs assuming $a=0$ and 
{ when $a>0$, we simply write
$$\int_x^\infty=\int_x^0+\int_0^\infty$$
and use the fact that $|\xi-a|>|\xi|$. }The { same conclusions follow.
}

Using \eqref{S.1.hat} and the inequality 
$(1+|x|) \leq (1+|\xi-\xi'|)$, we easily recover \eqref{S[1]} using
$\Norm{\widehat{\rho'}}{L^1} \leq C \Norm{\widehat{\rho'}}{L^{2,1}}$, Young's inequality,  and Plancherel's theorem. Estimates \eqref{S[1].x} and \eqref{S[1].xx} follow by differentiating \eqref{S.1.hat}.

To prove \eqref{S[1].l} it suffices to estimate 
$\Norm{\widehat{{\nu_0^\sharp}}}{L^{2,1}}$. Since 
$0 \leq x \leq |\xi'|$ and $|\xi| \leq |\xi'| \leq |\xi-\xi'|$
we may estimate  
$$
(1+|\xi|) (1+|x|) \left| \widehat{{\nu_{0\lambda}^\sharp}}(x,\xi) \right|
\leq 
\frac{1}{2\pi^2} 
	\int_x^\infty 
			(1+|\xi-\xi'|) 	\left| { \widehat{\overline{{\rho}}}~'(\xi-\xi')} \right|
			(1+|\xi'|) 		\left| \widehat{\rho}(\xi') \right| 
	\, d\xi' .
$$
and use Young's inequality together with the estimate
$ \norm{x f'(x)}{L^2} \leq \norm{f}{H^{2,2}}$. 

To prove \eqref{S[1].ll}, we compute
\begin{align}
\label{nu0.lam.2}
(\nu_0^\sharp)_{\lam \lam}
	&=	\frac{2i}{\pi^2}
			\int_{-\infty}^0 e^{2i\lam \xi}
			\xi^2
			\int_{x}^{\infty} \widehat{{\rho}}(\xi')
				{{ \widehat{\overline{\rho}}}~'(\xi-\xi')}
		\, d\xi' \, d\xi
\\
\nonumber
       &=-\frac{2i}{\pi^2}
       		\int_{-\infty}^0 e^{2i\lam \xi}
       				\xi^2
       					\frac{d}{d\xi}
       						\left(
       							\int_{x}^{\infty}\widehat{{\rho}}(\xi')
       							{{ \widehat{\overline{\rho}}}(\xi-\xi')}
       						\right)
       				\, d\xi' \, d\xi
\\
\nonumber
	&=\frac{2i}{\pi^2}
			\int_{-\infty}^0 e^{2i\lam \xi}
				\left( 2i\lambda\xi^2 +
				2\xi \right)
					\int_{x}^{\infty}
						\widehat{{\rho}}(\xi')
						{{ \widehat{\overline{\rho}}}(\xi-\xi')}
	\, d\xi' \, d\xi
\end{align}
Following the pattern of the previous arguments, it is easy to see that 
$$
\left|
	\langle\lambda\rangle^{-1}
	\frac{\dee^2 {\nu_0^\sharp}}{\dee\lambda^2}
\right|_{L^2_{\lambda}}
\lesssim(1+|x|)^{-1}
\norm{\rho}{H^{2,2}}^2
$$ 
\end{proof}

\textbf{(2) Mapping Properties of $S$, $S_x$, and $S_{xx}$.}
Next, we study $S$, $S_x$, and $S_{xx}$ as bounded operators on $L^2(\bbR)$ using \eqref{S.h}. We will need the Banach space
\begin{equation}
\label{X}
X
\equiv 
\left\{ 
	f \in L^2(\bbR): 
		\widehat{f}, \, \widehat{f}\,' \in L^1(\bbR)
\right\}
\end{equation}
It is not difficult to see that the embedding $H^{2,2}(\bbR) \hookrightarrow X$
is compact so that, in particular, any bounded subset $B$ in $H^{2,2}(\bbR)$ has compact closure in $X$.  

\begin{lemma}
\label{lemma:S[h]}
Suppose that $\rho \in H^{2,2}(\bbR)$.
\begin{enumerate}
\item[(i)]	 	$S$, $S_x$, and $S_{xx}$ are bounded operators on $L^2$ with norm 
				$\leq C \norm{\rho}{H^{2,2}}^2$ where $C$ is independent of $x$. 
\item[(ii)]	The operator $S$ is Hilbert-Schmidt with 
				$\norm{S}{\HS} \leq C \norm{\rho}{H^{2,2}}^2$ and
				$$\lim_{x \rarr \infty} \norm{S}{\HS} = 0$$ 
				uniformly in bounded subsets 
				of $H^{2,2}(\bbR)$.
\item[(iii)]	For any $a \geq 0$, the map 
				\begin{align*}
				(-a,\infty) \times X	&\longrightarrow		\calB(L^2)\\
				(x,\rho)	&\mapsto	S
				\end{align*}
				is continuous.
\end{enumerate}
\end{lemma}

\begin{remark}
\label{rem:S.cont}
Lemma \ref{lemma:S[h]}(iii) and the fact that $H^{2,2}(\bbR)$ is continuously embedded in $X$ implies that the map
\begin{align*}
(-a,\infty) \times H^{2,2}(\bbR) &\longrightarrow \calB(L^2) \\
(x,\rho)	&\mapsto S
\end{align*}
is also continuous.
\end{remark}

\begin{proof}
(i) Let $\calF$ denote the Fourier transform. From \eqref{S.h} we see that the operator $\calF S \calF^{-1}$ has integral kernel
\begin{equation}
\label{S.ker}
 K(\xi,\xi'';x) = \int_x^\infty { \widehat{\overline{{\rho}}}~'(\xi-\xi')} \hatrho(\xi'-\xi'') \, d\xi', \quad \xi \leq 0
 \end{equation}
up to trivial constants, so that
\begin{equation}
\label{S.est}
\norm{\calF (S \check{h})}{L^2} \leq C \norm{\hatrho}{L^1} \norm{\hatrho'}{L^1} \norm{h}{L^2}.
\end{equation}
The estimate \eqref{S.est} shows that $\norm{S}{\calB(L^2)}$ is bounded by $C \norm{\rho}{H^{2,2}}^2$
and from \eqref{X}, we also have $\norm{S}{\calB(L^2)} \leq C \norm{\rho}{X}^2$. Differentiating \eqref{S.ker} with respect to $x$ we have
\begin{equation}
\label{Sx.ker}
K_x(\xi,\xi'';x) = -{ \widehat{\overline{{\rho}}}~'(\xi-x)} \hatrho(x-\xi'')
\end{equation}
so that
$$
\norm{S_x}{\calB(L^2)} 	\leq \norm{S_x}{\HS} 
										\leq \norm{\hatrho~'}{L^2} \norm{\hatrho}{L^2}.
$$
Differentiating again we find 
$$
\norm{S_{xx}}{\calB(L^2)}	\leq \norm{S_{xx}}{\HS} 
											\leq \norm{\hatrho~''}{L^2}\norm{\hatrho}{L^2} + \norm{\hatrho~'}{L^2}^2.
$$

(ii) From \eqref{S.ker} and integration by parts we have
\begin{align*}
K(\xi, \xi'';x) &= 	-{ \widehat{\overline{{\rho}}}(\xi-x)}\hatrho(x-\xi'') 
						- \int_x^\infty { \widehat{\overline{{\rho}}}(\xi-\xi')}\hatrho'(\xi'-\xi'') \, d\xi' \\
				 &=	K_1(\xi,\xi'';x) + K_2(\xi,\xi'';x)
\end{align*}
where $\xi <0 $ and $\xi' \geq x \geq -a$. Clearly
\begin{equation}
\label{K1.est}
\norm{K_1}{L^2(\bbR^- \times \bbR)} 
	\leq \left(\int_x^\infty \left| \hatrho(t) \right|^2 \, dt \right)^{1/2} \,
			\norm{\hatrho}{L^2}
	\leq C_a (1+|x|)^{-1/2} \norm{\hatrho}{L^{2,1}} \norm{\hatrho}{L^2}.
\end{equation}
Since $|\xi'-\xi| \geq x$ and $|\xi-\xi'| \geq \xi'$, we have
$$
 \left| \hatrho(\xi-\xi') \right| \leq C_a (1+|x|)^{-1/2} (1+|\xi'|)^{-3/2} (1+|\xi-\xi'|)^2 \left|\hatrho(\xi-\xi') \right|
$$
It follows that
\begin{equation}
\label{K2.est}
\left| K_2(\xi, \xi'') \right| \leq C_a (1+|x|)^{-1/2} \norm{\hatrho}{L^{2,2}} \norm{\hatrho'}{L^2}
\end{equation}
The estimates \eqref{K1.est}--\eqref{K2.est} show that 
$$\norm{S}{\HS} \leq C_a (1+|x|)^{-1/2} \norm{\hatrho}{L^{2,2}} \norm{\hatrho}{H^1}$$
which proves (ii). 

(iii) Write $S = S_{x,\rho}$. Using the technique that proved \eqref{S.est} we have
\begin{equation}
\label{S.rho}
\norm{S_{x,\rho_1}-S_{x,\rho_2}}{\calB(L^2) } 
	\leq  		C_a \left(
							 	\norm{\hatrho_1 - \hatrho_2}{L^1} \norm{\hatrho_1'}{L^1}
							+	\norm{\hatrho_2}{L^1} \norm{\hatrho'_1-\hatrho_2'}{L^1}
					\right)
\end{equation}
uniformly in $x \geq -a$. 
On the other hand, by  \eqref{Sx.ker} and Young's inequality,
$$
\norm{\dee S/\dee x}{\calB(L^2)} \leq C \norm{\hatrho'}{L^1} \norm{\hatrho}{L^1}
$$
so that
\begin{equation}
\label{S.x}
\norm{S_{x,\rho}-S_{y,\rho}}{\calB(L^2)} \leq C|x-y|^{1/2} \norm{\hatrho'}{L^1} \norm{\hatrho}{L^1}.
\end{equation}
Combining \eqref{S.rho} and \eqref{S.x} we obtain the claimed continuity.
\end{proof}

\begin{remark}
\label{rem:Sh}
Since all estimates in the proof of Lemma \ref{lemma:S[h]} are bilinear in $\rho$, it follows that $\rho \mapsto S$, $\rho \mapsto S_x$, and $\rho \mapsto S_{xx}$ are locally Lipschitz maps from $H^{2,2}$ to the bounded operators on $L^2$. 
\end{remark}

\textbf{(3) Resolvent Estimates.} We can now construct the resolvent $(I-S)^{-1}$ as a bounded operator on $L^2$. Although we do not obtain the kind of explicit integral representation we obtained for resolvents in the direct problem, we are able to extend the resolvent family to a bounded operator on the space $L^2((-a,\infty) \times \bbR)$
by the uniformity of the estimate \eqref{S.res} with respect to $x$.

\begin{lemma}
\label{lemma:S.res}
Suppose that $\rho \in H^{2,2}(\bbR)$ and $a>0$. The resolvent $(I-S)^{-1}$ exists as a bounded operator on $L^2$ and 
\begin{equation}
\label{S.res}
\sup_{x \geq -a} \norm{(I-S)^{-1}}{\calB(L^2)} \leq C
\end{equation}
with $C$ uniform in $\rho$ in a fixed bounded subset of $H^{2,2}(\bbR)$. Moreover
\begin{equation}
\label{S.res.lip} 
\sup_{x \geq -a} \norm{(I-S_{x,\rho})^{-1} - (I-S_{x,\sigma})^{-1}}{\calB(L^2)}
\leq
C \norm{\rho-\sigma}{H^{2,2}}
\end{equation}
with $C$ uniform in $\rho,\sigma$ in a fixed bounded subset of $H^{2,2}(\bbR)$. 
\end{lemma}

\begin{proof}
The estimate \eqref{S.res.lip} follows from \eqref{S.res} and the second resolvent identity. It suffices to prove \eqref{S.res}.

By Proposition \ref{prop:RH.n.BC.unique}, Lemma \ref{lemma:S[h]}(ii), and Fredholm theory, the resolvent $(I-S_{x,\rho})^{-1}$ exists for all $(x,\rho) \in \bbR \times H^{2,2}(\bbR)$. 

To bound the norm of the resolvent, fix a bounded subset $B$ of $H^{2,2}(\bbR)$. By Lemma 
\ref{lemma:S[h]}(ii), there is an $R$ so that $\norm{S_{x,\rho}}{\calB(L^2)} < 1/2$ for $x \geq R$ 
and all $\rho \in B$. Thus $(I-S_{x,\rho})^{-1}$ exists for all such $(x,\rho)$ and 
$\norm{(I-S_{x,\rho})^{-1}}{\calB(L^2)} \leq 2$. 
To control the resolvent for $(x,\rho) \in [-a,R] \times B$, we note that $[-a,R] \times B$ is 
a compact subset of $\bbR \times X$ 
By Lemma \ref{lemma:S[h]}(iii) and the second resolvent 
formula, the map $(x,\rho) \rarr (I-S_{x,\rho})^{-1}$ is a continuous map from $\bbR \times X$ 
to $\calB(L^2)$. The continuous image of the set $[-a,R] \times B$ is compact, hence bounded, in 
$\calB(L^2)$. 
\end{proof}

\begin{remark}
\label{rem:S.R}
(i) The family of operators $(I-S_{x,\rho})^{-1}$ for $x \geq -a$ defines a bounded operator $R_\rho$ from $L^2((-a,\infty)\times \bbR)$ to itself by the formula
$$ (R_\rho f)(x,\dotarg) = (I-S_{x,\rho})^{-1} f(x,\dotarg).$$
By Lemma \ref{lemma:S.res} the map $\rho \rarr R_\rho$ is locally bounded and locally Lipschitz continuous from $H^{2,2}(\bbR)$ to $\calB(L^2((-a,\infty) \times \bbR))$. 
(ii) Lemma \ref{lemma:S[h]}(i) shows that $S$, $S_x$, and $S_{xx}$ extend in the same way to bounded operators on $L^2((-a,\infty)\times \bbR)$ to itself, Lipschitz continuous in $\rho$. 
\end{remark}

\textbf{(4) Solving for $\nu^\sharp$, $\nu_x^\sharp$, and $\nu_{xx}^\sharp$.}
We can now study the maps $\rho \rarr \nu^\sharp$, $\rho \rarr \nu^\sharp_x$, and $\rho \rarr \nu^\sharp_{xx}$. 

\begin{lemma}
\label{lemma:nu.x}
Fix $a \geq 0$. 
The maps $\rho \rarr \nu^\sharp$, $\rho \rarr \nu^\sharp_x$, and $\rho \rarr \nu^\sharp_{xx}$ are Lipschitz continuous from $H^{2,2}(\bbR)$ to $L^2((-a,\infty)\times \bbR)$. 
\end{lemma}

\begin{proof}
As $\nu^\sharp = R_\rho \nu_0^\sharp$ the continuity of $\rho \mapsto \nu^\sharp$ is an immediate consequence of \eqref{S[1]} and Remark \ref{rem:S.R}(i). Differentiating \eqref{RH+.nu}, we have
$$ \nu^\sharp_x = (\nu_0^\sharp)_x + S_x \nu^\sharp + S( \nu^\sharp_x ) $$
so that
$$ \nu^\sharp_x = R_\rho(\nu_0^\sharp + S_x \nu^\sharp)$$
which defines a Lipschitz continuous map $\rho \rarr \nu^\sharp_x$ by Remark \ref{rem:S.R}(i),(ii) and the continuity of $\rho \mapsto \nu^\sharp$. A similar argument using the identity
$$ \nu^\sharp_{xx} = (\nu_0^\sharp)_{xx} + S_{xx} \nu^\sharp+ 2S_x \nu^\sharp_x + S(\nu^\sharp_{xx}) $$
yields continuity of $\rho \rarr \nu^\sharp_{xx}$.
\end{proof}

\textbf{(5) Solving for $\nu_\lam$ and $\nu_{\lam\lam}$.} 
It remains to study the maps $\rho \rarr \nu^\sharp_\lam$ and $\rho \rarr \nu^\sharp_{\lam\lam}$.\

\begin{lemma}
\label{lemma:nu.lam}
Fix $a \geq 0$. 
The maps $\rho \rarr \nu^\sharp_\lam$ and $\rho \rarr  \laminv \nu^\sharp_{\lam\lam}$ are Lipschitz continuous from $H^{2,2}(\bbR)$ to $L^2((-a,\infty) \times \bbR)$. 
\end{lemma}

\begin{proof}
By the integral equation \eqref{RH+.nu} and the estimates \eqref{S[1].l}--\eqref{S[1].ll}, it suffices to show that the maps $\rho \rarr (S\nu^\sharp)_\lam$ and $\rho \rarr (S\nu^\sharp)_{\lam\lam}$ are Lipschitz continuous. From the identity \eqref{S.h}, we compute (up to trivial constants)
\begin{align}
\label{S[h].lam}
{\calF( S[h]_\lambda)}(x,\xi)
	&=
		\xi \,
				\int_x^\infty \left( \hatrho * \widehat{h} \right)(\xi')
					{ \widehat{\overline{{\rho}}}~'(\xi-\xi')} 
					, d\xi' \\
\nonumber
	&=	 		\xi (\hatrho*\widehat{h})(x) { \widehat{\overline{{\rho}}}(\xi-x)}
				- 	\xi
						\int_x^\infty
								(\hatrho~' * \widehat{h})(\xi') \, { \widehat{\overline{{\rho}}}(\xi-\xi')}
						\, d\xi'
\end{align}
where in the second step we integrated by parts. The first right-hand term in \eqref{S[h].lam} has
$L^2((-a,\infty)\times \bbR)$-norm estimated by $\norm{\hatrho}{L^1} \norm{ h}{L^2} \norm{\hatrho~'}{L^{2,1}}$. To estimate the second right-hand term in \eqref{S[h].lam} we note that
\begin{equation}
\label{S[h].lam1}
\left| \xi \int_x^\infty (\hatrho~' * \widehat{h})(\xi') { \widehat{\overline{{\rho}}}(\xi-\xi')} \, d\xi' \right|
\leq
\int_x^\infty |(\hatrho~' * \widehat{h})(\xi')| \, |\xi-\xi'|  \, |\hat{\rhobar}(\xi-\xi')| \, d\xi	\\
\end{equation}
since $\xi$ and $\xi'$ have opposite sign. 
By Young's inequality, the right-hand side of \eqref{S[h].lam1} has $L^2((-a,\infty)\times \bbR)$-norm estimated by $\norm{\hatrho~'}{L^1} \norm{\hatrho}{L^{2,1}} \norm{h}{2}$. Hence
$$ 
\norm{\frac{\dee S[h]}{\dee \lam}}{L^2((-a,\infty) \times \bbR)} \leq C\norm{\hatrho~'}{L^{1}}\norm{\hatrho}{H^{2,2}} \norm{h}{L^2}.
$$
Applying this estimate with $h = \nu^\sharp$ we see that $(\dee S[\nu^\sharp]/\dee \lam \in L^2((-a,\infty)\times \bbR)$. The Lipschitz continuity of $\nu^\sharp$ in $\rho$ together with the bilinear estimates above show that $\rho \rarr \dee S[\nu^\sharp]/\dee \lam$ is Lipschitz.

To study $\laminv \dee^2 S[\nu^\sharp] /\dee \lam$,  we use the same integration by parts trick used in 
\eqref{nu0.lam.2} to conclude that
$$
\left( \frac{\dee^2(S[\nu^\sharp])}{\dee \lam^2}\right)(x,\lam) =
	\frac{2i}{\pi}
		\int_{-\infty}^0
		 	e^{2i\lambda\xi}(2i\lam \xi^2 + 2\xi ) 
				\int_x^\infty 
					(\hatrho~' * \nu^\sharp )(\xi') \, { \widehat{\overline{{\rho}}}(\xi-\xi')}
		\, d\xi' \, d\xi.
$$
By the Plancherel theorem, it suffices to bound the $L^2((-a,\infty) \times \bbR)$-norm of the function
$$
G(x,\xi) =  (1+|\xi|)^2 \int_x^\infty(\hatrho~' * \nu^\sharp)(\xi') \, { \widehat{\overline{{\rho}}}(\xi-\xi')} 	\, d\xi'.
$$
As $\xi < 0$ and $\xi' \geq x$ we may estimate
$$
|G(x,\xi)| \leq 
	\int_x^\infty |(\hatrho~'*\widehat{\nu^\sharp})(\xi')|\, (1+|\xi-\xi'|)^2) |\widehat{\rhobar}(\xi-\xi')| \, d\xi'.
$$
By Young's inequality we get 
$$ \norm{G(x,\dotarg)}{L^2(\bbR)} \leq  \norm{\hatrho~'}{L^1} \norm{\widehat{\nu^\sharp}}{L^1} \norm{\hatrho}{L^{2,2}} $$
where 
$$
\norm{\widehat{\nu^\sharp}}{L^1}\leq 
\Vert\langle\xi\rangle\widehat{\nu^\sharp}\Vert_{L^2}
\leq\Vert\widehat{\nu^\sharp}\Vert_{L^2}+\norm{\frac{\dee\nu^\sharp}{\dee\lambda}}{L^2}
$$
which gives the desired estimate.
\end{proof}

\begin{proof}[Proof of Proposition \ref{prop:RH+}] 
An immediate consequence of Lemmas \ref{lemma:RH.mapping}, \ref{lemma:nu.x}, and \ref{lemma:nu.lam}.
\end{proof}

\subsection{Inversion}
\label{sec:RH.Inversion}

In this subsection we prove Proposition \ref{prop:RH.map}. Our arguments follow the analogous arguments for NLS in  \cite[\S 3]{DZ03} closely. 

To reconstruct the potential $q$ on the left, we use the standard trick (see, for example \cite[\S 3, p.~1058ff.]{DZ03}) of conjugating to a new ``left'' RHP that gives good estimates on the inverse map for $x<a$ but which is easily seen to reconstruct the same potential as the `right' RHP already considered. To do so we need the following technical result.

\begin{lemma}
\label{lemma:delta.left}
Consider the scalar RHP on $\bbR$:
\begin{enumerate}
\item[(1)] $\delta$ is analytic in $\bbC \setminus \bbR$ 
\item[(2)] $\delta$ has nontangential boundary values $\delta_\pm(z)$ as $\pm \Imag(z) \darr 0$ with
				$\delta^\pm -1 \in \dee C^\pm(L^2)$ and 
				$$\delta_+(\lam) = \delta_-(\lam) \left(1- \lam |\rho(\lam)|^2\right)^{-1}.$$
\end{enumerate}
For $\rho \in H^{2,2}(\bbR)$ with $1 - \lam |\rho(\lam)|^2 >0 $ strictly:
\begin{enumerate}
\item[(i)] The RHP has the unique solution
\begin{equation}
\label{delta}
\delta(z) = \exp\left( -\frac{1}{2\pi i} \int_{-\infty}^\infty \frac{1}{z-\lam} \log(1-\lam |\rho(\lam)|^2) \, d\lam \right)
\end{equation}
 where $c^{-1}<|\delta(z)|<c$ for a positive constant $c$ depending only on $\norm{\rho}{H^{2,2}}$.
\item[(ii)]  The function
$$ \Delta(\lam) = \delta_+(\lam)\delta_-(\lam) $$
satisfies $|\Delta(\lam)|=1$, $\overline{\Delta(\lam)} = \left(\Delta(\lam)\right)^{-1}$, and $\Delta, \Delta^{-1}$ are bounded multiplication operators from $H^{2,2}(\bbR)$ to itself.
\item[(iii)]	The relation
$ \lim_{|z| \rarr \infty} \delta(z) = 1 $
holds, where the limit is taken in any direction not tangent to the real axis.
\end{enumerate}
\end{lemma}

\begin{remark}
\label{rem:delta.to.delta.tilde}
Let $\tdelta(z)= \delta(z^2)$ for $z \in \Sigma$, and construct $a$, $\ba$, $b$, $\bb$ from $\rho$ as in the proof of Lemma \ref{lemma:rho.to.ab}. Then $\tdelta$ solves the scalar Riemann-Hilbert problem
\begin{align*}
\tdelta_+(\zeta)& = \tdelta_-(\zeta)\left(1 - \frac{b(\zeta)\bb(\zeta)}{a(\zeta)\ba(\zeta)}\right)\\
\tdelta_\pm - 1 &\in \dee C_\Sigma(L^2)
\end{align*}
with contour $\Sigma$, and $\tdelta_+(\zeta)=\ba(\zeta)$, $\tdelta_-(\zeta)=a(\zeta)^{-1}$.  A straightforward computation 
shows that the jump matrices \eqref{BC.v.right} and \eqref{BC.v.left} are related by
$$ 
v_{\ell}(\zeta) 
	= 	\diagmat{\tdelta_-(\zeta)^{-1}}{\tdelta_-(\zeta)} 
		v_{r}(\zeta) 
		\diagmat{\tdelta_+(\zeta)}{\tdelta_+(\zeta)^{-1}},
\quad \zeta \in \Sigma.
$$
It follows that if $M$ and $\bM$ are respective solutions to the Riemann-Hilbert problems $(\Sigma,v_r)$ and $(\Sigma,v_\ell)$, then
\begin{equation}
\label{BC:r.to.l}
\bM(x,z) = M(x,z) \diagmat{\tdelta(z)}{\tdelta(z)^{-1}}, \quad z \in \bbC \setminus \Sigma.
\end{equation}
\end{remark}

\begin{proof}
(i) The function 
$$g(\lam) = 1-\lam |\rho(\lam)|^2 $$ 
is strictly nonnegative and belongs to $L^{2,1}(\bbR) \cap H^2(\bbR)$ since 
$\rho \in H^{2,2}(\bbR)$.  
Since $g$ is continuous,  it is an easy 
consequence of Liouville's theorem that there is at most one solution. On the other 
hand, the boundary values of the function $\delta$ defined in \eqref{delta} are 
$$
\delta_\pm(\lam)
	=\exp\left(
				-\left[\pm \frac{1}{2}h(\lam)-\frac{1}{2}(Hh)(\lam)\right]
			\right), \quad h \equiv \log(g)
$$
by the Sokhotski-Plemelj  formulas \eqref{Plemelj-Sokhotski}. It follows that
$$
\frac{\delta_+(\lam)}{\delta_-(\lam)} 
	= \exp(-h(\lam)) = \left(1-\lam |\rho(\lam)|^2\right)^{-1}.
$$
so \eqref{delta} gives the unique solution.
Since $\log g \in H^1(\bbR)$, the Cauchy integral of $\log g$ 
is  a bounded function on $\bbC \setminus \bbR$ by standard estimates on the Cauchy integral (see, for example \cite[Lemma 23.3]{BDT88}), so that the exponential is bounded above and below in modulus.

 (ii) The function $\Delta(\lam)$ is given by
$$ \Delta(\lam) = \exp( (Hh)(\lam) ) $$
so $Hh \in H^2(\bbR)$ by the above properties of $g$.  With our normalization of $H$, $Hh$ is purely imaginary since $g$ 
is real, from which the identities $|\Delta(\lam)|=1$ and $\overline{\Delta(\lam)} = \left(\Delta(\lam)\right)^{-1}$ follow. It follows from the fact that $Hh \in H^2(\bbR)$ and $|\Delta(\lam)|=1$ that
$\Delta$ and $\Delta^{-1}$ are bounded multiplication operators on $H^{2,2}(\bbR)$.

(iii) follows by dominated convergence.

\end{proof}

Suppose now that $\bfN$ solves the original RHP and  let
$$ \bbfN(x,z) = \bfN(x,z) \Twomat{\delta(z)}{0}{0}{\delta(z)^{-1}}. $$
Then $\bbfN$ solves the RHP
\eqref{I:RH-}. Proposition \ref{prop:RH-} asserts that this RHP has a unique solution and defines a Lipschitz continuous map ${ \brho } \rarr \bq$. We now establish:

\begin{lemma}
\label{lemma:q.all}
For any $x \in \bbR$, $q(x) = \bq(x)$.
\end{lemma}

\begin{proof}
Since the maps $\rho \mapsto q$ and $\brho \mapsto \bq$ are Lipschitz, it suffices to prove equality for $\rho \in \calS(\bbR)$. In this case, the solutions $\bfN(x,z)$ and $\bbfN(x,z)$ have large-$z$ asymptotic expansions in $\bbC \setminus \bbR$ of the form
$$ 
\bfN(x,z) = \bfe_1 + \frac{N_{-1}(x)}{z} + \bigO{\frac{1}{z^2}},
\quad
\bbfN(x,z) = \bfe_1 + \frac{\bN_{-1}(x)}{z} + \bigO{\frac{1}{z^2}}
$$
while the function $\delta(z)$ satisfies
$$ 
\delta(z) = 1+ \bigO{\frac{1}{z}}.
$$
From these formulae, it is easy to see that
$$
N_{12}(x,z) = \bN_{12}(x,z) + \bigO{\frac{1}{z^2}}. 
$$
We now use the fact that the reconstruction formulas \eqref{RH+.q} for $q$ and \eqref{RH-.q} for $\bq$ are equivalent to the formulae
\begin{align*}
q(x)		&= 	\lim_{z \rarr \infty} z N_{12}(x,z)\\
\bq(x)	&= 	\lim_{z \rarr \infty} z \bN_{12}(x,z)
\end{align*}
to conclude that $q = \bq$. 
\end{proof}

Let $\chi \in C^\infty(\bbR)$ with $\chi(x) = 1$ for $x \geq 1$ and $\chi(x) =0$ for $x < -1$. Define
$$\calI(\rho) = q(x) \chi(x) + \bq(x) (1-\chi(x)).$$
The map $\calI$ is Lipschitz continuous from $H^{2,2}(\bbR)$ to itself. It remains to show that $\calI$ inverts $\calR$ and that $\calR$ is one-to-one. To this end, we use the following uniqueness result for the {Beals-Coifman solutions} of \eqref{LS.red}.

\begin{proposition}
\label{prop:BC}
Suppose that $q \in H^{2,2}(\bbR) \cap U$.
\begin{itemize}
\item[(i)]		There exists at most one matrix-valued solution $M(x,z)$, 
				analytic for $z \in \bbC \setminus \Sigma$,
				 of the problem
				\begin{align*}
				\frac{d}{dx} M(x,z) &=	-iz^2 \ad \sigma  (M) + z 
				{ Q(x) } M + P(x) M,\\
				\lim_{x \rarr +\infty} M(x,z) &= \One,\\
				M(x,z) &\text{ is bounded as }x \rarr -\infty.\\
				\end{align*}
				
\item[(ii)]	There exists at most one matrix-valued solution $M(x,z)$, 
				analytic for $z \in \bbC \setminus \Sigma$,
				 of the problem
				\begin{align*}
				\frac{d}{dx} M(x,z) &=	-iz^2 \ad \sigma  (M) + z  
				{Q(x)} M + P(x) M,\\
				\lim_{x \rarr -\infty} M(x,z) &= \One,\\
				M(x,z) &\text{ is bounded as }x \rarr +\infty.\\
				\end{align*}
\end{itemize}
\end{proposition}

\begin{proof}
We prove (i) since the proof of (ii) is similar. Suppose that $M_1$ and $M_2$ are two such solutions. It is easy to prove that there is a matrix $A(z)$ with $\det A(z)=1$ and 
$$
M_1(x,z) 	= M_2(x,z) e^{-iz^2 x \ad \sigma } A(z)
				= M_2(x,z) \Twomat{A_{11}(z)}{A_{12}(z) e^{-2iz^2 x}}
											{A_{21}(z) e^{2iz^2 x}}{A_{22}(z)}
$$
Using the exponential blow-up of the factors $e^{\pm i z^2 x}$ as $x \rarr \infty$ together with the asymptotic conditions it is easy to see that $A_{12}(z) = A_{21}(z) = 0$ for $z \in \bbC \setminus \Sigma$.
We can then use the asymptotic condition as $x \rarr +\infty$ to show that $A_{11}(z) = A_{22}(z) = 1$.
Hence $M_1=M_2$.
\end{proof}

\begin{lemma}
\label{lemma:invert}
For any $\rho \in H^{2,2}(\bbR) \cap V$, $\calI (\rho) \in U$ and $\calR (\calI (\rho)) = \rho$. Moreover, the map $\calR$ is one-to-one from $U$ onto $V$. 
\end{lemma}

\begin{proof}
To prove the first assertion, it suffices by Lipschitz continuity and density to show that $\calR (\calI (\rho) ) = \rho$ for $\rho \in \calS(\bbR) \cap V$. For given $\rho \in \calS(\bbR) \cap V$  we solve the Beals-Coifman integral equations for the RHP's \eqref{I:RH+} and \eqref{I:RH-}, obtaining solutions $\nu$ and $\bnu$.  By Proposition \ref{prop:nu.to.mu} of \S \ref{sec:RH.Schwartz} we can construct solutions $\mu$ and $\bmu$ to the Beals-Coifman integral equations for the corresponding RHP's on $\Sigma$. Now define
\begin{align*}
M(x,z) 
	&= \One
			+ \int_\Sigma \mu(x,\zeta) 
					\left(w_x^+(\zeta) + w_x^-(\zeta)\right)
					 \frac{1}{\zeta-z} \, \frac{d\zeta}{2\pi i},\\[5pt]
\bM(x,z)
	&=  \One
			+ \int_\Sigma \bmu(x,\zeta) 
					\left(\bw_x^+(\zeta) + \bw_x^-(\zeta)\right)
					 \frac{1}{\zeta-z} \, \frac{d\zeta}{2\pi i}.
\end{align*}
The functions $M$ and $\bM$ solve \eqref{LS.red} and are analytic in $z \in \bbC \setminus \Sigma$. Using the estimate \eqref{S[1]}, the boundedness of $(I-S)^{-1}$, the solution formula $\nu -1 = \nu_0^\sharp + (I-S)^{-1} \nu_0^\sharp$, and the construction of $\mu$ from $\nu$ in Proposition \ref{prop:nu.to.mu}, it is easy to see that
$ \norm{\mu-\One}{L^2} \leq C (1+|x|)^{-1}$ for $x > 0$. 
Using this estimate, the Riemann-Lebesgue lemma, and the formula
\begin{align*}
M(x,z) 	&= 	\One + 
					\int_\Sigma 
						\left(\mu(x,\zeta)-\One \right)
						\left(w_x^+ + w_x^-\right) \, 
						\frac{1}{\zeta-z} \, 
						\frac{d\zeta}{2\pi i}\\
			&\quad + 
					\int_\Sigma \left(w_x^+ + w_x^- \right)
						\frac{1}{\zeta-z} \, 
						\frac{d\zeta}{2\pi i}
\end{align*} 
we see that $M(x,z) \rarr \One$ as $x \rarr +\infty$. A similar argument shows that $\bM(x,z) \rarr \One$ as $x \rarr -\infty$.  It now follows from \eqref{BC:r.to.l} and Lemma \ref{lemma:delta.left}(i) that $M(x,z)$ is bounded as $x \rarr -\infty$.
Hence, $M(x,z)$ is the unique ``right'' Beals-Coifman solution for $q=\calI(\rho)$. 

Let $s(\zeta) = \bb(\zeta)/a(\zeta)$ where $a$ and $\bb$ are constructed from $\rho$ via Lemma \ref{lemma:rho.to.ab}. Then $M_\pm(x,\zeta)$ satisfy the jump relation
$$ 
M_+(x,\zeta) = M_-(x,\zeta)  v_x(\zeta), \quad
v_x(\zeta) = e^{-ix\zeta^2 \ad \sigma }\Twomat{1-s(\zeta)\overline{s(\zetabar)}}{s(\zeta)}{-\overline{s(\zetabar)}}{1}, 
$$
This is equivalent to the statement that $\calR(\calI(\rho))=\rho$ since the scattering data for the unique Beals-Coifman solutions corresponds exactly to $\rho$. 

Next, we prove that $\calR$ is one-to-one. Suppose that $q_1$ and $q_2$ are potentials with $\calR(q_1)=\calR(q_2)$.  We can construct ``right'' Beals-Coifman solutions $M_1$ and $M_2$ which 
both satisfy the \emph{same} jump relation. It follows that $(M_1-M_2)_+ = (M_1-M_2)_- v_x$ and $(M_1-M_2)_\pm \in \dee C(L^2)$
so that $M_1=M_2$ by uniqueness of solutions to the RHP. But then the reconstruction formulas show that $q_1=q_2$, so $\calR$ is one-to-one.
\end{proof}

\begin{proof}[Proof of Proposition \ref{prop:RH.map}]
An immediate consequence of Lemmas \ref{lemma:q.all} and \ref{lemma:invert}.
\end{proof}

\section{Inverse Scattering Solution to DNLS}
\label{sec:RHP-t}

%
%

The purpose of this section is to prove that the function
\begin{equation}
\label{q.recon.t.bis} 
q(x,t) = \calI\left( e^{-4i(\dotarg)^2 t} (\calR q_0)(\dotarg) \right)(x)
\end{equation}
solves the equation \eqref{DNLS2} if $q_0 \in \calS(\bbR) \cap U$ where $U$ is the spectrally determined set in Theorem \ref{thm:DNLS2}. 
{ Equation \eqref{zcr2}  from
 Appendix \ref{app:gauge}  } 
is  the zero-curvature representation of \eqref{DNLS2} where $q=q(x,t)$:
\begin{align}
\label{DNLS.Lax.x.bis}
\psi_x 	&=	-i\zeta^2 \psi + \zeta Q(x,t)\psi + P(x,t) \psi,\\
\label{DNLS.Lax.t.bis}
\psi_t		&=	-i\zeta^4 \psi + 2\zeta^3 Q(x,t) \psi
					+i\zeta^2 \diagmat{-|q|^2}{|q|^2} \psi	\\
\nonumber
			&\quad
					+ i\zeta \offdiagmat{q_x}{-\qbar_x} \psi 
					+ \frac{i}{4}\diagmat{|q|^4}{-|q|^4}\psi \\
\nonumber
			&\quad
					+\frac{1}{2}
						\diagmat{q_x\qbar-q \qbar_x}
									{q \qbar_x - q_x \qbar} \psi.
\end{align}
Recall that a \emph{fundamental solution} to this system is an invertible matrix-valued solution $\psi(x,t,\zeta)$.
The computations in Appendix \ref{app:gauge} imply the following criterion which is the key to our analysis.

\begin{lemma} 
\label{lemma:Lax}
{ Let  } $q \equiv q(x,t) \in C^\infty(\bbR \times \bbR)$, 
$$ 
Q(x,t) = \offdiagmat{q(x,t)}{\overline{q(x,t)}}, \quad
P(x,t) = \frac{i}{2} \diagmat{-|q(x,t)|^2}{|q(x,t)|^2},
$$
and suppose that there exists a fundamental solution 
$\psi(x,t,\zeta)$ of the system \eqref{DNLS.Lax.x.bis}--\eqref{DNLS.Lax.t.bis} for each $\zeta \in \Sigma$.
Then $q(x,t)$ is a classical solution of \eqref{DNLS2}.
\end{lemma}

Our main result is:

\begin{theorem}
\label{thm:RHP.t.q}
Suppose that $q_0 \in \calS(\bbR) \cap U$. Then the function \eqref{q.recon.t.bis} is a classical solution of \eqref{DNLS2}.
\end{theorem}

\begin{remark}
From the continuity of the solution map \eqref{DNLS2.sol}, it follows that if 
$q_0 \in H^{2,2}(\bbR) \cap U$, then the function \eqref{q.recon.t.bis} is a strong solution of \eqref{DNLS2}.
\end{remark}

By Lemma \ref{lemma:Lax}, it suffices to construct a fundamental solution for \eqref{DNLS.Lax.x.bis}--\eqref{DNLS.Lax.t.bis}. We will do so using the RHP.

Suppose that $M_\pm$ solve the RHP
\begin{align}
\label{RHP.t}
M_+(x,t,\zeta)	&=	M_-(x,t,\zeta) e^{it\theta \ad (\sigma)} v (\zeta)  \\
\nonumber
M_\pm - I		&\in 	\dee C(L^2)
\end{align}
where
$$
v(\zeta)	=	\Twomat{1-s(\zeta)\overline{s(\zetabar)}}{s(\zeta)}{-\overline{s(\zetabar)}}{1},
\quad
\theta(x,t,\zeta)= -\left( \zeta^2 \frac{x}{t}+ 2 \zeta^4 \right).$$
In terms of the transition matrix
$$ T(\zeta) = \twomat{a(\zeta)}{\bb(\zeta)}{b(\zeta)}{\ba(\zeta)} $$
we have $s(\zeta)= \bb(\zeta)/a(\zeta)$, and we have used the symmetries
$$ 
\ba(\zeta) = \overline{a(\zetabar)}, \quad 
\bb(\zeta)= \overline{b(\zetabar)}.
$$
We also have $a(-\zeta)=a(\zeta)$, $b(-\zeta)=-b(\zeta)$ so that $s(\zeta)$ is odd. 
We will assume that $s \in \calS(\Sigma)$. We have the factorization $v =(I-w_-)^{-1}(I+w_+)$ 
where
\begin{equation}
\label{RHP.w}  
w_- = \offdiagmat{s(\zeta)}{0}, \quad w_+ = \offdiagmat{0}{-\overline{s(\zetabar)}} 
\end{equation}
We define
\begin{equation}
\label{RHP.wxt}
w_{x,t}^\pm(\zeta) = e^{it\theta \ad \sigma } w^\pm(\zeta), 
\end{equation}
so that
\begin{equation}
\label{RHP.wxt.t}
\frac{\dee w_{x,t}^\pm}{\dee x}=-i\zeta^2 \ad \sigma  (w_{x,t}^\pm), \quad
\frac{\dee w_{x,t}^\pm}{\dee t} =- i\zeta^4 \ad \sigma  w_{x,t}^\pm. 
\end{equation}


\begin{proposition}
\label{prop:RHP.t}
Suppose that $a,b,\ba,\bb$ satisfy conditions (i)--(vii) in Problem 
\ref{prob:RH.M}  and that  $M_\pm$
solve the RHP \eqref{RHP.t}. { Let}
\begin{align*}
 Q(x,t) 	&= {  -  } \frac{1}{2\pi} \ad \sigma  
						\left[ \int_\Sigma \mu 
								\left(w_{x,t}^+  + w_{x,t}^- \right) 
						\right]    \\
P(x,t) 	&=	i Q(x,t) (\ad \sigma )^{-1} Q(x,t) = 
						\frac{i}{2} \Diagmat{-|q|^2}{|q|^2}.
\end{align*}
Then $M_\pm$  are fundamental solutions for the Lax equations
\begin{align}
\label{RHP.Lax.x}
\frac{\dee M_\pm}{\dee x} 
	&	= -i\zeta^2 \ad \sigma (M_\pm) + \zeta Q(x) M_\pm + P(x) M_\pm,\\
\label{RHP.Lax.t}
\frac{\dee M_\pm}{\dee t}(x,t,\zeta) 
	&=-2i\zeta^4 \ad \sigma (M_\pm) +  A(x,t,\zeta)M_\pm(x,t,\zeta) 
\end{align}
where
\begin{align}
\label{RHP.A}
A(x,t,\zeta) 
	&= 			2 	\zeta^3 \Offdiagmat{q}{\qbar}
					+ 	i\zeta^2 \Diagmat{-|q|^2}{|q|^2} 
					+	i\zeta\Offdiagmat{q_x}{-\qbar_x}\\[5pt]
\nonumber
	&\quad 		+ \frac{i}{4} \Diagmat{|q|^4}{-|q|^4} 
					+ \frac{1}{2} \Diagmat{q_x \qbar - q \qbar_x}{q\qbar_x  - q_x \qbar}
\end{align}
\end{proposition}

\begin{proof}
We have already shown in Proposition \ref{prop:recon} that, for each fixed $t$, $M_\pm$ obeys \eqref{RHP.Lax.x}. In Appendix \ref{app:time},  we show  that \eqref{RHP.Lax.t} also holds. Straightforward computation then shows that the functions $\psi_\pm = M_\pm e^{it\theta \sigma}$ obey \eqref{DNLS.Lax.x.bis}--\eqref{DNLS.Lax.t.bis}. 
\end{proof}
{
 The proof of Theorem \ref{thm:RHP.t.q} is }
an immediate consequence of Proposition \ref{prop:RHP.t} and Lemma \ref{lemma:Lax}.


\appendix

\section{Gauge Equivalence}
\label{app:gauge}

%
%

In this Appendix, we provide  details about  the  correspondence between the gauge transformation  relating solutions $u$ and $q$ of DNLS equations  \eqref{DNLS1} and \eqref{DNLS2}  ($\varepsilon=1$) and a matrix gauge transformation  relating their respective Lax pairs $(L,A)$ and $(L',A')$. 

 We write $v=\overline{u}$ and $r=\overline{q}$. 
We know from  the original paper of Kaup and Newell \cite{KN78} that  the DNLS equation \eqref{DNLS1} is equivalent to the zero-curvature condition
\footnote{This terminology refers to a geometrical interpretation of \eqref{pair} where the matrix operators $L$ and $A$ are seen as connection coefficients.}
\begin{equation}\label{zcr1}
 L_t - A_x + [L,A] =0 
 \end{equation}
where 
the operators $L$ and $A$ are given by 
\begin{equation*}
\begin{aligned}
L 	&= 	-i \zeta^2 \sigma + \zeta U(x), \\[0.2cm]
A 	&=	(-i)\twomat{A_{11}}{A_{12}}{A_{21}}{-A_{11}}.
\end{aligned}
\end{equation*}

Here $ U = \twomat{0}{u}{v}{0}$,
while
\begin{align*}
A_{11}	&=	2\zeta^4+\zeta^2 uv \\
A_{12}	&=	2i\zeta^3 u - \zeta u_x +i\zeta u^2 v\\
A_{21}	&=	2i\zeta^3 v + \zeta v_x +i\zeta u v^2 .
\end{align*}
The zero-curvature representation \eqref{zcr1}
$$ L_t - A_x+[L,A] = 
	\zeta \Twomat{0}{u_t-iu_{xx}-\left(u^2 v\right)_x}
			{v_t+iv_{xx}-\left(v^2 u\right)_x}{0} =0
$$
 gives the  evolution equations 
\begin{equation}
\label{DNLS.uv}
\begin{aligned}
u_t	&=	iu_{xx}+(u^2 v)_x, \\
v_t	&=	-iv_{xx}-(v^2 u)_x.
\end{aligned}
\end{equation}

\begin{proposition}
The zero-curvature representation associated to \eqref{DNLS2} has the form 
\begin{equation} \label{zcr2}
L_t '- A_x' + [L',A'] =0
\end{equation}
where the  Lax  pair  $(L',A')$  is  equivalent to $(L,A)$  through a matrix gauge transformation and 
\begin{equation}
\begin{aligned}
L'	&=	-i\zeta^2 \sigma + \zeta Q(x) + P(x)\\[0.2cm]
A'	&=	(-i) \twomat{A'_{11}}{A'_{12}}{A'_{21}}{-A'_{11}}
\end{aligned}
\end{equation}
with
\begin{equation}
\label{DNLS2.ABC}
\begin{aligned}
A'_{11}	&=	2\zeta^4 + \zeta^2 qr -\frac{1}{4}q^2 r^2 +\frac{i}{2}\left(q_x r- q r_x \right)\\
A'_{12}	&=	2i\zeta^3q -\zeta q_x \\
A'_{21}	&=	2i\zeta^3r +\zeta r_x  \ .
\end{aligned}
\end{equation}
\end{proposition}

\begin{proof}
A $2 \times 2$ matrix-valued function $G(x,t)$ defines a gauge transformation to a new Lax pair 
\begin{align}
\label{L'}
L'	&=	G L G^{-1} + G_x G^{-1} \\
\label{A'}
A' &=	G A G^{-1} + G_t G^{-1} { .}
\end{align}
Indeed,  if $\psi_x = L\psi$ and $\psi_t=A\psi$, then the function $\Psi=G \psi$ satisfies $\Psi_x = L'\Psi$ and $\Psi_t = A'\Psi$. 
We seek a gauge  transformation in  the   matrix form
\begin{equation}
\label{Omega}
G(x,t)=\twomat{e^{i\varphi}}{0}{0}{e^{-i\varphi}} .
\end{equation}
A simple computation shows that 
$$ L' = -i \zeta^2 \sigma + \zeta Q + P $$
where, setting
$$ q = e^{-2i\varphi} u, \quad r = e^{2i\varphi} v, $$ 
we have
$$ Q = \twomat{0}{q}{r}{0}, \quad P = \twomat{-i \varphi_x}{0}{0}{i\varphi_x}. $$
We wish to choose $\varphi$ so that\footnote{This condition insures that the RHP associated to the inverse scattering map  will be   properly normalized for large 
scattering parameter. }
$$P = i Q (\ad \sigma)^{-1} Q =\twomat{-\frac{i}{2}qr}{0}{0}{\frac{i}{2}qr}. $$
It follows that 
\begin{equation}
\label{phi.bis}
\varphi(x,t) = \frac{1}{2} \leftint qr \, dy = \frac{1}{2} \leftint uv \, dy. 
\end{equation}
We get $A'$ in the form  
$ A' = (-i)\twomat{A'_{11}}{A'_{12}}{A'_{21}}{-A'_{11}} $
with
\begin{align*}
A'_{11}	&=	2\zeta^4 + \zeta^2 qr - \varphi_t\\
A'_{12}	&=	2i\zeta^3q -\zeta u_x e^{2i\varphi}+ i\zeta q^2 r\\
A'_{21}	&=	2i\zeta^3r +\zeta v_x e^{-2i\varphi} + i\zeta qr^2 \ .
\end{align*}
We can express $u_x e^{-2i\varphi}$ and $v_x e^{2i\varphi} $ in terms of $q$ and $r$ by differentiating the identities $u=e^{2i\varphi}q$ and $v=e^{-2i\varphi}r$ to obtain
\begin{equation}\label{uvx}
u_x e^{-2i\varphi} 	=	q_x + i q^2 r, \quad 
v_x e^{2i\varphi}		=	r_x - iqr^2.
\end{equation}
We compute $\varphi_t$  using that $u$ and $v$ obey the equations \eqref{DNLS.uv} 
$$
\varphi_t = \left( \frac{1}{2} \leftint uv \, dy \right)_t
	=	\frac{i}{2}\left(u_x v - uv_x \right) + \frac{3}{4}(u^2 v^2) = -\frac{1}{4}q^2 r^2 +\frac{i}{2}\left(q_x r- q r_x \right)
$$
Finally, a short computation shows that the condition
$ L'_t-A'_x + [L',A'] = 0 $ gives the following equations (in the order $(11)$, $(12)$, $(21)$, $(22)$ of entries in the matrices):
\begin{align*}
-\frac{i}{2}\left(q_tr+qr_t\right)
	-\frac{i}{2}\left(r^2 qq_x+q^2rr_x\right)
	-\frac{1}{2}\left(rq_{xx}-qr_{xx}\right)
&=	0\\
q_t -iq_{xx}+q^2 r_x -\frac{i}{2}r^2q^3 &=0\\
r_t +ir_{xx}+r^2 q_x +\frac{i}{2}r^3 q^2 &=0\\
\frac{i}{2}\left(q_tr+qr_t\right)
	+\frac{i}{2}\left(r^2 qq_x+q^2rr_x\right)
	+\frac{1}{2}\left(rq_{xx}-qr_{xx}\right)
&=	0
\end{align*}
In particular, the $(12)$ and $(21)$ equations hold, the $(11)$ and $(22)$ equations are vacuous. It shows that \eqref{zcr2} give a zero-curvature representation of \eqref{DNLS2}, and that the transformation \eqref{q.gauge} indeed maps solutions of \eqref{DNLS1} to solutions of \eqref{DNLS2}.
\end{proof}

%
%

\section{Time-Evolution of Solutions to the RHP}
\label{app:time}

%
%

We prove that the solutions $M_\pm$ of the RHP \eqref{RHP.t} solve equation \eqref{RHP.Lax.t}, completing the proof of Proposition \ref{RHP.t}.   The computation is  similar to that presented  in the proof of Proposition \ref{prop:recon}  { except that now, we take into account the time evolution}.
We write 
$$g _\pm \doteq h_\pm$$ if 
$$g_+ - h_+ = C^+k, \,\, \, g_- -  h_- = C^-k \quad \text{ for the same function }k \in L^2(\Sigma).$$
Since the RHP \eqref{RHP.t} has a unique solution, we have:

\begin{lemma}
\label{lemma:RHP.t.vanish}
Suppose that $G_\pm \doteq 0$ and $G_+ = G_- v_{x,t}$. Then $G_+ = G_- = 0$.
\end{lemma}

We will differentiate the jump relation in \eqref{RHP.t}
and use a commutator formula to show that the function
$$
G_\pm = \frac{\dee M_\pm}{\dee t} - A(x,t,\zeta)M_\pm - M_\pm(2i\zeta^4 \sigma) 
$$
obeys $G_\pm \doteq 0$ and $G_+ = G_- v_{x,t}$, proving Proposition \ref{prop:RHP.t}. 
A computation similar to that leading to \eqref{jump.x} gives

\begin{equation}
\label{RHP.t.jump.t}
\frac{\dee M_+}{\dee t} + i\zeta^4 \ad \sigma (M_+)
	= \left( \frac{\dee M_-}{\dee t} + i\zeta^4 \ad \sigma M_- \right) v_{x,t}.
\end{equation}
 We need  to evaluate $i\zeta^4 \ad \sigma  (M_\pm)$ modulo terms $\doteq 0$. 
To this end, we  use the commutator  formula  \eqref{Cpm.com}
\begin{equation}
\label{RHP.t.com}
\zeta^m C^\pm \left[ f \right](\zeta) 
	\doteq		- \sum_{j=0}^{m-1} \zeta^{m-1-j} \frac{f_j}{2 \pi i}
\end{equation}
The function
\begin{align*}
f(x,t,\zeta) 	= 	\mu(x,t,\zeta) \left(w_{x,t}^+(\zeta) + w_{x,t}^-(\zeta) \right) 
				=	\Twomat{-\overline{s(\zetabar)}\mu_{12}(x,t,\zeta)}{s(\zeta) \mu_{11}(x,t,\zeta)}
									{-\overline{s(\zetabar)}\mu_{22}(x,t,\zeta)}{s(\zeta) \mu_{21}(x,t,\zeta)}
\end{align*}
and its moments
$$
f_j(x,t)		=	\int_\Sigma \zeta^j f(x,t,\zeta) \, d\zeta, \quad j=0,1,2,3.
$$
play a crucial role in the the computations since the solution of the RHP
has the large-$z$ asymptotic expansion
\begin{equation}
\label{RHP.t.M.z}
M(x,z) \sim { \One} - \frac{1}{z} \frac{f_0}{2\pi i} - \frac{1}{z^2} \frac{f_1}{2\pi i} + \ldots. 
\end{equation}
By symmetry, $f_0$ and $f_2$ vanish on the diagonal, while $f_1$ and $f_3$ vanish off the diagonal, so that in particular 
\begin{equation}
\label{RHP.t.f1f3}
\ad \sigma  (f_1) = \ad \sigma  (f_3) = 0.
\end{equation}
We recall the reconstruction formula
\begin{align}
\label{RHP.t.q}
\offdiagmat{q}{\qbar}	&=	-\frac{1}{2\pi} 
											\ad \sigma  (f_0) 	\\[5pt]
\nonumber
								&=	-\frac{1}{\pi} \int_\Sigma
										\offdiagmat{\mu_{11}(x,\zeta) s(\zeta) e^{-2is^2 x}}
														{\mu_{22}(x,\zeta) \overline{s(\zetabar)} e^{-2is^2 x}}
										\, d\zeta
\end{align}
which implies
\begin{equation}
\label{RHP.t.f0}
f_0	=	\pi \offdiagmat{-q}{\qbar}.
\end{equation}
The formula
\begin{equation}
\label{RHP.t.-1}
M_\pm = {  \One} + C^\pm f
\end{equation}
implies that
\begin{equation}
\label{RHP.t.-3}
\zeta(\One-M_\pm) \doteq \frac{1}{2i} \offdiagmat{-q}{\qbar} \doteq \frac{1}{2i} \offdiagmat{-q}{\qbar} M_\pm
\end{equation}
where the last step follows from the fact that $M_\pm \doteq { \One}$.
In the course of the computations,  we will need to evaluate  { the off-diagonal matrix $f_2$.}
It follows from \eqref{RHP.Lax.x} and the identity $f = M_+ - M_-$ that
$$ \frac{\dee f}{\dee x} = -i\zeta^2 \ad \sigma  (f) + \zeta Q(x) f + P(x) f.$$
Hence
\begin{align} \label{RHP.t.7}
\frac{d}{dx} f_0	&=	\int_\Sigma \left(   -i  { \zeta^2} \ad \sigma ( f ) + \zeta Q(x) f + P(x) f \right) \, d\zeta\cr
					&=	-i \ad \sigma  (f_2) + Q(x) f_1 + P(x) f_0 .
\end{align}

\begin{lemma} 
\label{lemma:dm.dt}
The identity
\begin{equation}
\label{RHP.t.0}
\frac{\dee M_\pm}{\dee t} \doteq 0.
\end{equation}
holds.
\end{lemma}

\begin{proof}
Recall that 
$ M_\pm = \mu( {\One}  \pm w_{x,t}^\pm) $
where $w_{x,t}$ is given by \eqref{RHP.wxt}. 
Assume that  $ { s   \in \calS(\Sigma) }$. We claim that $(\dee \mu/\dee t)(x,t,\dotarg) \in L^2(\Sigma)$. If so, then \eqref{RHP.t.0} follows by differentiating \eqref{RHP.t.-1}.

We prove that $(\dee \mu/\dee t)(x,t,\dotarg) \in L^2(\Sigma)$ using the Beals-Coifman integral equation satisfied by $\mu$
$$
\mu = { \One} + C^+(\mu w_{x,t}^-) + C^-(\mu w_{x,t}^+) = { \One} + \calC_w \mu.
$$
We obtain 
\begin{equation}
\label{RHP.t.BC}
\frac{\dee \mu}{\dee t}(x,t,\zeta) = g(x,t,\zeta) + \calC_w \left( \frac{\dee \mu}{\dee t} \right) 
\end{equation}
where
$$ 
g(x,t,\zeta)= 	C^-\left( \mu \frac{\dee w_{x,t}^+}{\dee t} \right) +
					C^+\left( \mu \frac{\dee w_{x,t}^-}{\dee t} \right).
$$
Since $\mu-{ \One } \in L^2(\Sigma)$ and $\dee w_{x,t}^\pm/\dee t \in L^\infty(\Sigma) \cap L^2(\Sigma)$, it follows that $g(x,t,\dotarg) \in L^2(\Sigma)$ for each $(x,t)$, { and eq. \eqref{RHP.t.BC} for $\dee \mu/ \dee t$} can be solved in $L^2(\Sigma)$. 
\end{proof}

From the { commutator formula \eqref{RHP.t.com} }
(applying it {  successively} for $j=4$ and $j=3$) and the identity \eqref{RHP.t.f1f3}, we have

\begin{align}
\label{RHP.t.1}
2i\zeta^4 \ad \sigma  (M_\pm)		
		&\doteq		
						 -	\zeta^3 \frac{1}{\pi} \ad \sigma  
									\left[  { f_0} \right]
						-	\zeta \frac{1}{\pi} \ad \sigma  
									\left[ { f_2} \right]\\[5pt]	
\nonumber	
		&\doteq		2\zeta^3 Q(x) M_\pm + 2\zeta^3 Q(x)(I-M_\pm) 
							-  \frac{\zeta}{\pi} \ad \sigma  \left[ {f_2}\right]\\[5pt]
\nonumber
		&\doteq		2\zeta^3 Q(x) M_\pm + \frac{\zeta^2}{\pi i }  Q(x) f_0 + \frac{ \zeta}{\pi i }Q(x)  f_1  
							+ \frac{Q(x)}{\pi i} f_2 
-   \frac{\zeta}{\pi} \ad \sigma  \left[f_2\right]
\end{align}
We wish to re-write the last four terms on the right-hand side of \eqref{RHP.t.1} as coefficients times $M_\pm$, 
modulo the equivalence relation. We will see that terms involving $f_1$ cancel so we will keep separate track of 
these. 
We have (using  again the identity matrix $ \One = M_\pm + (\One -M_\pm)$)
\begin{align*}
\frac{\zeta^2}{\pi i } Q(x)  f_0 
	&=		i \zeta^2 \diagmat{-|q|^2}{|q|^2} M_\pm
				 + i \diagmat{-|q|^2}{|q|^2}\left[ \zeta^2 ( { \One}-M_\pm) \right]\\[10pt]
\nonumber
\text{(by \eqref{RHP.t.com}, }m=2) \quad
	&\doteq	i \zeta^2 \diagmat{-|q|^2}{|q|^2} M_\pm(x)
	\\[5pt]
\nonumber
&\quad
				+ i \diagmat{-|q|^2}{|q|^2}\left[ \zeta \frac{f_0}{2\pi i} + \frac{f_1}{2\pi i} \right]\\[10pt]
\nonumber
	&\doteq	i \zeta^2 \diagmat{-|q|^2}{|q|^2}  M_\pm
				+\frac{ \zeta }{2}\offdiagmat{|q|^2 q}{|q|^2 \qbar} M_\pm\\[5pt]
\nonumber
	&\quad
				+  \frac{\zeta}{2} \offdiagmat{|q|^2 q}{|q|^2 \qbar}( { \One}-M_\pm) 
				+	\frac{1}{2\pi} \diagmat{-|q|^2}{|q|^2} f_1 \\[10pt]
\nonumber
\text{(by \eqref{RHP.t.com}, }m=1) \quad
	&\doteq		i \zeta^2 \diagmat{-|q|^2}{|q|^2}  M_\pm
				+ \frac{\zeta}{2}  \offdiagmat{|q|^2 q}{|q|^2 \qbar} M_\pm \\[5pt]
\nonumber
	&\quad	-	\frac{i}{4} \diagmat{|q|^4}{-|q|^4} M_\pm  
				+  \frac{1}{2\pi} \diagmat{-|q|^2}{|q|^2} f_1 .
\end{align*}
Collecting all the terms, we get
\begin{align}
\label{main}
&2i\zeta^4 \ad \sigma (M_\pm) -2\zeta^3Q(x)M_\pm\\[5pt]
\nonumber
&\qquad + 
		\diagmat{i\zeta^2|q|^2}{-i\zeta^2|q|^2}M_\pm
		- { \frac{1}{2} }
\offdiagmat{\zeta |q|^2q}{\zeta |q|^2 \overline{q}}M_\pm		
		\\[5pt]
\nonumber
\qquad & \doteq   { \frac{1}{4} } \diagmat{- i|q|^4}{i|q|^4}M_\pm   +   
{ \frac{1}{\pi} }Q(x)(\ad \sigma) ^{-1}Q(x) f_1  + { \frac{1}{\pi i} } Q(x) \zeta f_1   
\\[5pt]
\nonumber
&\qquad  +{ \frac{1}{\pi i} } Q(x) f_2  
-\dfrac{\zeta}{\pi}\ad (\sigma)\left[ f_2\right]
\end{align}
We  are now able to simplify the right hand side of (\ref{main}) using (\ref{RHP.t.7}):
{
\begin{align*}
&\frac{1}{\pi}Q(x)(\ad \sigma) ^{-1}Q(x) f_1 +    \frac{1}{\pi i}Q(x) f_2 
\\[10pt]
\nonumber
&\quad\doteq  -\frac{1}{2}  \diagmat{q\overline{q}_x}{\overline{q}q_x}M_\pm     + \frac{i}{4} \diagmat{|q|^4}{-|q|^4}M_\pm
\end{align*}
}
and
{
\begin{align*}
&\frac{1}{\pi i } Q(x)\zeta f_1 -\frac{\zeta}{\pi}\ad \sigma \left[ f_2\right]
=\offdiagmat{i\zeta q_x}{-i\zeta\overline{q}_x} -  \frac{1}{2}  \offdiagmat{\zeta|q|^2q}{ \zeta|q|^2\overline{q}  }\\[5pt]
\nonumber
&\quad\doteq \offdiagmat{i\zeta q_x}{-i\zeta\overline{q}_x}M_\pm+ \frac{1}{2}  \diagmat{q_x\overline{q}}{\overline{q}_x q}M_\pm\\[5pt]
\nonumber
&\qquad 
- \frac{1}{2}\ \offdiagmat{\zeta|q|^2q}{\zeta|q|^2\overline{q}}M_\pm
+   \frac{i}{4} \diagmat{ |q|^4}{-|q|^4}M_\pm.
\end{align*}
}
{ Eq.  (\ref{main}) now becomes  the equivalence relation: 
\begin{align}
\label{equivalence}
&2i\zeta^4 \ad \sigma (M_\pm) -2\zeta^3Q(x)M_\pm\\
\nonumber
&\quad	+\diagmat{i\zeta^2|q|^2}{-i\zeta^2|q|^2}M_\pm
			- \frac{i}{4}\diagmat{|q|^4}{-|q|^4}M_\pm\\[5pt]
\nonumber
&\quad 	-\offdiagmat{i\zeta q_x}{-i\zeta\overline{q}_x}M_\pm
			+\frac{1}{2}\diagmat{q\overline{q}_x  -q_x\overline{q}      }
			{\overline{q}q_x  -  \overline{q}_x q      }M_\pm
\doteq  \ 0 \ .
\end{align}
We combine (\ref{equivalence}) with Lemma \ref{lemma:dm.dt} and Lemma \ref{lemma:RHP.t.vanish} to
obtain eq. \eqref{RHP.Lax.t} and conclude the proof of  Proposition \ref{prop:RHP.t}.}

\section*{Acknowledgments} 
We are grateful to  P.\  Miller, D.\  Pelinovsky and Y.\ Shimabukoro for useful discussions. J.\ L.\ and P.\ P.\  thank the Fields Institute and the Department of Mathematics at the University of Toronto for hospitality during part of the time that this work was done. This work was partially supported by a grant from the Simons Foundation (359431 to Peter Perry).

\end{document}